\documentclass[11pt,a4paper]{article}
\usepackage[nottoc]{tocbibind}
\usepackage{graphicx}
\usepackage{enumerate}  
\usepackage{amsmath}
\usepackage{amsfonts}
\usepackage{psfrag}
\usepackage{color}
\usepackage{pgf,tikz}
\usepackage{mathrsfs}
\usetikzlibrary{arrows}
\usepackage{fancyhdr}
\usepackage{txfonts} 
\usepackage[colorlinks,pdfpagelabels,pdfstartview = FitH,bookmarksopen = true,bookmarksnumbered = true,linkcolor = blue,plainpages = false,hypertexnames = false,citecolor = red] {hyperref}

\newtheorem{theorem}{Theorem}
\def\loc{\operatorname{\rm{loc}}}

\newtheorem{definition}{Definition}
\newtheorem{remark}{Remark}

\newcommand{\Ruz}{R{\accent'27u}{\accent'24z}i{\accent'24c}ka}
\newtheorem{lemma}{Lemma}
\newtheorem{proposition}{Proposition}

\newcommand{\data}{\texttt{data}}

\newcommand{\divo}{\textnormal{div}}

\newcommand{\mint}{\mathop{\int\hskip -1,05em -\, \!\!\!}\nolimits}

\numberwithin{equation}{section}
\allowdisplaybreaks[4]

\newenvironment{proof}{\smallskip\noindent\emph{Proof.}\hspace{1pt}}%
{\hspace{-5pt}{\nobreak\quad\nobreak\hfill\nobreak$\square$\vspace{8pt}%
\par}\smallskip\goodbreak}

\def\en{\mathbb N}
\newcommand{\ratio}{\nu, L}
\def\er{\mathbb R}
\newcommand{\ern}{\mathbb{R}^n}
\newcommand{\tu}{\tilde u}
\newcommand{\tv}{\tilde v}
\newcommand{\erN}{\mathbb{R}^N}
\newcommand\eps\varepsilon
\def\eqn#1$$#2$${\begin{equation}\label#1#2\end{equation}}
\newcommand{\ernN}{\mathbb{R}^{N\times n}}
\newcommand{\ai}{a_{{\rm i}}}
\newcommand{\as}{a_{{\rm s}}}

\newcommand{\brx}{B_{r}}

\newcommand{\be}{\begin{equation}}
\newcommand{\ee}{\end{equation}}

\newcommand{\supp}{\operatorname{supp}}

\newcommand{\rr}{\varrho}

\newcommand{\dist}{\operatorname{dist}}
\newcommand{\const}{\operatorname{const}}
\newcommand{\snr}[1]{\lvert #1\rvert}

\newcommand{\nr}[1]{\lVert #1 \rVert}

\newcommand{\br}{B_{r}(x_{0})}

\newcommand{\RN}{\mathbb{R}^{N}}
\newcommand{\SN}{\mathbb{S}^{N-1}}
\newcommand{\N}{\mathbb{N}}

\newcommand{\vv}{\|v\|_{L^\infty(B_{r})}}
\def\name[#1, #2]{#1 #2}

\newcommand{\f}{\Phi}

\newcommand{\rif}[1]{(\ref{#1})}
\newcommand{\trif}[1] {\textnormal{\rif{#1}}}

\newcommand{\stackleq}[1]{\stackrel{\rif{#1}}{ \leq}}

\DeclareMathOperator*{\supess}{ess\, sup}
\DeclareMathOperator*{\infess}{ess\, inf}

\newcommand{\cd}[1]{\textcolor[rgb]{0.00,0.00,0.00}{#1}}
\newcommand{\cc}[1]{\textcolor[rgb]{0.00,0.00,0.00}{#1}}
\newcommand{\ccc}[1]{\textcolor[rgb]{0.00,0.00,0.00}{#1}}
\definecolor{ffqqqq}{rgb}{1.,0.,0.}
\definecolor{uuuuuu}{rgb}{0.26666666666666666,0.26666666666666666,0.26666666666666666}

\begin{document}

\title{Manifold constrained non-uniformly elliptic problems}

\author{
\textsc{Cristiana De Filippis\thanks{Mathematical Institute, University of Oxford, Andrew Wiles Building, Radcliffe Observatory Quarter, Woodstock Road, Oxford, OX26GG, Oxford, United Kingdom. E-mail:
      \texttt{Cristiana.DeFilippis@maths.ox.ac.uk}} \ and   Giuseppe Mingione \thanks{Dipartimento SMFI, Universit\`a di Parma, Viale delle Scienze 53/a, Campus, 43124 Parma, Italy
E-mail: \texttt{giuseppe.mingione@unipr.it}}}}

\date{\today
}

\maketitle

\thispagestyle{plain}

\vspace{-0.6cm}
\begin{abstract}
We consider the problem of minimizing variational integrals defined on \cc{nonlinear} Sobolev spaces of competitors taking values into the sphere. The main novelty is that the underlying energy features a non-uniformly elliptic integrand exhibiting different polynomial growth conditions and no homogeneity. We develop a few intrinsic methods aimed at proving partial regularity of minima and providing techniques for treating larger classes of similar constrained non-uniformly elliptic variational problems. In order to give estimates for the singular sets we use a  general family of Hausdorff type measures following the local geometry of the integrand. A suitable comparison is provided with respect to the naturally associated capacities.

\end{abstract}

       \vspace{5pt}                           %
\vspace{0.2cm}

\setlength{\voffset}{-0in} \setlength{\textheight}{0.9\textheight}

\setcounter{page}{1} \setcounter{equation}{0}
\tableofcontents

\section{Introduction}
In this paper we want to treat, from the regularity theory viewpoint, a special but yet significant class of non-uniformly variational problems characterized by the fact that minimizers and competitors take their values into the sphere. At the same time, we want to introduce a few intrinsic methods and viewpoints that should be useful in order to prove regularity theorems for more general classes of non-uniformly elliptic equations and functionals with geometric constraints. For this reason we shall combine and present both old techniques from \ccc{different perspectives} and new ones. Specifically, we shall consider a class of variational integrals of the type
\eqn{genF}
$$
W^{1,1}(\Omega,\erN) \ni w   \mapsto\mathcal F(w, \Omega):= \int_{\Omega} F(x, w, Dw) \, dx\;,
$$
where the main model is provided by the so-called double phase functional
\eqn{modello} 
$$
\left\{
\begin{array}{c}
\displaystyle W^{1,1}(\Omega,\erN) \ni w   \mapsto\mathcal P(w, \Omega):= \int_{\Omega} (|Dw|^p+a(x)|Dw|^q) \, dx
\\[10 pt] 
\quad 1 <p < q<N\,, \quad 0 \leq a(\cdot) \in C^{0, \alpha}(\Omega) \;. 
\end{array}
\right.
$$
Here, as in the rest of the paper, $\Omega \subset \er^n$ denotes (unless otherwise specified) a bounded open domain, $n \geq 2$, and, again unless otherwise \cc{stated}, we consider $N> 1$ and $a(\cdot)$ satisfies the condition in $\rif{modello}_{3}$ for some $\alpha \in (0,1]$. When $a(\cdot)\equiv 0$, the integral in \rif{modello} reduces to the familiar $p$-Dirichlet integral 
\eqn{pLap}
$$
 w   \mapsto  \int_{\Omega} |Dw|^p  \, dx\;,
$$
whose Euler-Lagrange equation is given by the $p$-Laplacean system
$
-\divo\, (|Du|^{p-2}Du)=0. 
$
The regularity theory for minimizers of the functional in \rif{pLap} has been treated at length starting from the seminal papers of Uraltseva \cite{Ur} and Uhlenbeck \cite{Uh}, in the scalar and vectorial case, respectively. For results concerning more general functionals as in \rif{genF} with $p$-growth, that is, modelled on the one in \rif{modello} and therefore satisfying
\eqn{cond-p}
$$
|z|^p \lesssim F(x, w, z)\lesssim |z|^p+1\;,
$$ see for instance \cite{KMbull, KMjems, Manth} and related references. Under suitable assumptions, the final outcome is that minima are locally of class $C^{1, \beta}$, for some $\beta \in (0,1)$, on a subset of full $n$-dimensional Lebesgue measure. The regularity theory in the case when both minimizers and competitors take values into a manifold $\mathcal M \subset \er^N$ poses additional difficulties. In particular, the case $\mathcal M = \SN$ is the $(N-1)$-dimensional sphere in $\erN$ has been treated extensively. The theory \cc{started with} the fundamental papers of Eells \& Sampson \cite{ES} and Schoen \& Uhlenbeck \cite{SU, SU2}\cc{,} \ccc{analyzing} harmonic maps, i.e., constrained minimizers of the functional in \rif{pLap} for $p=2$. The extension of  such basic results to the case $p\not =2$ has been done in the by now classical papers of Fuchs \cite{fuchs1, fuchs3}, Hardt \& Lin \cite{HL} and Luckhaus \cite{luck}. Moreover, several results have been extended to more general functionals with $p$-growth, that is, functionals as in \rif{genF} with $F(\cdot)$ satisfying \rif{cond-p}; see for instance \cite{HKL}. On the other hand, we notice that energies of the type in \rif{modello} do not satisfy conditions as in \rif{cond-p}, but rather, the more general and flexible ones
\eqn{cond-pq}
$$
|z|^p \lesssim F(x, w, z)\lesssim |z|^q+1\;,\quad 1 <p < q\;.
$$ 
These are known in the literature as $(p,q)$-growth conditions or non-standard growth conditions. They have pioneered by Uraltseva \& Urdaletova \cite{UU} and Zhikov \cite{Z1, Z2, Z3} in the context of Homogenization (see also the recent paper \cite{DGloria}). In the setting of the Calculus of Variations they have been systematically studied by Marcellini \cite{M1, M2}. We refer to \cite{Dark} for a reasonable survey \cd{on} the subject. Growth conditions of the type in \rif{cond-pq} \cc{often occur} when considering variational models for physical phenomena. For instance, in the setting of Homogenization, a model as the double phase functional can be used to describe a composite of two materials with hardening exponents $p$ and $q$ respectively, whose geometry is dictated by the zero set $\{a(x)=0\}$ of the coefficient $a(\cdot)$. Obviously, both in the case $a(\cdot)\equiv 0$ and in the one when $\inf\, a(\cdot)>0$, we have a functional with 
standard polynomial growth of the type in \rif{cond-p} (with $p$ replaced by $q$ in the second case). In the remaining \cc{one,} the nature of ellipticity of the functional $\mathcal P$ switches between the $p$ and $q$ rates accordingly to the value of $a(\cdot)$. For this reason models as those in \rif{modello} are particularly useful to describe strongly anisotropic media. We refer to the papers \cite{BCM1, Radu1, CM1,DOh, DP, PRR, Z1} for more results, different directions and related topics. Another, softer instance of functional with non-standard growth used to describe anisotropic models \cite{AM2, Z1, Z2, Z3} is the variable exponent one
\eqn{variable}
$$
 w   \mapsto  \int_{\Omega} |Dw|^{p(x)}  \, dx\;,
$$
that has attracted a lot of attention in the last years \cite{Radu2}; a match between the two cases has been recently proposed in \ccc{\cite{me,the4,ELM,OK1,Radu2,RT1,RT2}}.   
  The growth conditions in \rif{cond-pq} are typically linked to the non-uniform ellipticity of the Euler-Lagrange equations associated to the functionals in question. In case of \rif{modello}, such equation is
\eqn{eulerpq}
$$
-\divo\, A(x,Du)=0 \qquad \mbox{with} \quad A(x,z)=|z|^{p-2}z+(q/p)a(x)|z|^{q-2}z\;, 
$$
and therefore, the ellipticity ratio \cc{$\mathcal R(z, B)$} on any ball $B\subset \Omega$ touching the transition set $\{a(x)=0\}$, which is defined by 
$$
\mathcal R(z, B):= \frac{\mbox{$\sup_{x\in B}$ of the highest eigenvalue of}\ \partial_{z} A(x,z)}{\mbox{$\inf_{x\in B}$ of the lowest eigenvalue of}\  \partial_{z} A(x,z)} \approx  \ccc{1+\|a\|_{L^{\infty}(B)}|z|^{q-p}}\;,
$$
becomes unbounded as $|z|\to \infty$. This means that the equation in \rif{eulerpq} is non-uniformly elliptic. 
More specifically, the asymptotics of the ratio $\mathcal R(z, B)$ exhibit a delicate interplay between the size of $|z|^{q-p}$ and the one of the coefficient $a(\cdot)$ which is crucially close to the zero set $\{a(x)=0\}$. As a matter of fact, the rate $a(\cdot)$ approaches to zero rebalances the rate of potential blow-up. This is displayed in the sharp condition
$
q \leq p+\alpha
$, which \ccc{in \cite{BCM3,CM2,CM3} is found} to be necessary and sufficient for unconstrained, {\em bounded}, scalar minimizers of the model functional \rif{modello} to be regular; otherwise discontinuous minimizers of $\mathcal P$ may exist, \ccc{\cite{ELM, FMM}}.  Let us remark that both the model functional in \rif{modello} and the one in \rif{variable} fall in the realm of non-autonomous functionals defined in Musielak-Orlicz spaces. These are functionals of the type
\eqn{PPhi}
$$
 w   \mapsto  \int_{\Omega} \Phi(x, |Dw|)  \, dx\;, 
$$
where, \cc{$\Phi \colon \Omega \times [0, \infty)\to [0,\infty)$} is a Caratheodory function such that for each choice of $x \in \Omega$, the partial map $t \mapsto \Phi(x, t)$ is a Young function and thereby generates an Orlicz space that changes with $x$. Such functionals are naturally defined on Musielak-Orlicz-Sobolev spaces/classes $W^{1, \Phi}$, i.e., 
\eqn{Orlicz-Sob}
$$
\cc{W^{1, \Phi}(\Omega, \RN)= \left\{ \ f \in W^{1,1}(\Omega,\RN) \ \colon \ \Phi(\cdot, |f|)+ 
\Phi (\cdot, |Df|)\in L^1(\Omega) \ \right\}\;}.
$$
For such spaces we refer to the recent interesting monograph \cite{HH}.  
A main problem here is the one of finding general conditions ensuring the regularity of minimizers. This appears as a non-trivial and challenging issue. The main idea is that the regularity of minima of \rif{PPhi} is governed by a delicate interplay between the regularity of the function $x \mapsto \Phi(x,\cdot)$ and the growth conditions of $t \mapsto \Phi(\cdot,t)$. For instance, in the double phase case, relations as 
$q \leq p+ \alpha$ or $q/p \leq 1+\alpha/n$ define sharp conditions for regularity \cite{BCM3, CM1, CM2}. In the variable exponent case the log-modulus of continuity of $p(\cdot)$ is another instance of such conditions. See \cite{BCM2} for a picture concerning the similarities between \rif{modello} and \rif{variable} as particular cases of the one in \rif{PPhi}. More general conditions unifying those for \rif{modello} and \rif{variable} have been formulated in the interesting paper \cite{HHT} and lead to a full De Giorgi-Nash-Moser theory.  A general approach to the gradient regularity has been devised and suggested in \cite{BCM3}. The regularity problem in the case of constrained minimizers for functionals with non-standard growth conditions has recently received some attention, see the higher integrability \ccc{result} recently obtained in \cite{cristiana} and the \ccc{singular set estimates} proved for the variable exponent case in \cite{me}. The main difficulties essentially rely in the lack of a certain number of properties, that are typically linked to uniform ellipticity and that are essential in order to treat constrained minimizers. In this paper we undertake this issue in the model case of double phase energies, i.e., functionals of the type in \rif{genF} controlled by the one in \rif{modello} in the sense of \rif{assF} below. Moreover, we consider the case when the manifold is the $(N-1)$-dimensional sphere $\SN$ in $\er^N$, that already incorporates several of the new difficulties. Our aim here is also to propose an intrinsic approach which departs from the usual estimates, \cc{and it is designed for treating} the quantity $ \Phi(\cdot, |Dw|)$ which in this case is $|Dw|^p+a(\cdot)|Dw|^q$, as a sort of replacement of $|Dw|^p$. We shall therefore formulate  and use a certain number of tools (harmonic approximation lemmas, a priori estimates and so on) in terms of the quantity $\Phi(\cdot, |Dw|)$. Accordingly to this viewpoint, in order to characterize the singular sets, we shall use an intrinsic Hausdorff type measure aimed at catching the local geometry of the integrand $\Phi(\cdot, t)$. Such measures give back the standard Hausdorff measure in the case $\Phi(\cdot, t)= t^{p}$ as well as other examples of measures available in the literature. We then compare these measures to the natural capacities generated by functionals of the type in \rif{PPhi} and relate the corresponding outcomes to the size of the singular sets, that are indeed found to have zero capacity.

\subsection{Partial regularity} It is convenient to introduce some notation (see also Section \ref{notationsec} below). We shall denote 
\eqn{HK}
$$
H(x, z):= |z|^p+a(x)|z|^q\,, \qquad \mbox{for all} \ \ z \in \ernN\,,  \ \ x \in \Omega\;,
$$
where $a(\cdot)$ is as in \rif{modello} and recall that in the following it will always be $n \geq 2$ and $N>1$ (this last one, unless otherwise stated). Moreover, with $B\Subset\Omega$ being a ball, we introduce the auxiliary Young functions
\begin{flalign}\label{H+-}
\cc{\begin{cases}
\ H^{-}_{B}(z):=\snr{z}^{p}+a_{i}(B)\snr{z}^{q},\quad &\mbox{where} \ \ a_{i}(B):=\inf_{x\in B}a(x);\\
\ H^{+}_{B}(z):=\snr{z}^{p}+a_{s}(B)\snr{z}^{q},\quad &\mbox{where} \ \ a_{s}(B):=\sup_{x\in B}a(x)\;.
\end{cases}}
\end{flalign}
With abuse of notation, we shall keep on denoting 
$H(x, t) = t^p+a(x)t^q$ for $t\geq 0$ (and the like for the functions in \rif{H+-}), that is when in \rif{HK} $z$ is a non-negative number. In the rest of the paper, with $B\subset \er^n$ being a ball, we shall denote by $r(B)$ its radius. Following \cite{BCM3}, we \cc{then} consider variational integrals of the type in \rif{genF}, 
where $F\colon \Omega\times \RN \times \ernN\to \er$ is a Carath\'eodory integrand, such that $z \mapsto F(\cdot, z)$ is $C^1(\ernN)\cap C^2(\ernN\setminus\{0\})$ and satisfies the following assumptions:
\begin{flalign}\label{assF}
\begin{cases}
\nu H(x,z) \leq F(x,v,z)\le LH(x,z)\\
\ \snr{\partial F(x,v,z)}\snr{z}+\snr{\partial^{2}F(x,v,z)}\snr{z}^{2}\le LH(x,z)\\
\ \nu(\snr{z}^{p-2}+a(x)\snr{z}^{q-2})\snr{\xi}^{2}\le \langle \partial^{2}F(x,v,z)\xi,\xi\rangle\\
\ \snr{\partial F(x_{1},v,z)-\partial F(x_{2},v,z)}\snr{z}\le L\omega(\snr{x_{1}-x_{2}})[H(x_{1},z)+H(x_{2},z)]+L\snr{a(x_{1})-a(x_{2})}|z|^{q}\\
\ \snr{F(x,v_{1},z)-F(x,v_{2},z)}\le L\omega(\snr{v_{1}-v_{2}})H(x,z)
\;.
\end{cases}
\end{flalign}
These are assumed to hold whenever $x, x_{1}, x_{2}\in \Omega$, $v, v_1, v_2 \in \erN$, $z, z_1, z_2 \in \ernN\setminus \{0\}$, $\xi \in \mathbb{R}^{N\times n}$, where $0<\nu\le 1 \leq L$ and $\vartheta \in (0,1)$ are fixed constants and, for every non-negative number $t$, 
\eqn{omegabeta}
$$
\omega(t):=\min\{t^{\beta},1\}\;, \qquad \beta \in (0,1]\;, 
$$
is defined as the standard concave $\beta$-H\"older modulus of continuity. Note that, here as in the following, by \cc{"$\partial$"} we always mean the partial derivative with respect to the gradient variable $z$. We finally consider the necessary structure assumption to deal with the vectorial case, that is, we assume that for every choice of $(x, v)\in \Omega \times \er^N$ there exists a function \ccc{$\tilde F_{x,v}(\cdot)\equiv \tilde F(x, v, \cdot)\colon [0,\infty) \to [0, \infty)$} of class $C^1[0, \infty)\cap C^2(0, \infty)$, such that 
\begin{flalign}\label{assF2}
\begin{cases}
 \ F(x,v,z)= \tilde F(x,v,|z|)\ \mbox{holds for every $z \in \ernN$} \ \mathrm{with} \ t \mapsto \tilde{F}(x,y,t) \ \mbox{being \cc{non-decreasing,}}\\[7 pt]
\displaystyle
 \ \left|\tilde F''(x,v,t+s)-\tilde F''(x,v,t)\right|\leq \frac{LH(x,t)}{t^2}\left(\frac{|s|}{t}\right)^{\beta_1}\;.
\end{cases}
\end{flalign}
The inequality in the last line is assumed to hold whenever $s, t \in \er$ are such that $t>0$ and $2|s|< t$, and with a fixed constant \ccc{$\beta_1 \in (0,1]$} (which is independent of the considered $(x, v)$). 
As for the exponents $p,q$, we assume 
\eqn{mainbound}
$$
p< q < p+\alpha\,, \qquad q< N\;. 
$$
We remark that the inequality $q < p+\alpha$ in the last display, apart from the missing equality case, is a sharp condition for regularity, as shown in \cite{ELM, FMM}. The second inequality $q< N$ relates the growth conditions 
of the problem and the topological properties of the target manifold, which is in this case $\SN$. 
This is necessary in order to use certain projection operators (see Lemma \ref{L4} below). Assumptions of this kind 
are considered by Hardt \& Lin \cite{HL, HKL} in the convex case and by Hopper \cite{hopper} in the quasiconvex one. The related definition of local minimizer we are going to consider is the following:
\begin{definition}\label{defimain} A function $u \in W^{1,1}_{\loc}(\Omega,\SN)$ is a local \cd{minimizer} of the functional $\mathcal F$ defined in \trif{genF} under assumptions \trif{assF}$_1$, if and only if \cd{$H(\cdot,Du) \in L^{1}_{\loc}(\Omega)$} and the minimality condition 
$\mathcal F(u,{\rm supp} \, (u-v)) \leq \mathcal F(v,{\rm supp} \, (u-v))$ is satisfied whenever $v \in W^{1,1}_{\rm{\rm{loc}}}(\Omega,\SN)$ 
is such that ${\rm supp} \, (u-v) \subset \Omega$.
\end{definition}
By definition a local minimizer belongs to \cd{$W^{1,p}_{\loc}(\Omega,\SN)$}; in the rest of the paper we shall appeal such local minimizers sometimes as constrained local minimizers to emphasize that the presence of the constraint $|u|=1$.  
We notice when $a(\cdot)\equiv 0$ assumptions \rif{assF} reduce to the standard ones considered for functionals with $p$-growth when considering partial regularity problems (see for instance \cite{KM1, KM2, MMS2, Dark} and related references). In particular, assumptions \cc{\eqref{assF}-\eqref{assF2}} are devised to cover functionals of the type 
$$
w
\mapsto \int_{\Omega} \left[F_1(x, w, Dw) + a(x)F_{2}(x, w, Dw)\right] \, dx\;,
$$
where $F_1(\cdot)$ and $F_2(\cdot)$ \cc{have $p$- and $q$-growth}, respectively, accordingly to the standard assumptions described for instance in \cite{KM1}. Another functional covered by \cc{our set of assumptions} is 
$$
W^{1,1}(\Omega) \ni w   \mapsto \int_{\Omega} b(x,w)H(x, Dw) \, dx\;, $$
where $0< \nu_1 \leq b(x, v)\leq L_1$, for some constants $\nu_1, L_1$. \cd{Here, $b(\cdot)$ is a H\"older-continuous function}. For later convenience, we shall denote 
\eqn{dataref} 
$$\ccc{\data \equiv 	\data \left(n,N,\nu,L,p,q,\alpha,[a]_{0,\alpha}\right)}
%\data \left(n,N,p,q,\ratio,\|a\|_{L^\infty}, \alpha,[a]_{0,\alpha}, \beta,\|H(\cdot,Du)\|_{L^1(\Omega)}\right)
\;,$$
as the set of basic parameters intervening in the problem. Our first main result is the following:
\begin{theorem}\label{main1} Let $u \in W^{1,1}_{\loc}(\Omega,\SN)$ be a local minimizer of the functional $\mathcal F$ in \trif{genF} under the assumptions \trif{assF}-\trif{mainbound}. There exists $\delta_g \equiv \delta_g(\data)>0$ such that
\eqn{maggioreiniziale}
$$
H(\cdot, Du) \in L^{1+\delta_g}_{\loc}(\Omega)\;.
$$ Moreover, there exist \ccc{$\beta_0\equiv \beta_0(\data,\beta,\beta_{1})>0$} and an open subset $\Omega_u \subset \Omega$, called the regular set, such that
\eqn{mainassertion} 
$$Du \in C^{0, \beta_0}_{\loc}(\Omega_u,\ernN) \qquad \mbox{and}\qquad |\Omega\setminus\Omega_u|=0\;.$$ 
In the case $p(1+\delta_g)>n$ we have $\Omega = \Omega_u$. When $p(1+\delta_g)\leq n$, there exists a number \ccc{$\eps\equiv \eps(\data,\beta)$}, such that a point 
$x_0\in \Omega$ belongs to $\Omega_u$ iff
\eqn{careps}
$$
\left[H^{-}_{\br}\left(\frac{\eps}{r}\right)\right]^{-1}\mint_{\br}H(x, Du) \, dx< 1
$$
holds for some ball $B_{r}(x_0)\Subset \Omega$ with $r \leq 1$. Finally, as for the so-called singular set $\Sigma_u:=\Omega\setminus \Omega_u$, it follows that
\eqn{careps2}
$$
\Sigma_u = \left\{x_0 \in \Omega \, \colon \, \limsup_{\rr \to 0}\, \left[H^{-}_{B_{\varrho}(x_0)}\left(\frac{1}{\rr}\right)\right]^{-1}\mint_{B_{\varrho}(x_0)}H(x, Du) \, dx> 0\right\}\;. 
$$
\end{theorem}
The $\eps$-regularity condition \rif{careps} differs from the usual ones given in the case of functionals with $p$-growth as it gives an intrinsic quantified version of the \cd{amount of} energy needed for regularity; see also the interesting paper \cite{DSV} for the case of autonomous functionals. The shape of \rif{careps} suggests an intrinsic path to estimate the size of the so-called singular set $\Omega\setminus\Omega_u$. Indeed, this can be done via a general definition of certain Hausdorff type measures that can be useful in general contexts too; for this we refer to the next section. It is worth remarking that the \cd{results} of Theorem \ref{main1} \cd{continue} to hold in the case of unconstrained {\em bounded} minimizers and it is new in the vectorial case (it extends the scalar one in \cite{BCM3}). In the unconstrained case the condition $q<N$ in \rif{cond-pq} can be dropped (see Remark \ref{onlyre} below).  
\begin{remark}\label{Intermediate}
\emph{It is still possible to get a partial regularity result by weakening the assumptions on the function $\tilde F_{x, v}(t)\equiv F(x, v, t)$ considered in \eqref{assF2}. Specifically, we can drop \eqref{assF2}$_2$, thereby \cc{replacing \eqref{mainassertion}} with the weaker outcome
\eqn{mainassertion-senza} 
$$u \in C^{0, \beta_2}_{\loc}(\Omega_u,\ernN) \qquad \mbox{and}\qquad |\Omega\setminus\Omega_u|=0\;,$$ 
for every $\beta_2 <1$. }
\end{remark}

\subsection{Weighted Hausdorff measures, intrinsic capacities and singular sets}\label{weight}
Here we shall be slightly more general than what is needed in the present setting, as we wish to settle down a general approach valid also for other contexts. We shall produce a family of Hausdorff type measures that are naturally linked to general functionals of the type in \rif{PPhi}. In the following \cd{we consider} a Carath\'eodory function $\Phi \colon \Omega \times [0, \infty) \to [0, \infty)$, i.e., such that $x \mapsto \Phi (x, t)$ is measurable for every $t\geq 0$ and 
$t \mapsto \Phi (x, t)$ is continuous and non-decreasing for almost every $x \in \Omega$. Here, $\Omega\subset \ern$ denotes an open subset. Moreover, we assume that $\Phi(x, 0)=0$ and that $\lim_{t\to \infty}\, \Phi(x, t)=\infty$ for every $x\in \Omega$. \cd{We also} assume that
\eqn{misurabile}
$$
\Phi(x, t) \lesssim m(x) t^n\;, \qquad \cc{\mbox{for all}} \ t\geq 1\;, \ \ \mbox{a.e.} \ x \in \Omega\;, \quad \mbox{where}\ \ccc{0\le m(\cdot) \in L^{1}_{\mathrm{loc}}(\Omega)}
$$
\cc{and that} 
\begin{equation}\label{0777-1}
\mbox{there exists $\beta_3 \in (0,1)$ such that $\Phi(x, \beta_3)\leq 1$ and $\Phi\left(x, 1/\beta_3\right)\geq1$ for every $x\in \Omega$}\;,
\end{equation}
\begin{equation}\label{0777-22}
\frac{\Phi(x, s)}{s} \lesssim \frac{\Phi(x, t)}{t} \qquad \mbox{whenever} \ \ \ 
0 < s \leq t\;,\cd{ \ \ \mbox{for all} \ \ x\in \Omega}\;.
\end{equation}
These assumptions, also considered in \cite{BHH}, are trivially verified by all the relevant model examples motivating us; see Remark \ref{lista} below.
To proceed, for any $n$-dimensional open ball $B\subset \Omega$ (there is no difference in the following in taking closed balls in this respect) of radius $r(B)\in (0,\infty)$, we define
\eqn{openballs}
$$
h_{\Phi}(B)=\int_{B}\Phi\left(x, 1/\mathrm{r}(B)\right)\, dx\;.
$$
Notice that this function is always finite and that this is guaranteed by \rif{misurabile}\cd{. It results:} 
\begin{flalign*}
\cd{h_{\Phi}(B)\lesssim \mint_{B}m(x)\ dx\;.}
\end{flalign*}
We then use the standard Carath\'eodory's construction to obtain an outer measure. For this, let $E\subset \Omega$ be any subset. We define the weighted $\kappa$-approximating Hausdorff measure of $E$, $\mathcal{H}_{\Phi,\kappa}(E)$ with $0< \kappa \leq 1$, by
\begin{flalign}\label{077}
&\mathcal{H}_{\Phi,\kappa}(E)=\inf_{\mathcal{C}_{E}^{\kappa}}\sum_{j}h_{\Phi}(B_{j})\;,\\
&\mathcal{C}_{E}^{\kappa}=\left \{\ \{B_{j}\}_{j\in \N}\ \mathrm{is \ a \ countable \ collection \ of \ balls \ B_j \subset \Omega \ covering \ }E\, \ \mathrm{such \ that \ }\mathrm{r}(B_{j})\le \kappa \  \right\}\nonumber. 
\end{flalign}
As $0<\kappa_{1}<\kappa_{2}<\infty$ implies $\mathcal{C}_{E}^{\kappa_{1}}\subset \mathcal{C}_{E}^{\kappa_{2}}$, we have that $\mathcal{H}_{\Phi, \kappa_1}(E)\ge \mathcal{H}_{\Phi,\kappa_{2}}(E)$ and there exists the limit
\begin{flalign}\label{078}
\mathcal{H}_{\Phi}(E):=\lim_{\kappa \to 0}\mathcal{H}_{\Phi,\kappa}(E)=\sup_{\kappa>0}\mathcal{H}_{\Phi,\kappa}(E)\;.
\end{flalign}
When considering functionals of the type in \rif{PPhi}, it is convenient to localize the $x$-dependence and locally compare the starting integrand $\Phi(\cdot)$ with \cc{similar maps} that are independent of $x$. This means that, with a ball $B\subset \Omega$ being fixed, we consider the functions $t \mapsto \infess_{x\in B}\, \Phi(x, t)$ and 
$t \mapsto \supess_{x\in B}\, \Phi(x, t)$, and define 
\eqn{defipre}
$$
h_{\Phi}^+(B)=|B|\supess_{x \in B}\, \Phi\left(x, 1/r(B)\right)\quad \mbox{and}\qquad 
h_{\Phi}^-(B)=|B|\infess_{x \in B}\, \Phi\left(x, 1/r(B)\right)
$$
so that
$
h_{\Phi}^-(B)\leq h_{\Phi}(B) \leq h_{\Phi}^+(B).
$
Accordingly, keeping \rif{077}-\rif{078}, we finally set
\eqn{comingfrom}
$$
\mathcal{H}_{\Phi,\kappa}^{\pm}(E)=\inf_{\mathcal{C}_{E}^{\kappa}}\sum_{j}h_{\Phi}^{\pm}(B_{j})\quad\mbox{and}\quad 
\mathcal{H}_{\Phi}^{\pm}(E)=\lim_{\kappa \to 0}\mathcal{H}_{\Phi,\kappa}^{\pm}(E)\;.
$$
The above definitions obviously imply that $\mathcal{H}_{\Phi, \kappa}^{-}(E)\leq \mathcal{H}_{\Phi, \kappa}(E) \leq \mathcal{H}_{\Phi, \kappa}^{+}(E)$ holds for every $\kappa \in (0,1]$ and therefore, upon letting $\kappa \to 0$, it follows that
\eqn{connetti}
$$
\mathcal{H}_{\Phi}^{-}(E)\leq \mathcal{H}_{\Phi}(E) \leq \mathcal{H}_{\Phi}^{+}(E)\;.
$$
\begin{remark}\label{lista}
\emph{
Definition \rif{078} is aimed at catching and unifying several instances of similar objects. Furthermore, let us notice that
\begin{itemize}
\item In the case $\Phi(x, t)\equiv t^p$ for $p\leq n$, then $\mathcal{H}_{\Phi}$ is equivalent (up to \cd{constants}) to the usual $(n-p)$-dimensional spherical Hausdorff measure. 
\item In the case $\Phi(x, t)\equiv t^{p(x)}$ for $p(x)\leq n$ being a continuous function defined on an open subset $\Omega$, then $\mathcal{H}_{\Phi}$ falls in the class of the variable exponent Hausdorff measures \cc{studied} in \cite{niemi, turesson}.
\item In the case $\Phi(x, t)\equiv w(x) t^{p}$ for $p\leq n$ and $w(\cdot)$ being a \cd{non-negative and measurable} function, $\mathcal{H}_{\Phi}$ is equivalent to the weighted Hausdorff measures introduced \cc{in} \cite{niemi, turesson}, with particular emphasis on the situations when $w(\cdot)$ is a Muckenhoupt weight. 
\item The case we are mostly interested in is when $\Phi(x, t)= [H(x,t)]^{1+\delta}\equiv [t^p+a(x)t^q]^{1+\delta}$ for some $\delta \geq 0$, with $H(\cdot)$ as in  \trif{HK} and under the condition that $q(1+\delta)\leq n$. In this case we shall use the notation 
$
\mathcal{H}_{\Phi}\equiv \mathcal{H}_{H^{1+\delta}}$ and $
\mathcal{H}_{\Phi}^{\pm}\equiv \mathcal{H}^{\pm}_{H^{1+\delta}}. $
\end{itemize}}
\end{remark}
\begin{remark}
\emph{By standard arguments, i.e., those of the type needed in the case of the usual Hausdorff measures, the set function $\mathcal{H}_{\Phi}$ turns out to be a \cd{Borel-regular} measure (here we adopt the standard terminology from \cite{EG}). We notice that the definition of $\mathcal{H}_{\Phi,\kappa}$ is invariant when using open or closed balls in \rif{openballs}. As for the set functions $\mathcal{H}_{\Phi}^{\pm}$ in \rif{comingfrom}, these turn out to be Borel regular measures too. We mention an alternative way to describe measures as $\mathcal{H}_{\Phi}^{\pm}$. This occurs upon replacing \eqref{defipre} by 
\eqn{closedballs}
$$
h_{\Phi}^+(B)=|B|\sup_{x \in B}\, \Phi\left(x, 1/r(B)\right)\quad \mbox{and}\qquad 
h_{\Phi}^-(B)=|B|\inf_{x \in B}\, \Phi\left(x, 1/r(B)\right)\;.
$$
In this case the corresponding set functions $\mathcal{H}_{\Phi}^{\pm}$ are again \cd{Borel} measures and are Borel regular too if $\Phi(\cdot)$ is continuous. Alternatively, one can use in the definition \rif{closedballs} closed balls instead of open ones, thereby always getting automatically a Borel regular measure.}
\end{remark}
In order to place the above measures in the setting of regularity of minimizers and to connect the three measures appearing in \rif{connetti}, we next consider the following assumption:
\eqn{controllo}
$$
\supess_{x \in B}\, \Phi\left(x, \beta_4 t\right)\leq c_d\infess_{x \in B}\, \Phi\left(x, t\right)$$ to hold whenever  $1 \leq t \leq 1/r(B)$ \   for all balls $B\subset \Omega$ with $r(B)\leq 1$ and for some constants $\beta_4 \in (0,1], c_d\geq 1$\;.
This assumption is known to be crucial to prove the local H\"older continuity of {\em bounded} minimizers of functionals of the type in \trif{PPhi} and in certain Harmonic Analysis questions related to Musielak-Orlicz spaces; see \cite{HHT, Hasto}. In this respect, assumption \rif{controllo} is sharp by the examples in \cite{ELM, FMM}. When applied to the choice $\Phi (x, t)= t^p +a(x)t^q$ and \ccc{$a(\cdot)\in C^{0,\alpha}(\Omega)$}, \rif{controllo} amounts to require that $q\leq p+\alpha$ as first considered in \cite{CM2}; \cc{see Proposition \ref{triviaros1}} and again \rif{condsi} below. An immediate consequence of \rif{controllo} is the following fact, whose proof is reported in \cc{Section \ref{cap-proofs}}. 
\begin{proposition}\label{triviaros1} Assume that \trif{controllo} holds. Then, for any subset $E\subset \Omega$ it follows that 
\eqn{othercontrol}
$$
\mathcal{H}_{\Phi}^{+}(E) \leq \frac{c_d}{\beta_4^n} \mathcal{H}_{\Phi}^{-}(E)\;.
$$
As a consequence, if $a(\cdot)\in C^{0, \alpha}(\Omega)$, $q\leq p+\alpha$ and $\delta\ge 0$, then there exists a constant $c\equiv c ([a]_{0, \alpha},\delta)\geq 1$ such that the following inequality holds for every subset $E\subset \er^n$:
\eqn{othercontrol2}
$$
\mathcal{H}_{H^{1+\delta}}^{-}(E) \leq \mathcal{H}_{H^{1+\delta}}(E) \leq \mathcal{H}_{H^{1+\delta}}^{+}(E) \leq c \mathcal{H}_{H^{1+\delta}}^{-}(E)\;.
$$
\end{proposition}
Following \cite{BHH}, we now introduce a notion of (relative) capacity generated by the function $\Phi(\cdot)$. For a compact subset $K\subset \mathbb{R}^{n}$, we denote
\eqn{cappi0}
$$\cc{{Cap}_{\Phi}^*(K)\equiv {Cap}_{\Phi}^*(K,\Omega):=\inf_{f \in \mathcal{R}(K)}\int_{\Omega}\f(x,\snr{Df})\, dx}$$ 
where 
$$\cd{\mathcal{R}(K):=\left\{ \ f \in W^{1,\Phi}(\Omega)\cap C_{0}(\Omega) \colon \  f\ge 1 \ \mathrm{in}\ K, \ f\ge 0 \  \right\}\;.}
$$
As usual, for open subsets $U\subset \Omega$ we set
$$
{Cap}_{\Phi}(U) := \sup_{K\subset U, \  K\rm {\ is \ compact}}\, {Cap}_{\Phi}^*(K) $$ and then, for general sets $E\subset \Omega$ we finally define
 $${Cap}_{\Phi}(E) := \inf_{E\subset \tilde U\subset \Omega, \ \tilde U\rm {\ is \ open}}\, {Cap}_{\Phi}(\tilde U)\;.
$$
It turns out that, under the present assumptions on $\Phi(\cdot)$, we have ${Cap}_{\Phi}^*(K)={Cap}_{\Phi}(K)$, whenever $K \subset \Omega$ is a compact subset and therefore the symbol ${Cap}_{\Phi}^*$ will not be used anymore, \ccc{see \cite[Proposition 6.3]{BHH}}. Anisotropic capacities of this kind have been studied at length in the literature. Classical reference in this respect are \cite{choquet, frehse, HKM, mazya, niemi}. Here we refer to the recent paper \cite{BHH}, where such capacities have been studied in detail under the assumptions in \rif{0777-1}-\rif{0777-22} considered here. These ensure that ${Cap}_{\Phi}$ enjoys the standard properties of Sobolev capacities; in particular ${Cap}_{\Phi}$ is a Choquet capacity in the sense that
\eqn{cho}
$$
{Cap}_{\Phi}(E) = \sup \, \left\{ {Cap}_{\Phi}(K) \, \colon \, \mbox{$K\subset E$ and $K$ is compact} \right\}
$$
holds for every Borel set $E\subset \Omega$, \ccc{\cite[Remark 3.6]{BHH}}.  Needless to say, in the case $\Phi(x, t)\equiv t^p$, $ {Cap}_{\Phi}$ coincides with the usual relative $W^{1,p}$-capacity. In the following we shall denote ${Cap}_{H}={Cap}_{\Phi}$ when $\Phi(x,t)=t^p+a(x)t^q$. Exactly as in the case of the $W^{1,p}$-capacity, we can prove a relation between capacity and Haudorff measures. For this, we need some more assumptions. Specifically, we assume that there exist $1<p<q<\infty$ such that
\begin{equation}\label{0777-2}
\frac{\Phi(x, s)}{s^p} \leq c_g\frac{\Phi(x, t)}{t^p} \quad \mbox{and}\quad \frac{\Phi(x, t)}{t^q}\le c_g \frac{\Phi(x, s)}{s^q} \qquad \mbox{whenever} \ 
0 < s \leq t\;,
\end{equation}
for some $c_g \geq 1$. We then have 
\begin{theorem}\label{L11}
Assume that \trif{controllo} and \trif{0777-2} are in force. Let $E\subset \mathbb{R}^{n}$ be such that $\mathcal{H}_{\f}(E)<\infty$, then $ {Cap}_{\Phi}(E)=0$.
\end{theorem}
It is now possible to \cd{improve the estimates of the Hausdorff measure of} the singular set $\Sigma_u:=\Omega\setminus \Omega_u$ from Theorem \ref{main1}. This is in the following:
\begin{theorem}\label{T3} Let $u \in W^{1,1}_{\loc}(\Omega,\SN)$ be a local minimizer of the functional $\mathcal F$ in \trif{genF} under the assumptions \trif{assF}-\trif{mainbound}, and let $\Omega_u\subset \Omega $ be its regular set in the sense of Theorem \ref{main1}. Assume that $q(1+\delta_g)\leq n$, where $\delta_g$ is the number appearing in \trif{maggioreiniziale}. Then
\eqn{zerom}
$$
\mathcal{H}_{H^{1+\delta_g}}(\Omega\setminus \Omega_u)=0 \quad \mbox{and therefore}\quad  {Cap}_{H^{1+\delta_g}}(\Omega\setminus \Omega_u)=0\;. 
$$
In particular, we have 
\eqn{zerop}
$$
\mathcal{H}^{n-p-p\delta_g}\left(\Sigma_u\right)=0\;,
$$
and
\eqn{zeroq}
$$
\mathcal{H}^{n-q-q\delta_g}\left(\Sigma_u \cap \{a(x)>0\}\right)=0\;. 
$$
\end{theorem}
\subsection{Overview of the paper}
As mentioned at the beginning of the Introduction, the aim of this paper is not only to prove regularity results for constrained \cc{local} minimizers of double phase functionals, but also to expose intrinsic techniques bound to cover general functionals of the type in \rif{PPhi}. In this sense, this paper further develops the ideas introduced in \cite{BCM3} to get general regularity methods for non-autonomous functionals and also simplifies some of the arguments presented there. Moreover, the techniques considered here provide new results also in the unconstrained case. For instance, a partial regularity theory which is analogous to the classical one for standard $p$-functionals can be derived in the double phase case too \ccc{(see Remark \ref{onlyre})}. The paper is structured as follows. In Section \ref{notationsec} we fix some notation. In Section \ref{sec3} we establish some basic energy and higher integrability inequalities adapting the path developed in \cite{CM1, CM2} to the manifold constrained case. This is based on a projection argument exposed in Lemma \ref{L4}. Moreover, we derive the precise form of the Euler-Lagrange equation of functionals of the type in \rif{genF}, under assumptions \cc{\rif{assF}-\rif{assF2}}. We finally readapt a Morrey type decay estimate originally proved in \cite{CM2} \ccc{(see Theorem \ref{morreydecayth})}. In Section \ref{harmonicsec} we develop an intrinsic harmonic type approximation result (compactness lemma), which is Lemma \ref{int}. The main novelty is that the energy bounds and the approximation are given directly in the intrinsic terms of a Musielak-Orlicz energy, rather that a more typical Orlicz one, as usually done in the literature \cite{DSV, DM1}. The lemma is quantitative, in the sense that it reveals a power type dependence of the constants. It therefore extends a similar result previously obtained in \cite{BCM3}, which was there considered in a more classical Orlicz setting. It is interesting to note that the proof of Lemma \ref{int} involves the use of an a priori smallness assumption (see \rif{har1} below) which is exactly the one which is needed to prove partial regularity in the subsequent Section \ref{morreydecayth}. The conceptual advantage of using such an approach becomes clear in Section \ref{morreydecayth}, where partial regularity and Theorem \ref{main1} are proved. The intrinsic approach adopted in Lemma \ref{int} allows \cc{avoiding} to readapt the elaborate arguments of \cite{BCM3, CM1, CM2} as at this point we can directly use the intrinsic Morrey decay estimate of Theorem \ref{morreydecayth} as a natural reference estimate. This incorporates the regularity information on the solutions indeed developed in \cite{BCM3, CM1, CM2}. The final outcome is a treatment which is close to the classical one proposed by Simon \cite{Simon} in the case \cd{of} harmonic maps. Finally, in Section \ref{cap-proofs} we develop the arguments concerning the Hausdorff type measures presented in Section \ref{weight}.

\section{Notation}\label{notationsec}
In this paper, following a usual custom, we denote by $c$ a general constant larger than one. Different occurences from line to line will be still denoted by $c$, while special occurrences will be denoted by $c_1, c_2,  \tilde c$ or the like. Relevant
dependencies on parameters will be \cc{emphasized} using parentheses, i.e., ~$c_{1}\equiv c_1(n,p,\nu, L)$ means that $c_1$ depends on $n,p,\nu, L$. We denote by $ B_r(x_0):=\{x \in \er^n \, : \,  |x-x_0|< r\}$ the open ball with center $x_0$ and radius $r\equiv r(B)>0$; when not important, or clear from the context, we shall omit denoting the center as follows: $B_r \equiv B_r(x_0)$. Very often, different balls in the same context will share the same center. We shall also denote $B_1 = B_1(0)$
if not differently specified. \cd{Moreover}, with $B$ being a given ball with radius $r\equiv r(B)$ and $\gamma$ being a positive number, we denote by $\gamma B$ the concentric ball with radius $\gamma r$. \cd{Finally, when referring to balls in $\mathbb{R}^{N}$, we will stress it with the apex "$N$", i.e.: $B_{r}^{N}(a_{0})$ is the open ball with center $a_{0}\in \mathbb{R}^{N}$ and positive radius $r$.} With $\mathcal B \subset \er^{n}$ being a measurable subset \cc{having} finite and positive measure $|\mathcal B|>0$, and with $g \colon \mathcal B \to \er^{k}$, $k\geq 1$, being a measurable map, we shall denote by  $$
   (g)_{\mathcal B} \equiv \mint_{\mathcal B}  g(x) \, dx  := \frac{1}{|\mathcal B|}\int_{\mathcal B}  g(x) \, dx
$$
its integral average. Similarly, with $\gamma \in (0,1)$ we denote 
$$
[g]_{0,\gamma; \mathcal B} := \sup_{x,y \in \mathcal B, x \not= y} \, 
\frac{|g(x)-g(y)|}{|x-y|^\gamma}\;.$$
In the case of the reference coefficient function $a(\cdot)$ in \rif{modello}, we shall denote $[a]_{0, \alpha}\equiv [a]_{0, \alpha; \Omega}$. As usual when dealing with $p$-Laplacean type problems, we shall often use the auxiliary vector fields $V_p,V_q\colon \er^n \to \er^n$ defined by
\eqn{vpvq}
$$
V_p(z):= |z|^{(p-2)/2}z \qquad \mbox{and}\qquad V_q(z):= |z|^{(q-2)/2}z
$$
whenever $z \in \ernN$; we notice that 
\eqn{elemH}
$$H(x, z)= \snr{V_{p}(z)}^2+a(x)\snr{V_{q}(z)}^2\;, \qquad \cc{\mbox{for all}} \ x\in \Omega,\ z\in \ernN\;.$$
A useful related inequality is contained in the following
\begin{flalign}\label{V}
\snr{V_{t}(z_{1})-V_{t}(z_{2})}\approx (\snr{z_{1}}+\snr{z_{2}})^{(t-2)/2}\snr{z_{1}-z_{2}}, \quad t \in \{p,q\}\;,
\end{flalign}
where the equivalence holds up to constants depending only on $n,N,t$. As a consequence of \rif{assF}$_3$, it can be proved that 
\eqn{monotonicity}
$$ 
|V_p(z_1)-V_p(z_2)|^2 + a(x) |V_q(z_1)-V_q(z_2)|^2  \leq c \left[ \partial F(x,v,z_1)-\partial F(x, v,z_2)\right]\cdot ( z_1-z_2)
$$
holds whenever $z_1, z_2 \in \ernN$, $x\in \Omega$, $v\in \erN$ and  
with $c \equiv c (n,N,\nu, p,q)$. For this see for instance \cite{BCM3, KM1}. We similarly have, again from $\rif{assF}_{3}$ 
\begin{flalign}\label{pre}
&\snr{V_{p}(z_{2})-V_{p}(z_{1})}^{2}+a(x)\snr{V_{q}(z_{2})-V_{q}(z_{1})}^{2}
+\partial F(x,v,z_1)\cdot (z_{2}-z_{1})\nonumber \\
&\quad \quad \quad \le c\left[F(x,v, z_{2})-F(x,v,z_{1})\right]\;,
\end{flalign}
\cc{whenever $z_{1},z_{2}\in \mathbb{R}^{N\times n}$,} again for $c\equiv c (n,N,\nu, p,q)$. We next recall some basic terminology about \cc{Musielak-Orlicz-Sobolev spaces}. The space \cc{$W^{1, H}(\Omega,\RN)$} is defined as in \rif{Orlicz-Sob} with the choice $\Phi(\cdot)\equiv H(\cdot)$, with the local variants being defined in the obvious way and $W^{1, H}_0(\Omega)=W^{1, H}(\Omega)\cap W^{1,p}_0(\Omega)$. In the same way, we set 
$$\cc{W^{1, H}(\Omega,\SN):= \left\{\ w \in W^{1,H}(\Omega,\er^N)\, \colon \, |w|=1 \ \mbox{holds a.e.  }\right\}\;,}$$
with the local variants defined in a similar fashion. Finally, with $u \in W^{1, H}(\Omega,\SN)$ we denote the Dirichlet class 
$$\cc{W^{1, H}_u(\Omega,\SN):= \left\{\ w \in W^{1,H}(\Omega,\SN)\, \colon \, u-w \in W^{1, 1}_0(\Omega,\er^N) \right\}\; \ }.$$
We similarly define the Dirichlet class of unconstrained maps $W^{1, H}_u(\Omega,\er^N)$.  
Moreover, with \cd{$w \in W^{1,H}(\tilde{\Omega},\er^N)$ and $\tilde{\Omega}$} being a domain that allows for a trace operator (for instance, this happens when \cd{$\partial\tilde{\Omega}$} is Lipschitz), we denote by \cd{$\texttt{tr}(w,\partial \tilde{\Omega})$} the trace of \cc{$w$ on $\partial \tilde{\Omega}$}.

\section{Basic material}\label{sec3}
\subsection{Caccioppoli's and higher integrability inequalities}\label{cacciosec}
Following the path established in  \cite{CM1, CM2}, in this section we gather a few technical inequalities for minimizers of functionals with double 
phase. The main difference is that now the setting is the one of constrained variational problems. Therefore, in several cases, we shall confine ourselves to give the necessary modifications to the proofs proposed in \cite{CM1, CM2}. 
We start with the following lemma; this provides an extension result in the spirit of \cite{HKL}.
\begin{lemma}\label{L4}
Let $\tilde \Omega \subset \Omega$ be a bounded, Lipschitz domain in $\mathbb{R}^{n}$ and $v\in W^{1,H}(\tilde \Omega,\RN)$ be such that $v(\partial \tilde \Omega)\subset \SN$. Then there exists $c\equiv c(n,N,p,q)$ and $\tilde{v}\in W^{1,H}(\tilde \Omega,\SN)$ satisfying
\begin{flalign}\label{023}
\int_{\tilde \Omega}H(x,D\tilde{v}) \, dx\le c\int_{\tilde \Omega}H(x,Dv) \, dx \quad and \quad v-\tilde{v} \in W^{1,1}_0(\tilde \Omega,\RN)\;.
\end{flalign}
\end{lemma}
\begin{proof}
\cd{For $a\in B^{N}_{1/2}(0)\subset \mathbb{R}^{N}$} and $v$ as in the statement of the lemma, define the map 
$$ 
\cc{v^{a}(x):=\frac{v(x)-a}{\snr{v(x)-a}}\,,\qquad x \in \tilde \Omega\;.}$$
Clearly, it is 
$$
\snr{Dv^{a}(x)}\le \frac{2\snr{Dv(x)}}{\snr{v(x)-a}}\,,$$
\cc{so that we can estimate}
\begin{eqnarray*}
\int_{B^{N}_{1/2}(0)}H(x, Dv^{a}) \ da&\le &c\int_{B^{N}_{1/2}(0)}H\left(x,\frac{Dv}{|v-a|}\right) \ da\nonumber \\
& =  &c\left(\snr{Dv}^{p}\int_{B_{1/2}^{N}(0)}  \frac{da}{\snr{v-a}^{p}}+a(x)\snr{Dv}^{q}
\int_{B^{N}_{1/2}(0)} \frac{da}{\snr{v-a}^{q}}\right) \le cH(x,Dv)\;.
\end{eqnarray*}
Here $c\equiv c(N,p,q)$ and we have used the assumption \ccc{$\rif{mainbound}_{2}$}; this makes the integrals in the above line finite. Integrating over $\tilde \Omega$, using Fubini's theorem and the content of the last display, we obtain
$$
\mint_{B_{1/2}^{N}(0)}\int_{\tilde \Omega}H(x,Dv^{a}) \, dx\, da\le c\int_{\tilde \Omega}H(x,Dv) \, dx\;.
$$
\cc{By Chebyshev inequality, this} yields the existence of $a_{0}\in B^{N}_{1/2}(0)$ such that
\begin{flalign}\label{021}
\int_{\tilde \Omega}H(x,Dv^{a_{0}}) \, dx \le c\int_{\tilde \Omega}H(x,Dv) \, dx\;,
\end{flalign}
with $c\equiv c(n,N,p,q)$. Let us consider the projector
\begin{flalign*}
\cc{\Pi_{a}(y):=\frac{y-a}{\snr{y-a}}, \quad \mbox{for} \quad y \in \SN \ \mbox{and} \ a \in B^{N}_{1/2}(0)\;.}
\end{flalign*}
\cd{Such a projector} is a bilipschitz map $\SN$ into itself, and it is such that 
\begin{flalign}\label{lip}
\cd{\left[\nabla (\Pi_{a}^{-1}) \right]_{0,1}\le c=c(N)\;,}
\end{flalign}
an estimate which is independent of $a\in B^{N}_{1/2}(0)$. \cd{Since $v^{a}(x)\in \mathbb{S}^{N-1}$ for a.e. $x\in \Omega$} \ccc{and all $a\in B^{N}_{1/2}(0)$}, we may define \cd{$\tilde{v}:=\Pi_{a_{0}}^{-1}\circ v^{a_{0}}$} which has the requested features. In fact, since $v(\partial \tilde \Omega)\subset \SN$, we have
%$$
%\left. \tilde{v}\right |_{\partial \tilde \Omega}=\Pi_{a_{0}}^{-1}\left(\left.v^{a_0}\right |_{\partial \tilde \Omega}\right)=\Pi_{a_{0}}^{-1}\left(\left.\Pi_{a_{0}}\eft(v\right |_{\partial \tilde \Omega}\right)\right)=\left.v\right |_{\partial \tilde \Omega},
%$$
$$
\cd{\texttt{tr} (\tilde{v},\partial \tilde{\Omega})=\Pi_{a_{0}}^{-1}\left(\texttt{tr}(v^{a_0},\partial\tilde{\Omega})\right)=\Pi_{a_{0}}^{-1}\left(\Pi_{a_{0}}\left(\texttt{tr}(v,\partial \tilde{\Omega})\right)\right)=\texttt{tr}(v,\partial \tilde{\Omega})}
$$
and, by \rif{lip} and \rif{021},
\begin{flalign*}
\int_{\tilde \Omega}H(x,D\tilde{v}) \, dx\le c\int_{\tilde \Omega}H(x,Dv^{a_0}) \, dx\le c\int_{\tilde \Omega}H(x,Dv) \, dx\;,
\end{flalign*}
where $c\equiv c(n,N,p,q)$, so that \rif{023} is proved in view of the last two displays.
\end{proof}
\begin{remark}\label{onlyre}
\emph{The condition $q<N$ in \trif{cond-pq} enters only in the proof of the above lemma and therefore can be dropped when  adapting the proofs given here to the unconstrained case.}
\end{remark}
Lemma \ref{L4} allows to derive in the new constrained setting a number of preliminary tools that have been already obtained and used in the unconstrained one \cite{CM1, CM2}. We shortly report them, with some additional modification and informations. 
\begin{lemma}[Caccioppoli's Inequality]\label{L5}
Let $u \in W^{1,H}(\Omega,\SN)$ be a constrained local minimizer of the functional $\mathcal F$ in \trif{genF} under (only) assumptions \trif{assF}$_1$ and $q\leq p+\alpha$. Then there exists \ccc{$c\equiv c(n,N,\nu,L,p,q)>0$} such that for any choice of concentric balls $B_{r}\subset B_{R}\Subset \Omega$ there holds
\begin{flalign}\label{caccine}
\int_{B_{r}}H(x, Du) \, dx\le c\int_{B_{R}}H\left(x,\frac{u-(u)_{B_{R}}}{R-r} \right) \, dx
\end{flalign}
and, if $R\leq 1$, it also holds that
\eqn{caccine2dopo}
$$
\int_{B_{R/2}}H(x, Du) \, dx\le c\int_{B_{R}}H^{-}_{B_R}\left(\frac{u-(u)_{B_{R}}}{R} \right) \, dx\;,
$$
for \cc{$c\equiv c(n,N,\nu,L,p,q, [a]_{0, \alpha})>0$}. 
Moreover, if
\begin{flalign}\label{condizione}
\cd{\inf_{B_R}\, a(x)\leq 4 [a]_{0, \alpha}R^\alpha}
\end{flalign}
\cd{holds and again it is} $R\leq 1$, then \eqref{caccine2dopo} reduces to
\begin{flalign}\label{caccine2}
\int_{B_{R/2}}H(x, Du) \, dx\le c\int_{B_{R}}\left|\frac{u-(u)_{B_{R}}}{R} \right|^p \, dx\;,
\end{flalign}
with \cd{$c \equiv c (n,N,\nu,L,p,q,[a]_{0, \alpha}) $}. Finally, these facts still hold for an unconstrained local minimizer $u \in (W^{1,H}\cap L^{\infty})(\Omega,\er^{N})$, with all the constants depending in addition on $\|u\|_{L^{\infty}}$. 
\end{lemma}
\begin{proof} The proof is a modification of the one originally given in \cite[Theorem 1.1, (1.8)]{CM1}; we furthermore specialize to the case of constrained minimizers, the unconstrained one being totally analogous. In the following all the balls will be concentric to the ones mentioned in the statement of the lemma. 
With $r\le t<s \le R$, we determine a cut-off function $\eta\in C^{\infty}_{c}(B_{s})$ such that
$ 
\chi_{B_{t}}\le \eta\le \chi_{B_{s}}$ and $\snr{D\eta}\le 4/(s-t).
$ 
Consider now the function
$
w (x)=u(x)-\eta(u-(u)_{B_{R}}).
$
Since $\eta$ is smooth and $u \in W^{1,H}(B_{s},\SN)$, then obviously $w  \in W^{1,H}(B_{s},\RN)$ and 
$
u-w  \in W^{1,1}_0(B_s,\RN).
$
Lemma \ref{L4} yields the existence of $\tilde{w }\in W^{1,H}(B_{s},\SN)$ such that
\rif{023} holds with $\tilde \Omega \equiv B_s$, 
where $c\equiv c(n,N,p,q)$. The minimality of $u$, \trif{assF}$_1$ and \rif{023} 
(with  $\Omega \equiv B_s$) yield
\begin{eqnarray}
\notag \nu\int_{B_{s}}H(x, Du) \, dx&\le&\int_{B_{s}}F(x,u,Du) \, dx\leq \int_{B_{s}}F(x,\tilde{w },D\tilde{w }) \, dx
\le  L\int_{B_{s}}H(x, D\tilde{w }) \, dx\\
 \notag & \le&  c\int_{B_{s}}H(x,Dw ) \, dx\le c\int_{B_{s}\setminus B_{t}}\left[H(x,(1-\eta)Du)+H(x, (u-(u)_{B_{R}})\otimes D\eta) \right]\, dx\nonumber \\
&  \leq &c\int_{B_{s}\setminus B_{t}}H(x, Du) \, dx+c\int_{B_{s}\setminus B_{t}}H\left(x,
\frac{u-(u)_{B_{R}}}{s-t} \right) \, dx\;,\label{bigrowth}
\end{eqnarray}
with \ccc{$c\equiv c(N,\nu,L,p,q)$}. The proof of \rif{caccine} can be now concluded by filling the hole and iteration, as in \cite[Theorem 1.1, (1.8)]{CM1}, 
see also \cite{CM2}. As for the proof of \rif{caccine2dopo} we simply estimate (as it is $q<p+\alpha$ and $R\leq 1$), for $x \in B_R$
\begin{eqnarray*}
H\left(x,\frac{u-(u)_{B_{R}}}{R} \right)& \leq & H_{B_R}^{-} \left(\frac{u-(u)_{B_{R}}}{R} \right)
+\sup_{B_{R}}\, \left[a(x)-\ai(B_R)\right] \left|\frac{u-(u)_{B_{R}}}{R} \right|^q\\
&\leq & H_{B_R}^{-} \left(\frac{u-(u)_{B_{R}}}{R} \right)
+\ccc{2^{\alpha+q-p} [a]_{0, \alpha} R^{\alpha+p-q} \left|\frac{u-(u)_{B_{R}}}{R} \right|^p}\leq 
cH_{B_R}^{-} \left(\frac{u-(u)_{B_{R}}}{R} \right)\;,
\end{eqnarray*}
\ccc{for $c=c(p,q,\alpha)$,} and \cd{\rif{caccine2dopo}} follows from \rif{caccine} with \cd{$r=R/2$}. Finally, for \rif{caccine2}, we similarly observe that (still $x \in B_R$)
$$
\cd{\ai(B_R) \left|\frac{u-(u)_{B_{R}}}{R} \right|^q \stackleq{condizione} 8[a]_{0, \alpha} R^{\alpha+p-q} \left|\frac{u-(u)_{B_{R}}}{R} \right|^p\leq c  \left|\frac{u-(u)_{B_{R}}}{R} \right|^p\;,}
$$
so that \cd{\rif{caccine2}} follows from \rif{caccine2dopo} and the proof is complete. 
\end{proof}
We proceed with
\begin{lemma}[Intrinsic Sobolev-Poincar\'e inequality]\label{poinclemma} 
Let $v \in (W^{1,H}\cap L^{\infty})(\Omega,\erN)$, $N\geq 1$, and $B_r\Subset \Omega $ be a ball with radius $r \leq 1$, and assume that $q\leq p+\alpha$. Then the following inequality holds \begin{flalign}\label{poinc}
\mint_{B_{r}}H\left(x,\frac{v-(v)_{B_{r}}}{r}\right) \, dx \le c\left(\mint_{B_{r}}[H(x,Dv)]^{d} \, dx\right)^{1/d}\;,
\end{flalign}
where \cc{$c\equiv c(n,N,p,q,[a]_{0,\alpha}, \|v\|_{L^{\infty}(B_r)})\geq 1$} and $d\equiv d(n,p,q)< 1$. In \trif{poinc} we can replace $v-(v)_{B_{r}}$ by $v$ in case we also have that 
\cd{$\texttt{tr} (v,\partial B_{r}) \equiv 0$}. Finally, we can still replace $v-(v)_{B_{r}}$ by $v$, provided $v\equiv 0$ on $A\subset B_r$ and $|A|/|B_r|>\gamma >0$; in this last case the constant $c$ depends also on $\gamma$. 
\end{lemma}
\begin{proof} The proof is implicit in the one of \cite[Theorem 1.2]{CM2}, with minor modifications that are left to the reader. See also \cite{OK2}. 
\end{proof}
\begin{lemma}[Inner higher integrability]\label{C1} 
Let $u \in W^{1,H}(\Omega,\SN)$ be a constrained local minimizer of the functional $\mathcal F$ in \trif{genF} under (only) assumptions \trif{assF}$_1$ and \trif{mainbound}. There exists a positive integrability exponent $\delta_{g}\equiv \delta_{g}(\data)$, such that the following reverse inequality holds for every $B_{2R}\subset \Omega$ such that $R \leq 1$:
\eqn{maggiore-original}
$$
\left(\mint_{B_{R}}[H(x,Du)]^{1+\delta_{g}} \, dx\right)^{1/(1+\delta_{g})}\le c\mint_{B_{2R}}H(x, Du) \, dx\;,
$$
where $c\equiv c(\data)$.
\end{lemma}
\begin{proof} Also in this case, the proof follows the one for \cite[Theorem 1.2]{CM2}, which in turn only uses the assumed bound $q < p+\alpha$ and the validity of \rif{caccine}. \end{proof}
\begin{lemma}[Higher integrability up to the boundary]\label{C2} Let $u \in W^{1,H}(\Omega,\SN)$ be such that $H(\cdot,Du)\in L^{1+\delta}_{\rm{loc}}(\Omega)$, for some $\delta>0$, and, for $B_R \Subset \Omega$, $R \leq 1$, let $v \in W^{1,H}_{u}(B_{R},\SN)$ be a solution of 
$$
v\mapsto \min_{w\in W^{1,H}_{u}(B_{R},\SN)} \int_{B_{R}}F(x,w,Dw) \, dx\;,
$$
where the Carath\'eodory integrand $F(\cdot)$ satisfies (only) \trif{assF}$_1$ and \trif{mainbound}. Then there exists a positive exponent $\sigma_{g} \in (0, \delta)$ and a constant $c\geq 1$, both depending on \cc{$n,N,\nu,L,p,q,\alpha, [a]_{0,\alpha}$}, such that
%\eqn{maggiorev}
$$
\left(\mint_{B_{R}}[H(x,Dv)]^{1+\sigma_{g}} \, dx\right)^{1/(1+\sigma_{g})}\le c\left(\mint_{B_{R}}H(x,Du)]^{1+\sigma_{g}} \, dx\right)^{1/(1+\sigma_{g})}\;.
$$
Moreover, in the above display, $\sigma_g$ can be replaced by any smaller and positive number.  
\end{lemma}
\begin{proof} 
With $x_0\in B_R$, let us fix a ball $B_{r}(x_{0})\subset \mathbb{R}^{n}$ such that it is $\snr{B_{r}(x_{0})\setminus B_{R}}>\snr{B_{r}(x_{0})}/10$. Let us fix $r/2<t<s<r$ and take a cut-off function 
$
\eta \in C^{1}_{c}(B_{s}(x_{0}))$ such that 
$\chi_{B_{t}(x_{0})}\le \eta \le \chi_{B_{s}(x_{0})}$ and $\snr{D\eta}\le 4/(s-t)$.
The function $v-\eta(v-u)$ coincides with $v$ in $\partial B_R$ (in the sense of traces) and therefore we can apply Lemma \ref{L4}. This provides us with a map $w\in W^{1,H}_{v}(B_{s}(x_{0})\cap B_{R},\SN)$ such that 
$$
\int_{B_{s}(x_{0})\cap B_{R}}H(x,Dw) \, dx\leq \cc{c(n,N,\nu,L,p,q)}  \int_{B_{s}(x_{0})\cap B_{R}}H(x,D(v-\eta(v-u))) \, dx\;.
$$
\cc{The minimality of $v$ and \rif{assF}$_1$, together with the above inequality, yield}
\begin{align*}
& \int_{B_{s}(x_{0})\cap B_{R}}H(x,Dv) \, dx\le \frac{L}{\nu} \int_{B_{s}(x_{0})\cap B_{R}}H(x,Dw) \, dx\\
& \quad \leq c \int_{B_{s}(x_{0})\setminus B_t(x_0)\cap B_{R}}H(x,Dv) \, dx + 
c \int_{B_{s}(x_{0})\cap B_{R}}H(x,Du) \, dx+ c \int_{B_{s}(x_{0})\cap B_{R}} H\left(x,\frac{v-u}{s-t}\right) \, dx\;,
\end{align*}
\cc{with $c=c(n,N,\nu,L,p,q)$}. By filling the hole and iterating as for instance done in \cite[Proof of Theorem 1.8]{CM1}, we arrive at 
$$
 \int_{B_{r/2}(x_{0})\cap B_{R}}H(x,Dv) \, dx\leq c \int_{B_{r}(x_{0})\cap B_{R}} H\left(x,\frac{v-u}{r}\right) \, dx
 + c\int_{B_{r}(x_{0})\cap B_{R}}H(x,Du) \, dx\;,
$$
for \cc{$c\equiv c (n,N,\nu,L,p,q)$}. 
From this point on, we can follow the proof of \cite[Lemma 5]{cristiana3} but using the method of \cite[Theorem 1.2]{CM2}\cd{, see also \cite[Lemma 10]{me}.} 
\end{proof} 
\begin{remark}\label{reun} 
\emph{The assertion of Lemma \ref{C1} continues to hold in the case of unconstrained local minimizers $u \in W^{1,H}(\Omega,\erN)$, \ccc{$N>1$}, such that $u \in L^{\infty}(\Omega)$; in this case $c$ and $\delta_g$ also depend on $\|u\|_{L^{\infty}}$; for this see the original proof in \cite{CM2} and the extensions made in \cite{OK2}. Moreover, Lemma \ref{C2} still holds when $u,v \in L^{\infty}(\Omega, \er^N)$ and, also in this case, 
$c$ and $\sigma_g$ again depend on $\|u\|_{L^{\infty}(B_R)}$ and $\|v\|_{L^{\infty}(B_R)}$. The proof follows again the one proposed in \cite[Lemma 5]{cristiana3} and \cite[Theorem 1.2]{CM2}. For later convenience we discuss a case when assumptions in \rif{assF} are relaxed. Instead of \rif{assF}$_1$, we consider 
\eqn{assFalt}
$$
\tilde \nu(M) H(x,z) \leq F(x,v,z)\le \tilde L(M)H(x,z)
$$
to be satisfied as in \rif{assF}$_1$, whenever $|v|\leq M$, where $0 < \nu (M) \leq 1 \leq L(M)$ are, respectively, non-increasing and non-decreasing functions of \cc{$M\geq 3N$}. Both Lemma \ref{C1} and of Lemma \ref{C2} hold assuming \rif{assFalt} instead of \rif{assF}$_1$, with exponents $\delta_g, \sigma_g$ depending again on $\|u\|_{L^{\infty}(B_R)}$ and $\|v\|_{L^{\infty}(B_R)}$. This can be easily seen (for instance in the proof of Lemma \ref{C1}) by observing that $\|\tilde w\|_{L^{\infty}(B_R)}\leq 3\|u\|_{L^{\infty}(B_R)}$ and therefore \rif{assFalt} can be used with $M$ depending only on $\|u\|_{L^{\infty}(B_R)}$ in \rif{bigrowth}. In the same way, the content of Lemma \ref{L5} still holds under assumptions \rif{assFalt}. 
} 
\end{remark}

\begin{remark}\label{reun2}
\emph{ The content of Lemma \ref{C2} applies in particular to the case when the function $a(x)\equiv a_0\geq 0$ is  constant and $H(x, z)\equiv H_0(z):= |z|^p+a_0|z|^q$. In this case assumption \cc{$\rif{mainbound}_{1}$} is not necessary and the statement continues to hold whenever $p\leq q$ are arbitrary.  }
\end{remark}
\subsection{On the Euler-Lagrange equation under non-standard growth conditions}\label{eldp}

Let us consider a ball $B_{r}\Subset \Omega$ and $v \in W^{1,H}(B_{r},\SN)$ being a solution of the frozen Dirichlet problem
\begin{flalign}\label{elfrz}
v\mapsto \min_{w \in W^{1,H}_{u}(B_{r},\SN)}\int_{B_{r}}g(x,Dw) \, dx\;,
\end{flalign}
where, \cc{with $\bar{u} \in \er^N$ being fixed}, we have set
\eqn{lagg}
$$\cc{g(x,z):=F(x,\bar{u},z)=\tilde{F}(x,\bar{u},\snr{z})\;,}$$ for $x\in \Omega$ and $z \in \er^{N\times n}$ and of course $F(\cdot)$ is the Carath\'eodory integrand considered in \cd{\rif{genF}}; this time we assume that $F(\cdot)$ satisfies only \rif{assF}$_{1,2}$. By definition, $g(\cdot)$ matches \cd{$\eqref{assF}_{1}$-$\eqref{assF}_{2}$}. Because of the non-standard growth conditions considered here, we cannot derive the Euler-Lagrange equation for \eqref{elfrz} in the usual way, adopting variations defined through smooth $\varphi$ and then concluding via a density argument. We shall rather use a direct argument,  eventually leading to establish that
\begin{eqnarray}
&&\cc{\notag \int_{B_{r}}\tilde{F}'(x,\bar{ u},\snr{Dv})\frac{Dv}{\snr{Dv}}\cdot D\varphi-\tilde{F}'(x,\bar{u},\snr{Dv})\snr{Dv}(v\cdot\varphi) \, dx}\\
&&\qquad \cc{=\int_{B_{r}}\partial F(x,\bar{u},\snr{Dv})\cdot D\varphi-\tilde{F}'(x,\bar{u},\snr{Dv})\snr{Dv}(v\cdot\varphi) \, dx=0}
\label{EL1} 
\end{eqnarray}
holds whenever $\varphi\in (W^{1,H}_{0}\cap L^{\infty})(B_{r},\RN)$; needless to say, the symbol $\tilde{F}'$ denotes the derivative of $\tilde{F}$ with respect to the last variable. For the sake of completeness we report all the details. To proceed, for $s\in (0,1)$, define the variation $v_{s}:=\Pi(v+s\varphi)$, where $\Pi(y)=y/\snr{y}$ for $y \in \RN\setminus \{0\}$. \cd{Clearly, for $s$ sufficiently small, $v_{s}\in W^{1,H}_{u}(B_{r},\mathbb{S}^{N-1})$}. The minimality of $v$ and the very definition of $v_{s}$ tell us that
\begin{flalign}\label{el1}
0\le &\int_{B_{r}}\frac{g(x,Dv_{s})-g(x,Dv)}{s} \, dx=\frac 1s\int_{B_{r}}\left(\int_{0}^{1}\partial g(x,\lambda Dv_{s}+(1-\lambda)Dv) \ d\lambda \right)\cdot(Dv_{s}-Dv) \, dx\;.
\end{flalign}
We aim to pass to the limit with $s\to 0$ in \eqref{el1} via the dominated convergence theorem. A direct computation and \cd{the fact that} $\nabla \Pi (v)Dv=Dv$ show that
\begin{flalign}\label{el2}
\snr{Dv_{s}-Dv}&=\snr{\nabla \Pi(v+s\varphi)(Dv+sD\varphi)-\nabla \Pi(v)Dv}\nonumber \\
&\leq \snr{\nabla \Pi(v+s\varphi)-\nabla \Pi(v)}|Dv| + \snr{s\nabla \Pi(v+s\varphi)D\varphi}\le cs\left(\snr{\varphi}\snr{Dv}+\snr{D\varphi}\right)\;,
\end{flalign}
where $c\equiv c(\nr{\nabla\Pi}_{\infty},\nr{\nabla^{2}\Pi}_{\infty})\equiv c(N)$. Plugging \eqref{el2} into the last term in the right-hand side of \eqref{el1} we obtain
\begin{eqnarray*}
\frac 1s\left | \ \left(\int_{0}^{1}\partial g(x,\lambda Dv_{s}+(1-\lambda)Dv) \ d\lambda \right)\cdot(Dv_{s}-Dv)\  \right |&\le & c(\snr{\varphi}\snr{Dv}+\snr{D\varphi})\int_{0}^{1}\snr{\partial g(x,\lambda Dv_{s}+(1-\lambda)Dv)} \ d\lambda\\
&\le & c[\snr{\mathrm{(I)}}+\snr{\mathrm{(II)}}]\;,
\end{eqnarray*}
with $c\equiv c(N)$. From $\eqref{assF}_{2}$, Young's inequality and \eqref{el2}, we estimate
\cd{
\begin{eqnarray*}
\snr{\mathrm{(I)}}%\le& c\nr{\varphi}_{L^{\infty}(B_{r})}\int_{0}^{1}\snr{\partial g(x,\lambda Dv_{s}+(1-\lambda)Dv)}\snr{Dv}d\lambda\nonumber \\
&\le& c\nr{\varphi}_{L^{\infty}(B_{r})}\int_{0}^{1}\frac{H(x,\lambda Dv_{s}+(1-\lambda)Dv)}{\snr{\lambda Dv_{s}+(1-\lambda)Dv}}\snr{Dv} \ d\lambda\nonumber \\
&\le & c\nr{\varphi}_{L^{\infty}(B_{r})}\int_{0}^{1}[H(x,\lambda Dv_{s})+H(x,(1-\lambda)Dv)+H(x,Dv)]\ d\lambda\le c[H(x,D\varphi)+H(x,Dv)]\;,
\end{eqnarray*}
}
where $c\equiv c(N,L,p,q,\nr{\varphi}_{L^{\infty}(B_{r})})$. In a similar way we also have
\cd{
\begin{eqnarray*}
\snr{\mathrm{(II)}}%\le & c\int_{0}^{1}\snr{\partial g(x,\lambda Dv_{s}+(1-\lambda)Dv)}\snr{D\varphi}d\lambda\nonumber \\
&\le &c\int_{0}^{1}\frac{H(x,\lambda Dv_{s}+(1-\lambda)Dv)}{\snr{\lambda Dv_{s}+(1-\lambda)Dv}}\snr{D\varphi} \ d\lambda\nonumber \\
&\le &c\int_{0}^{1}[H(x,\lambda Dv_{s})+H(x,(1-\lambda)Dv)+H(x,D\varphi)] \ d\lambda\le c[H(x,D\varphi)+H(x,Dv)]\;,
\end{eqnarray*}
}
with \cd{$c\equiv c(N,L,p,q)$}. Merging the content of the last three displays yields
\eqn{el7}
$$
\frac 1s\left | \ \left(\int_{0}^{1}\partial g(x,\lambda Dv_{s}+(1-\lambda)Dv) \ d\lambda \right)\cdot (Dv_{s}-Dv)\  \right |\le  c[H(x,Dv)+H(x,D\varphi)]\;,
$$
again for $c\equiv c(N,L,p,q,\nr{\varphi}_{L^{\infty}(B_{r})})$. 
Finally, by the regularity of $z\mapsto g(\cdot,z)$, \eqref{el2} and $\eqref{assF}_{2}$ we notice that
\begin{flalign*}
\begin{cases}
\displaystyle  \snr{\partial g(x,\lambda Dv_{s}+(1-\lambda)Dv)} \ d\lambda\le c\left(\frac{H(x,Dv)}{\snr{Dv}}+\frac{H(x,D\varphi)}{\snr{D\varphi}}\right) \\  \\ 
\displaystyle \lim_{s\to 0}\partial g(x,\lambda Dv_{s}+(1-\lambda)Dv)=\partial g(x,Dv)\;,
\end{cases}
\end{flalign*}
thus, by the dominated convergence theorem,
\begin{flalign}\label{el5}
\lim_{s\to 0}\int_{0}^{1}\partial g(x,\lambda Dv_{s}+(1-\lambda)Dv) \ d\lambda=\partial g(x,Dv)\;.
\end{flalign}
\cc{Using} \eqref{el2} and \eqref{el5} we can compute
\begin{flalign}\label{el6}
\lim_{s\to 0}\frac 1s&\left(\int_{0}^{1}\partial g(x,\lambda Dv_{s}+(1-\lambda)Dv) \ d\lambda\right)\cdot(Dv_{s}-Dv)=\partial g(x,Dv)\cdot\nabla \Pi(v)D\varphi\nonumber \\
&+\lim_{s\to 0}\left(\int_{0}^{1}\partial g(x,\lambda Dv_{s}+(1-\lambda)Dv) \ d\lambda\right)\cdot \frac{\nabla \Pi(v+s\varphi)-\nabla \Pi(v)}{s}Dv\nonumber \\
&=\partial g(x,Dv)\cdot\left(\nabla \Pi(v)D\varphi+\nabla^{2}\Pi(v)\varphi Dv \right)\;.
\end{flalign}
Now, \eqref{el1}, \eqref{el7}, \eqref{el6} and the dominated convergence theorem render
\begin{flalign*}
0\le \int_{B_{r}}\partial g(x,Dv)\cdot \left(\nabla \Pi(v)D\varphi+\nabla^{2}\Pi(v)\varphi Dv\right) \, dx\;.
\end{flalign*}
The same argument with $s \in (-1,0)$ finally yields 
\begin{flalign*}
\int_{B_{r}}\partial g(x,Dv)\cdot\left(\nabla \Pi(v)D\varphi+\nabla^{2}\Pi(v)\varphi Dv\right) \, dx=0\;. 
\end{flalign*}
Taking into account the symmetry of the Jacobian of the projector, we can conclude that
\cc{\begin{flalign*}
0&=\int_{B_{r}}\tilde{F}'(x,\bar{ u},\snr{Dv})\frac{Dv}{\snr{Dv}}\cdot \left(\nabla \Pi(v)D\varphi+\nabla^{2}\Pi(v)\varphi Dv\right) \, dx \nonumber \\
&=\int_{B_{r}}\tilde{F}'(x,\bar{ u},\snr{Dv})\frac{Dv}{\snr{Dv}}\cdot D\varphi-\frac{\tilde{F}'(x,\bar{ u},\snr{Dv})}{\snr{Dv}}A_{v}(Dv,Dv)\varphi \, dx\nonumber \\
&=\int_{B_{r}}\tilde{F}'(x,\bar{ u},\snr{Dv})\frac{Dv}{\snr{Dv}}\cdot D\varphi-\tilde{F}'(x,\bar{u},\snr{Dv})\snr{Dv}(v\cdot\varphi )\, dx\;,
\end{flalign*}}
where in the last \ccc{line} we have used the explicit expression of the second fundamental form $A_{v}(\cdot, \cdot)$ of $\SN$; see also \cite[Section 2.2]{Simon}. We have therefore proved the validity of \rif{EL1}. 
\subsection{A Morrey type decay estimate}\label{morreydecsec}
In this section we briefly revisit some scalar regularity results reported in \cite{BCM3, CM2}, adapting them to the vectorial case. We consider unconstrained local minimizers of functionals of the type
\eqn{funzionale-costretto}
$$
w\mapsto \int_{B_{r}}g(x,Dw) \, dx
$$ 
under the structure condition \rif{lagg}. We then have the following:
\begin{theorem}[\cite{BCM3, CM2}]\label{morreydecayth} Let $h \in (W^{1,H}\cap L^{\infty})(\Omega, \erN)$ be a local minimizer of the functional in \trif{funzionale-costretto} under assumptions \cc{\trif{assF}$_{1,2,3,4}$, \eqref{assF2}$_{1}$,\trif{lagg} and $\eqref{mainbound}_{1}$}. Then for every $\sigma \in (0,n]$ there exists a constant \ccc{$c \equiv c (\texttt{data}, \|u\|_{L^\infty(\Omega)}, \sigma)$} such that \eqn{morreydecayH}
$$\int_{B_{t}}H(x,Dh) \, dx \le c\left(\frac ts\right)^{n-\sigma}\int_{B_{s}}H(x,Dh) \, dx$$ holds whenever 
\cd{$B_t\subset B_s \subset \Omega$} are concentric balls such that $t \leq 1$. 
\end{theorem}
\begin{proof} The proof can be obtained tracking the ones given for \cite[Proposition 3.4]{CM2} and \cite[Theorem 2]{BCM3}, according to the remarks made in the proof of \cite[Theorem 5.2]{CM2}, that describes the modifications to make with respect to the scalar case. Reference \cite{BCM3} is actually more suitable as the regularity results are proved for general functionals, without assuming the splitting structure considered in \cite{CM2}. The remarks given in \cite[Section 5,2]{CM2} will be also useful here. As an outcome of the proofs of \cite[Proposition 3.4]{CM2} and \cite[Theorem 2]{BCM3}, estimate \rif{morreydecayH} follows provided 
\cd{$B_t\subset B_s \subset \Omega_0 \Subset \Omega$} are concentric balls and with an additional dependence of the constant $c$ on $\dist (\Omega_0, \Omega)$, but under the full bound $q \leq p+\alpha$. As remarked in \cite{BCM3, CM2}, it is possible to reach the borderline case $q\leq p+\alpha$ in the scalar case by using the preliminary local H\"older continuity of $h$ for some exponent $\gamma \in (0,1)$ (see \cd{\cite[Proposition 3.1]{CM2}}); the same happens in \cd{\cite[Theorem 6]{BCM3}}. This comes along with an a propri estimate of the type $[h]_{0, \gamma;\Omega_0}< c$, where, amongst the other things, $c$ depends also on $\dist (\Omega_0, \Omega)$. This is exactly the point where the dependence on $\dist (\Omega_0, \Omega)$ comes from in the final statement of \rif{morreydecayH} from \cite[Proposition3.4]{CM2} and \cite[Theorem 2]{BCM3}. On the other hand, as already remarked in the proof of \cite[Theorem 5.2]{CM2}, when considering the bound $q<p+\alpha$ we can avoid using that $u \in C^{0, \gamma}_{\loc}(\Omega)$ and in this way, taking into account the proofs in \cite{BCM3}, we arrive at \rif{morreydecayH} with the dependence of the constant $c$ as described in the statement of Theorem \ref{morreydecayth}. Notice that, in order to prove \rif{morreydecayH}, in \cite{BCM3, CM2} it is also necessary to replace the a priori Lipschitz estimate for minima of frozen functionals in \cite[(132)]{BCM3} with an analogous one for the vectorial case. This is discussed in Remark \ref{discuss-re} \cc{below}. Notice that here we are \cc{not assuming that} the function $\tilde F(\cdot)$ satisfies \rif{assF2}$_2$. \end{proof}
\begin{remark}\label{discuss-re} \emph{
Let us consider a local minimizer $v \in W^{1, H_0}(B_r, \er^{N\times n})$ of the functional 
\eqn{funzionale-costretto-zero}
$$
w\mapsto \int_{B_{r}}g(x_0,Dw) \, dx \qquad x_0\in \Omega\;,
$$ 
where $H_0(z)\equiv |z|^p + a(x_0)|z|^q$. The 
following estimate holds:
\eqn{suppistima}
$$
\sup_{B_{r/2}}\, H_0(Dw) \leq c \mint_{B_r} H_0(Dw)\, dx \;,
$$
where $\cd{c\equiv c (n,N,\nu, L,p,q)}$. This estimate plays a crucial role in the proofs given in \cite{BCM3, CM2}, and these are concerned with the scalar case. To get that this result holds in our vectorial case too it is sufficient to prove that
\eqn{equivalenti}
$$
\cc{\tilde{F}''(x,\bar{u},t) t \approx \tilde{F}'(x,\bar{u},t) \qquad \mbox{for every}\ t>0}
$$  
(for implied constants depending only on \cd{$n,N,\nu, L,p,q$}) and then appeal for instance to \cite[Lemma 5.8]{DSV2}. Indeed, \cd{following \cite[Lemma 3.4]{meom}, we see that} 
\eqn{iduetermini}
$$
\cc{ \partial^2 g(x_0, z) 
  = \tilde{F}''(x_0,\bar{u},\snr{z}) \frac{z \otimes z}{|z|^2} + \tilde{F}'(x_0,\bar{u},\snr{z})\left[ \frac{\mathbb{I}_{N\times n}}{|z|} - \frac{z \otimes z}{|z|^3} \right] \;, }
$$ 
holds for every $z \in \er^{N\times n}$ such that $|z|\not=0$; here it is $\mathbb{I}_{N\times n} = \delta_{ij}\delta_{\alpha \beta}$. Testing the above inequality for $\xi \bot z$ and for \cc{$\xi = z$} and using \eqref{assF}$_{2,3}$ yields 
\eqn{leduezero}
$$ \cc{\frac{\nu H(x_0,t)}{t} \leq \tilde{F}'(x_0,\bar{ u},t)\leq \frac{LH(x_0,t)}{t}  \quad \mbox{and} \quad \frac{\nu H(x_0,t)}{t^2}  \leq \tilde{F}''(x_0,\bar{u},t)\leq  \frac{LH(x_0,t)}{t^2} \;,}$$ respectively, for every $t>0$, so that \rif{equivalenti} follows. Notice that here we are only assuming that $\tilde F(\cdot)$ satisfies only  \rif{assF2}$_1$.}
\end{remark}
\begin{remark}\label{equiequi} 
\emph{
This is a side remark of later use. Assuming that the function $\tilde F(\cdot)$ satisfies \rif{assF2}$_{1,2}$, \cd{as in \cite[Lemma 3.4]{meom}}, by using \rif{leduezero} we get that \rif{assF2}$_2$ can be reformulated as
\eqn{newform}
$$\left|\tilde F''(x,v,t+s)-\tilde F''(x,v,t)\right|\leq \cd{c (n,N,\nu, L,p,q)}\tilde F''(x,v,t)\left(\frac{|s|}{t}\right)^{\beta_1}\;.$$
}
\end{remark}

\section{Harmonic type approximation}\label{harmonicsec}
In this section we revisit the arguments of \cd{\cite{BCM3, DSV}}, to give two kinds of harmonic type approximation lemmas. The most peculiar one is the first, which is given in terms of a generalized Young functions (specifically, $H(\cdot)$), rather than a usual Young function. Therefore all the arguments used there will be of intrinsic type. This perfectly combines with the type of intrinsic estimates already proved in \cite{BCM3, CM1, CM2}, as we shall see in the next section when \cd{showing} regularity theorems. \cc{Accordingly to the notation already established in \rif{H+-}}, with $B_\varrho \Subset \Omega$ being a ball, we shall denote
\eqn{lemappe} $$
H^{-}_{B_\varrho}(t)=t^{p}+\ai(B_\varrho)t^{q}\qquad \mbox{and} \qquad 
H^{+}_{B_\varrho}(t)=t^{p}+\as(B_\varrho)t^{q}\;.$$ 
We shall again denote, with abuse of notation, $H^{-}_{B_\varrho}(z)\equiv H^{-}_{B_\varrho}(|z|)$ and so forth, also in the case $z \in \ernN$. 

\begin{remark}\label{h-}
\emph{We collect some features of the functions in \rif{lemappe}. We first notice that \ccc{$H^{\pm}_{B_{\varrho}}(\cdot)$} is a Young function in the sense of \cite[Section 2]{DSV2}, and satisfies the $\Delta_{2}$-condition. Since \ccc{$t\mapsto H^{\pm}_{B_{\varrho}}(t)$} is strictly increasing and strictly convex, its inverse \ccc{$(H^{\pm}_{B_{\varrho}})^{-1}$} is strictly increasing and \cd{strictly concave} and \ccc{$(H^{\pm}_{B_{\varrho}})^{-1}(0)=0$}, thus \ccc{$(H^{\pm}_{B_{\varrho}})^{-1}$} is subadditive. Therefore, for all $\lambda\ge 0$, \cc{the subadditivity and the monotonicity} of \ccc{$(H^{\pm}_{B_{\varrho}})^{-1}(t)$} yield 
\eqn{homh}
$$
\ccc{(H^{\pm}_{B_{\varrho}})^{-1}(\lambda t)\le (\lambda+1)(H^{\pm}_{B_{\varrho}})^{-1}(t)}\;.
$$
In particular, if $\lambda \ge 1$, \ccc{$(H^{\pm}_{B_{\varrho}})^{-1}(\lambda t)\le 2\lambda (H^{\pm}_{B_{\varrho}})^{-1}(t)$}. Next, notice that if $B_{\varrho}=B_{\varrho}(x_{0})\Subset \Omega$, then the function \ccc{$(x_{0},\varrho,t)\mapsto H^{\pm}_{B_{\varrho}(x_{0})}(t)$} is continuous \cc{on $\Omega\times [0,\infty)\times [0,\infty)$}. This easily follows from the H\"older continuity of $a(\cdot)$. Finally, for $x_{0}$ fixed, if $\varrho_{1}\le \varrho_{2}$, then $H^{-}_{B_{\varrho_{1}}(x_{0})}(t)\ge H^{-}_{B_{\varrho_{2}}(x_{0})}(t)$, \ccc{($H^{+}_{B_{\varrho_{1}}(x_{0})}(t)\le H^{+}_{B_{\varrho_{2}}(x_{0})}(t)$)} holds uniformly in $t\ge 0$, and, as a consequence, $(H^{-}_{B_{\varrho_{1}}(x_{0})})^{-1}(t)\le(H^{-}_{B_{\varrho_{2}}(x_{0})})^{-1}(t)$, \ccc{(resp. $(H^{+}_{B_{\varrho_{1}}(x_{0})})^{-1}(t)\ge(H^{+}_{B_{\varrho_{2}}(x_{0})})^{-1}(t)$),} holds too for all $t\ge 0$.
}
\end{remark}
We start with a classical lemma (see \cite{BCM1} for a description and references), which is concerned with some properties of Maximal operators with respect to the so called gradient truncation. We recall that the Hardy-Littlewood maximal operator is defined as follows
$$
\mathcal M(f)(x) := \sup_{B_{\varrho}(x) \subset \er^n}\, \mint_{B_{\varrho}(x)} |f(y)|\, dy\;,\quad x \in \er^n\;,
$$
whenever $f \in L^1_{\loc}(\er^n)$.
\begin{lemma}\label{liptr}
Let $\cd{B_{\rr}}\subset \mathbb{R}^{n}$ be a ball and $w\in W^{1,1}_{0}(B_{\varrho},\RN)$ (trivially extended by zero outside $B_{\varrho}$). Then for any $\lambda>0$ there exists $w_{\lambda}\in W^{1,\infty}_{0}(B_{\varrho},\RN)$ such that
\begin{flalign}\label{liptr1}
\nr{Dw_{\lambda}}_{L^{\infty}(B_{\varrho},\mathbb{R}^{N\times n})}\le c\lambda\;,
\end{flalign}
for some positive constant $c\equiv c(n,N)$. Moreover, it holds that
\begin{flalign}\label{liptr2}
B_{\varrho}\cap\{w\not =w_{\lambda}\}\subset \left(B_{\varrho}\cap \{\mathcal M(\snr{Dw})>\lambda\} \right)\cup \ \mbox{negligible set}\;.
\end{flalign}
\end{lemma}
We have a first quantitative harmonic approximation type lemma. 
\begin{lemma}[Intrinsic and quantitative $g(x,\cdot)$-harmonic approximation]\label{int}
Let $B_{r}\subset \mathbb{R}^{n}$ be a ball with radius $r\le 1$ and such that \cc{$B_{2r} \Subset \Omega$}, \ccc{$\varepsilon \in (0,1)$} and $v\in (W^{1,H}\cap L^{\infty})(B_{r},\RN)$, $N\geq 1$, be a function satisfying
\begin{flalign}\label{har1}
\mint_{B_{r}}H(x,Dv) \, dx \le c_{1}H^{-}_{B_{r}}\left(\frac{\varepsilon}{r}\right)\;,
\end{flalign}
\eqn{maggiorefa}
$$
\left(\mint_{B_{r/2}}[H(x,Dv)]^{1+\delta} \, dx\right)^{1/(1+\delta)}\le \tilde c_1\mint_{B_{r}}H(x, Dv) \, dx\;,
$$
and 
\begin{flalign}\label{har2}
\left |\ \mint_{B_{r/2}}\partial g(x,Dv)\cdot D\varphi \, dx \  \right |\le c_{2}\varepsilon^{t} \mint_{B_{r}}\left[H(x,Dv)+H\left(x,\nr{D\varphi}_{L^{\infty}(B_{r/2})}\right)\right] \, dx\;,
\end{flalign}
for some $t\in(0,1]$, $\delta \in (0,1)$ and all $\varphi \in C^{\infty}_{c}(B_{r/2},\RN)$, where $g\colon \Omega \times \ernN \to [0, \infty)$ is of the type in \trif{lagg} under assumptions $\cd{\eqref{assF}_{1,2,3,4}}$, and where $c_{1}, \tilde c_1$ and $c_{2}$ are fixed constants larger than one. Then there exists $h \in W^{1,H}_{v}(B_{r/2},\RN)$ such that
\begin{flalign}\label{hh}
\int_{B_{r/2}}\partial g(x,Dh)\cdot D\varphi \, dx =0 \quad \mbox{for \ all \ } \varphi \in W^{1,H}_0(B_{r/2}, \er^N)\;,
\end{flalign}
\eqn{massimo}
$$
\|h\|_{L^{\infty}(B_{r/2})} \leq \sqrt{N}\|v\|_{L^{\infty}(B_{r/2})}
$$
and 
\begin{flalign}\label{hhh}
\mint_{B_{r/2}}\left(\snr{V_{p}(Dv)-V_{p}(Dh)}^{2}+a(x)\snr{V_{q}(Dv)-V_{q}(Dh)}^{2}\right) \, dx \le c\varepsilon^{m}\mint_{B_{r}}H(x,Dv) \, dx\;,
\end{flalign}
with $c\equiv c(\data_0)$ and $m=m\left(\data, \vv, t, \delta\right)$ (see \trif{dataref2} below for the meaning of $\data_0$). Finally, the function \cc{$h \in W^{1,H}_{v}(B_{r/2},\RN)$} is the unique solution of the Dirichlet problem
\begin{flalign}\label{bdp}
h \mapsto \min_{w\in W^{1,H}_{v}(B_{r/2},\RN)} \int_{B_{r/2}}g(x,Dw) \, dx\;.
\end{flalign}
\end{lemma}
\begin{remark} \emph{The assumptions considered in Lemma \ref{int} are tailored to the situations where the Lemma will be applied. In the typical applications, $v$ is a minimizer of a constrained problem as considered in Theorem \ref{main1}. This means that the condition $v\in (W^{1,H}\cap L^{\infty})$ is automatically satisfied. For the same reason, assumption \rif{maggiorefa} is satisfied by Lemma \ref{C1}. Finally, the smallness condition \rif{har1} typically occurs when proving partial regularity theorems (see next section). We also wish to point out that the proof we are going to give here allows for further generalizations to cases where instead of the function $H(\cdot)$ one considers more general instances, as for example those in Section \ref{weight}.} 
\end{remark}
\begin{proof} In the following we shall abbreviate, as in \rif{dataref}, as follows
\eqn{dataref2} 
$$\cc{\data_0 \equiv 	\data_0 \left(n,N,\nu,L,p,q, \alpha,[a]_{0,\alpha}, \|v\|_{L^\infty(B_{r})}, c_1, \tilde c_1, c_2\right)\;.}$$
By a standard approximation argument we notice that, if \eqref{har2} holds for every $\varphi \in C^{\infty}_{c}(B_{r/2},\RN)$, then it also holds for every $\varphi \in W^{1,\infty}_{0}(B_{r/2},\RN)$. Now, let $h \in W^{1,H}_{v}(B_{r/2},\RN)$ be the unique solution to the Dirichlet problem \rif{bdp}. This can be obtained as follows. First, notice that solutions are always unique as a consequence of the strict convexity of $z \mapsto g(\cdot, z)$. Then, existence of $h \in W^{1,p}_{v}(B_{r/2},\RN)$ \cd{results} by Direct Methods of the Calculus of Variations. \cd{By minimality and $\eqref{assF}_{1}$, there holds}
\begin{flalign}\label{har5}
\mint_{B_{r/2}}H(x, Dh) \, dx \le \frac{L}{\nu}\mint_{B_{r/2}}H(x,Dv) \, dx\;,
\end{flalign}
\cd{so $h$ satisfies \eqref{hh}.} Moreover, thanks to the assumptions in \rif{lagg} and $\cd{\rif{assF2}_{1}}$, we can apply the maximum principle in \cite[Theorem 2.3]{leosie} and this yields
\rif{massimo}. In particular, we conclude with $h \in W^{1,H}_{v}(B_{r/2},\RN)$. 
Thanks to Remark \ref{reun} and \rif{maggiorefa}, we can also apply Lemma \ref{C2} (with $v$ replaced by $h$ and $u$ replaced by $v$ in the present situation); this, together with \rif{har5}, yield that
\begin{eqnarray}
\notag \left(\mint_{B_{r/2}}[H(x,Dh)]^{1+\sigma_{g}} \, dx\right)^{\frac{1}{1+\sigma_{g}}} &\le  &c\left(\mint_{B_{r/2}}[H(x,Dv)]^{1+\sigma_{g}} \, dx\right)^{\frac{1}{1+\sigma_{g}}}\\
&\leq& c\left(\mint_{B_{r/2}}[H(x,Dv)]^{1+\delta} \, dx\right)^{\frac{1}{1+\delta}}\stackleq{maggiorefa} c\mint_{B_{r}}H(x,Dv) \, dx\;,\label{har6}
\end{eqnarray}
for positive constants $c \equiv c (\data, c_1)$ and $\sigma_g\equiv  \sigma_g\left(\data, \vv\right)$, with $c\geq 1$ and $\sigma_g \in (0,\delta)$. This peculiar dependence of the constants is also a consequence of \rif{massimo} (see again Remark \ref{reun}). In the application of Lemma \ref{C2} we are indeed getting rid of the dependence on $\|h\|_{L^\infty}$ by means of \rif{massimo}. Now, notice that there is no loss of generality in assuming that
$\int_{B_{r}}H(x,Dv) \, dx >0$, otherwise $v\equiv \const$ on $B_{r}$ and the thesis trivially holds for $h\equiv \const$. From Remark \ref{h-}, we have that $t\mapsto H^{-}_{B_{r}}(t)$ is a bijection, so there is a unique $\lambda>0$ such that
\begin{flalign}\label{har3}
\cc{H^{-}_{B_{r}}(\lambda)=M\mint_{B_{r}}H(x,Dv) \, dx\;,}
\end{flalign}
holds for some \cc{$M\ge 1$} whose size will be fixed later. Set $w=v-h \in W^{1,H}_{0}(B_{r/2},\RN)$ and consider $w_{\lambda}\in W^{1,\infty}_{0}(B_{r/2},\RN)$ given by Lemma \ref{liptr}, which satisfies \eqref{liptr1} and \eqref{liptr2}. We deduce that
\begin{eqnarray}\label{har7}
\frac{\snr{B_{r/2}\cap\{w\not =w_{\lambda}\}}}{\snr{B_{r/2}}}&\stackleq{liptr2} & \frac{\snr{B_{r/2}\cap \{\mathcal M(\snr{Dw})>\lambda\}}}{\snr{B_{r/2}}}\nonumber \\
&\stackrel{\mbox{Chebyshev}}{\leq} &\frac{c}{[H^{-}_{B_{r}}(\lambda)]^{1+\sigma_{g}}} \mint_{B_{r/2}}[H^{-}_{B_{r}}(\mathcal M(\snr{Dw}))]^{1+\sigma_{g}}\, dx\nonumber \\
&\stackrel{\mbox{maximal}}{\leq} &\frac{c}{[H^{-}_{B_{r}}(\lambda)]^{1+\sigma_{g}}}\mint_{B_{r/2}}[H^{-}_{B_{r}}(Dw)]^{1+\sigma_{g}} \, dx\nonumber\\
&\le &\frac{c}{[H^{-}_{B_{r}}(\lambda)]^{1+\sigma_{g}}}\mint_{B_{r/2}}[H^{-}_{B_{r}}(Dh)]^{1+\sigma_{g}}+[H^{-}_{B_{r}}(Dv)]^{1+\sigma_{g}} \, dx\nonumber \\
&\stackleq{har6} &\frac{c}{[H^{-}_{B_{r}}(\lambda)]^{1+\sigma_{g}}}\left(\mint_{B_{r}}H(x,Dv) \, dx\right)^{1+\sigma_{g}}\stackleq{har3} \cc{\frac{c}{M^{1+\sigma_{g}}}}\;,
\end{eqnarray}where $c\equiv c(\data_0)$. Now we test the weak formulation of \eqref{bdp} against $w_{\lambda}$ to get
\begin{flalign}\label{tt22}
\notag \mathcal{T}_{1}:=&\mint_{B_{r/2}}\left(\partial g(x,Dv)-\partial g(x,Dh)\right)\cdot Dw_{\lambda}\chi_{\{w=w_{\lambda}\}} \, dx\\
=&\mint_{B_{r/2}}\partial g(x,Dv)\cdot Dw_{\lambda}\, dx -\mint_{B_{r/2}}\left(\partial g(x,Dv)-\partial g(x,Dh)\right)\cdot Dw_{\lambda}\chi_{\{w\not=w_{\lambda}\}} \, dx=:\mathcal{T}_{2}+\mathcal{T}_{3}\;.
\end{flalign}
Upon setting (recall the definition in \rif{vpvq})
\eqn{defiV}
$$\mathcal{V}^{2}:=\snr{V_{p}(Dv)-V_{p}(Dh)}^{2}+a(x)\snr{V_{q}(Dv)-V_{q}(Dh)}^{2}\;,$$ the strict monotonicity \eqref{042} implies there exists a constant \cc{$c\equiv c(n,N,\nu,p,q)$} such that
$$
\mathcal{T}_{1}\ge \frac{1}{c}\mint_{B_{r/2}}\mathcal{V}^{2}\chi_{\{w=w_{\lambda}\}} \, dx\;.
$$
Let us consider term $\mathcal{T}_{2}$; for this we start observing 
\begin{flalign*}
\lambda\stackrel{\rif{har3}}{=}\cc{\left(H^{-}_{B_{r}}\right)^{-1}\left(M\mint_{B_{r}}H(x,Dv) \, dx \right)}\stackrel{\rif{homh},\rif{har1}}{\leq} \cc{2c_{1}M\left(H^{-}_{B_{r}}\right)^{-1}}\left(H^{-}_{B_{r}}\left(\frac{\varepsilon}{r}\right)\right)\le \cc{\frac{2c_1\varepsilon M}{r}\;.}
\end{flalign*}
From this last inequality and \cc{$\eqref{mainbound}_{1}$} we can estimate
\cc{
\begin{eqnarray}\label{har9}
\notag H^{+}_{B_{r}}(\lambda)&=&H^{-}_{B_{r}}(\lambda)+\left[\as(B_{r})-\ai(B_{r})\right]\lambda^{q}\nonumber \\
&\le & H^{-}_{B_{r}}(\lambda)+cr^{\alpha-(q-p)}(\varepsilon M)^{q-p}\lambda^{p}\le c\left[1+(\varepsilon M)^{q-p}\right]H^{-}_{B_{r}}(\lambda)\;,
\end{eqnarray}}with $c\equiv c(p,q,\alpha, [a]_{0,\alpha},c_{1})$. Now we have 
\eqn{har10}
$$
\cc{\mint_{B_{r}}H\left(x,\nr{Dw_{\lambda}}_{L^{\infty}(B_{r/2})}\right) \, dx \stackleq{liptr1} c\mint_{B_{r}}H(x,\lambda) \, dx \le  cH^{+}_{B_{r}}(\lambda)\stackleq{har9} c\left[1+(\varepsilon M)^{q-p}\right]H^{-}_{B_{r}}(\lambda)\;,}
$$
where $c\equiv c(n,N,p,q,\alpha, [a]_{0,\alpha}, c_{1})$, so that
\cc{
\begin{eqnarray*}
\notag \snr{\mathcal{T}_{2}} &\stackrel{\rif{har2} }{\leq} & c_{2}\varepsilon^{t} \mint_{B_{r}}\left[H(x,Dv)+H\left(x,\nr{Dw_\lambda}_{L^{\infty}(B_{r/2})}\right)\right] \, dx
\\
&\stackrel{\rif{har10}}{\leq} &c\eps^t\mint_{B_{r}}H(x,Dv) \, dx+ c \varepsilon^{t}
\left[1+(\varepsilon M)^{q-p}\right]H^{-}_{B_{r}}(\lambda)\notag \\ &\stackleq{har3} &c \varepsilon^{t}M
\left[1+(\varepsilon M)^{q-p}\right] \mint_{B_{r}}H(x,Dv) \, dx\;,
\end{eqnarray*}}
for $c\equiv c(n,N,p,q,\alpha, [a]_{0,\alpha},c_{1},c_{2})$. Finally, for $\mathcal{T}_{3}$, we fix $\kappa \in (0,1)$ to be chosen later on and estimate as follows:
\begin{eqnarray*}%\label{har12}
\snr{\mathcal{T}_{3}}&\stackleq{tt22} &\mint_{B_{r/2}}\left(\snr{\partial g(x,Dh)}+\snr{\partial g(x,Dv)}\right)\snr{Dw_{\lambda}}\chi_{\{w\not =w_{\lambda}\}} \, dx\nonumber \\
&\stackleq{assF} &c\mint_{B_{r/2}}\left(\frac{H(x,Dh)}{\snr{Dh}}+\frac{H(x,Dv)}{\snr{Dv}}\right)\snr{Dw_{\lambda}}\chi_{\{w\not =w_{\lambda}\}} \, dx\nonumber \\
&\stackrel{\mbox{Young}}{\leq} & \kappa \mint_{B_{r/2}}\left[H(x,Dh) +H(x,Dv) \right]\, dx +\frac{c}{\kappa^{q-1}}\mint_{B_{r/2}}H(x,Dw_{\lambda})\chi_{\{w\not =w_{\lambda}\}} \, dx\nonumber \\
&\stackleq{har5} & c\kappa \mint_{B_{r/2}}H(x,Dv) \, dx +\frac{c}{\kappa^{q-1}}\mint_{B_{r/2}}H(x,Dw_{\lambda})\chi_{\{w\not =w_{\lambda}\}} \, dx\nonumber \\
&\stackleq{liptr1} & c\kappa \mint_{B_{r}}H(x,Dv) \, dx +\frac{c}{\kappa^{q-1}}\frac{\snr{B_{r/2}\cap\{w\not=w_{\lambda}\}}}{\snr{B_{r/2}}}H^{+}_{B_{r}}(\lambda)\nonumber \\
&\stackleq{har7} & \cc{c\kappa \mint_{B_{r}}H(x,Dv) \, dx +\frac{c}{\kappa^{q-1}M^{1+\sigma_{g}}}H^{+}_{B_{r}}(\lambda)}\nonumber \\
&\stackleq{har9} &\cc{c\kappa\mint_{B_{r}}H(x,Dv) \, dx +\frac{c}{\kappa^{q-1}M^{1+\sigma_{g}}}\left[1+(\varepsilon M)^{q-p}\right]H^{-}_{B_{r}}(\lambda)}\nonumber \\
&\stackleq{har3} &\cc{c\left\{\kappa+\frac{1}{\kappa^{q-1}M^{\sigma_{g}}}\left[1+(\varepsilon M)^{q-p}\right]\right\}\mint_{B_{r}}H(x,Dv) \, dx}\;,
\end{eqnarray*}
with $c\equiv c(\data_0)$. Collecting the estimates found for $\mathcal{T}_{1}, \mathcal{T}_{2}$ and $\mathcal{T}_{3}$ to \rif{tt22}, we get
\begin{eqnarray}\label{har13}
\notag \mint_{B_{r/2}}\mathcal{V}^{2}\chi_{\{w=w_{\lambda}\}} \, dx &\le& \cc{c\left\{\kappa+\varepsilon^{t}M+\varepsilon^{t+q-p}M^{q-p+1}+\frac{1}{\kappa^{q-1}M^{\sigma_{g}}} +\frac{(\varepsilon M)^{q-p}}{\kappa^{q-1}M^{\sigma_{g}}}\right\}\mint_{B_{r}}H(x,Dv) \, dx}\nonumber\\
&=:&\cc{cS(\kappa, \varepsilon,M)\mint_{B_{r}}H(x,Dv) \, dx}\;,
\end{eqnarray}
again with $c\equiv c(\data_0)$. Now let $\theta \in (0,1)$ be a number to be fixed in some lines. \cc{From H\"older's inequality, \eqref{har7} and \eqref{har5} we obtain}
$$
\left(\mint_{B_{r/2}}\mathcal{V}^{2\theta}\chi_{\{w\not =w_{\lambda}\}} \, dx\right)^{1/\theta}\le \left(\frac{\snr{B_{r/2}\cap\{w\not=w_{\lambda}\}}}{\snr{B_{r/2}}}\right)^{\frac{1-\theta}{\theta}}\mint_{B_{r/2}}\mathcal{V}^{2} \, dx\le \cc{\frac{c}{M^{(1+\sigma_{g})\frac{1-\theta}{\theta}}}\mint_{B_{r}}H(x,Dv)\, dx}
$$
for $c\equiv c(\data_0)$ and, again by H\"older's inequality and \eqref{har13},
$$
\cc{\left(\mint_{B_{r/2}}\mathcal{V}^{2\theta}\chi_{\{w =w_{\lambda}\}} \, dx\right)^{1/\theta}\le cS(\kappa,\varepsilon,M)\mint_{B_{r}}H(x,Dv) \, dx\;,}
$$
with $c\equiv c(\data_0)$. Merging the content of the last two displays now gives
\eqn{har16}
$$
\cc{\left(\mint_{B_{r/2}}\mathcal{V}^{2\theta}\, dx\right)^{1/\theta}\le c\left\{S(\kappa,\varepsilon,M)+\frac{1}{M^{(1+\sigma_{g})\frac{1-\theta}{\theta}}}\right\}\mint_{B_{r}}H(x,Dv) \, dx\;}.
$$
In the above inequality $\eps$ is fixed in the statement of the theorem, while $\kappa \in (0,1)$ and \cc{$M\geq 1$} are still free parameters to be chosen arbitrarily. We take 
$$\cc{M=\frac{1}{\varepsilon^{\frac{t}{2}}}> 1}\qquad \mbox{and}\qquad \kappa=\varepsilon^{\frac{\sigma_{g}t}{4(q-1)}}\in (0, 1)$$and set $$\bar{m}:=\frac{t\sigma_{g}}{4}\min\left\{1,\frac{1}{q-1}\right\},$$ so that, recalling the expression of $S(\kappa, \varepsilon,M)$ in \rif{har13}, we find
$$
\cc{S(\kappa, \varepsilon,M)+\frac{1}{M^{(1+\sigma_{g})\frac{1-\theta}{\theta}}} \leq 5 \eps^{\bar{m}}+ \eps^{\frac{t(1+\sigma_{g})(1-\theta)}{2\theta}} \leq 6 \eps^{\tilde m(\theta)}}
$$
where
%$$\tilde{m}(\theta):=\min\left\{ \frac{\sigma_{g}t}{4},\frac{t}{2},\frac{\sigma_{g}t}{4(q-1)},q-p, \frac{t(1+\sigma_{g})(1-\theta)}{2\theta} \right\}\equiv m(\sigma_g, t, q,p, \theta)\;,$$
\eqn{dipendenzam}
$$\tilde{m}(\theta):=\min\left\{ \cd{\bar{m}}, \frac{t(1+\sigma_{g})(1-\theta)}{2\theta} \right\}\equiv \tilde m(\sigma_g, t, q, \theta)\equiv \tilde m(\data, \vv, t,\delta, \theta)\;,$$
and therefore \eqref{har16} reads as
\begin{flalign}\label{har17}
\left(\mint_{B_{r/2}}\mathcal{V}^{2\theta} \, dx\right)^{\frac{1}{2\theta}}\le c\varepsilon^{\frac{\tilde{m}(\theta)}{2}}\left(\mint_{B_{r}}H(x,Dv) \, dx\right)^{\frac12}
\end{flalign}
with $c\equiv c(\data_0)$. The final dependence on the various constants of $m$ in \rif{dipendenzam} has been obtained recalling that $\sigma_g\equiv  \sigma_g\left(\data, \vv\right)$; notice also that the dependence upon the initial higher integrability exponent $\delta$ appearing in \rif{maggiorefa} comes from the restriction \cc{$\sigma_g< \delta$}. Next, notice that from the very definition of $\mathcal{V}$ in \rif{defiV}, and using  \rif{har6}, we readily infer
\eqn{har177}
$$
\left(\mint_{B_{r/2}}\mathcal{V}^{2(1+\sigma_{g})} \, dx\right)^{\frac{1}{2(1+\sigma_{g})}}\leq c\left(\mint_{B_{r}}H(x,Dv) \, dx\right)^{\frac12}\;,
$$
again for $c\equiv c(\data_0)$. Next, we choose 
\eqn{sceltatheta}
$$
\theta:=\frac{1+\sigma_{g}}{(1+2\sigma_{g})} \equiv \theta\left(\data, \vv, \delta\right)
$$
and apply H\"older's inequality with conjugate exponents $\frac{2(1+\sigma_{g})}{1+2\sigma_{g}}$ and $2(1+\sigma_{g})$, to get
\begin{flalign*}
\mint_{B_{r/2}}\mathcal{V}^{2} \, dx \le \left(\mint_{B_{r/2}}\mathcal{V}^{\frac{2(1+\sigma_{g})}{1+2\sigma_{g}}}\, dx\right)^{\frac{1+2\sigma_{g}}{2(1+\sigma_{g})}}\left(\mint_{B_{r/2}}\mathcal{V}^{2(1+\sigma_{g})} \, dx\right)^{\frac{1}{2(1+\sigma_{g})}}\stackrel{\rif{har17}, \rif{har177}}{\leq} c\varepsilon^{\frac{\tilde{m}(\theta)}{2}}\mint_{B_{r}}H(x,Dv) \, dx\;,
\end{flalign*}
with where $c\equiv c(\data_0)$. This concludes the proof of \rif{hhh}, and of Lemma \ref{int}, by fixing $m:= \tilde{m}(\theta)/2$ and the dependence $m=m\left(\data, \vv, t, \delta\right)$ claimed in the statement follows by looking at \rif{dipendenzam} and \rif{sceltatheta}. 
\end{proof}
We next report another harmonic type approximation lemma of the type already considered in \cite{BCM3}. On the contrary to Lemma \ref{int}, this one involves a classical Young function $H_0(\cdot)$, i.e., no dependence on $x$ is considered
 \eqn{H0} 
$$
H_0(t)=t^{p}+a_0t^{q}\;, \qquad a_0\geq0\;.
$$
This time we shall consider a $C^1(\ernN)\cap C^2(\ernN\setminus\{0\})$-regular integrand 
$g_{0}\colon \ernN \to [0, \infty)$ such that
\eqn{assF22}
$$
g_0(z)= \tilde g_0(|z|)\ \mbox{holds for every $z \in \ernN$} \ \mathrm{with} \ t \mapsto \tilde g_0(t) \ \mbox{non-decreasing}
$$
where $\tilde g_0\colon [0, \infty) \to [0, \infty)$ of class $C^1[0, \infty)\cap C^2(0, \infty)$. We shall consider the following set of 
assumptions:
\begin{flalign}\label{assg0}
\begin{cases}
\ \nu H_{0}(z)\le g_{0}(z)\le LH_{0}(z)\\
\ \snr{\partial g_{0}(z)}\snr{z}+\snr{\partial^{2}g_{0}(z)}\snr{z}^{2}\le LH_{0}(z)\\
\ \nu (\snr{z}^{p-2}+a_{0}\snr{z}^{q-2})\snr{\xi}^{2}\le \langle \partial^{2}g_{0}(z)\xi,\xi\rangle\;,
\end{cases}
\end{flalign}
considered with the same notation as in \trif{assF}, for suitable numbers $0 < \nu \leq  1  \leq L < +\infty$ (not necessarily the same as appearing in \rif{assF}).  
We then have the following approximation lemma, which is a different version of \cite[Lemma 1]{BCM3}:
\begin{lemma}[Quantitative $g_{0}$-harmonic approximation]\label{harg}
Let $B_{r}\subset \mathbb{R}^{n}$ be a ball with radius $r\le 1$, $\varepsilon \in (0,1]$ and $v\in W^{1,H_{0}}(B_{r},\RN)$, $N\geq 1$, be a function satisfying 
\eqn{maggiorefadopo}
$$
\left(\mint_{B_{r/2}}[H_0(Dv)]^{1+\delta} \, dx\right)^{1/(1+\delta)}\le \tilde c_1\mint_{B_{r}}H_{0}(Dv) \, dx\;,
$$
for some $\delta \in (0,1)$ and
\begin{flalign}\label{harg2}
\left |\ \mint_{B_{r/2}}\partial g_{0}(Dv)\cdot D\varphi \, dx \  \right |\le c_{2}\varepsilon^{t} \mint_{B_{r}}\left[H_{0}(Dv)+H_{0}\left(\nr{D\varphi}_{L^{\infty}(B_{r/2})}\right)\right] \, dx\;,
\end{flalign}
where $t \in (0,1]$ and all $\varphi \in C^{\infty}_{c}(B_{r/2},\RN)$, where $\tilde c_{1}$ and $c_{2}$ are absolute constants and under assumptions \trif{assF22}-\trif{assg0}. Then there exists $h_{0} \in W^{1,H_{0}}_{v}(B_{r/2},\RN)$ such that
\begin{flalign}\label{hhg}
\int_{B_{r/2}} \partial g_{0}(Dh_{0})\cdot D\varphi \, dx =0 \quad \mbox{for \ all \ }\varphi \in W^{1,H_{0}}_0(B_{r/2},\RN)\;,
\end{flalign}
\eqn{massimo00}
$$
\|h_0\|_{L^{\infty}(B_{r/2})} \leq \sqrt{N}\|v\|_{L^{\infty}(B_{r/2})}
$$
and 
$$
\mint_{B_{r/2}}\left(\snr{V_{p}(Dv)-V_{p}(Dh_{0})}^{2}+a_{0}\snr{V_{q}(Dv)-V_{q}(Dh_{0})}^{2} \right)\, dx \le c\varepsilon^{m}\mint_{B_{r}}H_{0}(Dv) \, dx\;,
$$
\cc{with $c\equiv c(n,N,\nu,L,p,q,\tilde c_1, c_2)$ and $m=m(n,N,\nu,L,p,q, t, \delta)$}. Finally, the function \cc{$h_0 \in W^{1,H_0}_{v}(B_{r/2},\RN)$} is the unique solution of the Dirichlet problem
\begin{flalign}\label{bdpg}
h \mapsto \min_{w\in W^{1,H_{0}}_{v}(B_{r/2},\RN)}\int_{B_{r/2}}g_{0}(Dh_{0}) \, dx\;.
\end{flalign}
\end{lemma}
\begin{proof} The proof is rather close to that of Lemma \ref{harg}. For this reason we only give a sketch of it. Again, in \eqref{harg2} we can consider $\varphi \in W^{1,\infty}_{0}(B_{r/2},\RN)$. This time we define $h_{0}\in W^{1,H_{0}}(B_{r/2},\RN)$ as the unique solution of the Dirichlet problem
\rif{bdpg}, so that \rif{hhg}-\rif{massimo00} hold. Moreover, $\eqref{assg0}_{1}$ and minimality yield
\begin{flalign}\label{harg5}
\mint_{B_{r/2}}H_{0}(Dh_{0}) \, dx \le \frac{L}{\nu} \mint_{B_{r/2}}H_{0}(Dv) \, dx\;.
\end{flalign}
By Lemma \ref{C2} (with constant coefficients, see Remark \ref{reun2}), we get, as for \rif{har6} and using \rif{harg5}, that 
\begin{flalign}\label{harg6}
\cd{\left(\mint_{B_{r/2}}[H_0(Dh_0)]^{1+\sigma_{g}} \, dx\right)^{\frac{1}{1+\sigma_{g}}}\le c\left(\mint_{B_{r/2}}[H_0(Dv)]^{1+\sigma_{g}} \, dx\right)^{\frac{1}{1+\sigma_{g}}}\stackleq{maggiorefadopo} c\mint_{B_{r}}H_0(Dv) \, dx\;,}
\end{flalign}
holds for \cc{$\sigma_{g}\equiv \sigma_g(n,N,\nu,L,p,q)\in (0,\delta)$} and with \cc{$c\equiv c(n,N,\nu,L,p,q, \tilde c_1)$}. Proceeding as for the proof of Lemma \ref{int}, we find $\lambda>0$ such that
\begin{flalign}\label{harg3}
\cc{H_{0}(\lambda)=M\mint_{B_{r}}H_{0}(Dv) \, dx\;,}
\end{flalign}
for some \cc{$M\ge 1$} to be specified later on. Set $w=v-h_{0}\in W^{1,H_{0}}_{0}(B_{r/2},\mathbb{R}^{N})$ and consider $w_{\lambda}$ given by Lemma \ref{liptr} matching \eqref{liptr1}-\eqref{liptr2}. As for the proof of \rif{har7}, but using \eqref{harg6} and \eqref{harg3}, we deduce that
\eqn{harg7}
$$
\frac{\snr{B_{r/2}\cap \{w \not = w_{\lambda}\}}}{\snr{B_{r/2}}}\le  \frac{c}{[H_{0}(\lambda)]^{1+\sigma_{g}}}\left(\mint_{B_{r}}H_{0}(Dv) \, dx\right)^{1+\sigma_{g}}\stackleq{harg3} \cc{\frac{c}{M^{1+\sigma_{g}}}}\;,
$$
with \cc{$c\equiv c(n,N,\nu,L,p,q, \tilde c_1)$}. Now we test \rif{hhg} against $w_{\lambda}$ and set
\begin{flalign*}
\mathcal{T}_{1}:=&\mint_{B_{r/2}}\left(\partial g_{0}(Dv)-\partial g_{0}(Dh_{0})\right)\cdot Dw_{\lambda}\chi_{\{w=w_{\lambda}\}} \, dx\\
=&\mint_{B_{r/2}}\partial g_{0}(Dv)\cdot Dw_{\lambda} \, dx -\mint_{B_{r/2}}\left(\partial g_{0}(Dv)-\partial g_{0}(Dh_{0})\right)\cdot Dw_{\lambda}\chi_{\{w\not = w_{\lambda}\}} \, dx=:\mathcal{T}_{2}+\mathcal{T}_{3}\;.
\end{flalign*} 
This time, as in \rif{defiV}, we set $\mathcal{V}_{0}^{2}:=\snr{V_{p}(Dv)-V_{p}(Dh_{0})}^{2}+a_{0}\snr{V_{q}(Dv)-V_{q}(Dh_{0})}^{2}$. By monotonicity of $\partial g_0(\cdot)$ (which is similar to \rif{monotonicity} for $a(x)\equiv a_0$), there is \cc{$c\equiv c(n,N,\nu,p,q)$} such that
\begin{flalign}\label{harg8}
\mathcal{T}_{1}\ge \frac{1}{c}\mint_{B_{r/2}}\mathcal{V}_{0}^{2}\chi_{\{ w=w_{\lambda} \}} \, dx\;.
\end{flalign}
As for $\mathcal{T}_{2}$, from \eqref{harg2}, \eqref{liptr1} and \eqref{harg3}, we obtain
\begin{flalign}\label{harg9}
\cc{\snr{\mathcal{T}_{2}}\le c\varepsilon^{t}M\mint_{B_{r}}H_{0}(Dv) \, dx\;,}
\end{flalign}
where \cc{$c\equiv c(n,N,L,p,q,c_{1})$}. Finally, for $\mathcal{T}_{3}$, we fix $\kappa \in (0,1)$ to be chosen. Then, by using $\eqref{assg0}_{2}$, Young's inequality, \eqref{harg5}, \eqref{harg7} and \eqref{liptr1}, and proceeding as in the proof of the analogous term $\mathcal T_3$ from Lemma \ref{int}, we have 
\begin{eqnarray}\label{harg12}
\snr{\mathcal{T}_{3}} 
 &\le &c\mint_{B_{r/2}}\left(\frac{H_{0}(Dh_{0})}{\snr{Dh_{0}}}+\frac{H_{0}(Dv)}{\snr{Dv}}\right)\snr{Dw_{\lambda}}\chi_{\{w\not = w_{\lambda}\}} \, dx\nonumber \\
&\le &c\kappa \mint_{B_{r/2}}H_{0}(Dv) \, dx +\frac{c}{\kappa^{q-1}}\mint_{B_{r/2}}H_{0}(Dw_{\lambda})\chi_{\{w\not =w_{\lambda}\}} \, dx\nonumber \\ &\le&  \cc{c\left(\kappa+\frac{1}{\kappa^{q-1}M^{\sigma_{g}}}\right) \mint_{B_{r}}H_{0}(Dv) \, dx}\;,
\end{eqnarray}with \cc{$c\equiv c(n,N,L,p,q,\tilde c_1)$}. From estimates \eqref{harg8}-\eqref{harg12} we get
$$
\cc{\mint_{B_{r/2}}\mathcal{V}^{2}_{0}(x) \chi_{\{w=w_{\lambda}\}}\le c\left\{\kappa+ \varepsilon^{t}M+\frac{1}{\kappa^{q-1}M^{\sigma_{g}}} \right\}\mint_{B_{r}}H_{0}(Dv) \, dx}\;,%:=cS(\kappa,\varepsilon,M)\mint_{B_{r}}H_{0}(Dv) \, dx\;,
$$
with \cc{$c\equiv c(n,N,\nu,L,p,q,\tilde c_1, c_2)$}. Starting from the last inequality, the rest of the proof goes exactly as the one for Lemma \ref{int}, after \rif{har13}. 
\end{proof}
Finally, an elementary Young type inequality.
\begin{lemma}\label{elementare} Let $H_0(\cdot)$ be the function defined in \trif{H0} . Then, whenever $\kappa\in (0,1)$ it holds that 
\eqn{Youngh}
$$
st \leq \kappa 
H_0(t) + \kappa^{-1/(p-1)}H_0^*(t)\,, \qquad \ccc{\mbox{for all}} \ \ s, t \geq 0\;.
$$ 
where
$$
H_0^*(t):=\sup_{s>0}\, \left(st-H_0(s)\right)\,, \qquad \ccc{\mbox{for all}} \ \  t \geq 0
$$
denotes the convex conjugate function to $H_0(\cdot)$. 
\end{lemma}
\begin{proof}
Notice that, for \cc{$A\geq 1$}, as it is $q\geq p$, we have 
\begin{eqnarray*}
\cc{H_0^*(At)}&=& \sup_{s>0}\, \cc{\left(sAt-H_0(s)\right) = A^{\frac{p}{p-1}} 
\sup_{s>0}\, \left(sA^{-\frac{1}{p-1}} t-A^{-\frac{p}{p-1}}H_0(s)\right)}\\
&\leq & \cc{A^{\frac{p}{p-1}}\sup_{s>0}\, \left(sA^{-\frac{1}{p-1}} t-H_0(sA^{-\frac{1}{p-1}})\right)=
A^{\frac{p}{p-1}}H_0^*(t)}\;.
\end{eqnarray*}
Therefore, we find, for $\kappa \in (0,1)$
$$
st \leq H_0\left(\kappa^{1/p} t\right)+H_0^*\left(t/\kappa^{1/p}\right) \leq \kappa 
H_0(t) + \kappa^{-1/(p-1)}H_0^*(t)
$$
that is, \rif{Youngh}. 
\end{proof}

\section{Proof of Theorem \ref{main1}}\label{mainproof}
 In the following, $u \in W^{1,H}_{\rm{loc}}(\Omega,\SN)$ \cc{is as in the statement of Theorem \ref{main1}}. We start recalling Lemma \ref{C1}, according to which there exists $\delta_g\equiv \delta_g(\data)>0$ such that 
$H(\cdot, Du)\in L^{1+\delta_g}_{\loc}(\Omega)$ holds, i.e., \rif{maggioreiniziale} is proved. 
For the proof of Theorem \ref{main1}, we first treat the case when $p(1+\delta_{g})\le n$, and then we describe how to get the result in the remaining one $p(1+\delta_{g})>n$. The proof now goes in six steps. The first three are devoted to the proof of the partial H\"older continuity of a constrained \cd{local} minimizer of \eqref{genF}; in particular, in the third step we describe the regular and the singular sets. In the fourth step we exploit this continuity to move to a single chart. Step five is devoted to show partial H\"older continuity for the gradient in the regular set. In the final step we briefly mention how to treat the case $p(1+\delta_{g})>n$.\\
\emph{Step 1: Freezing.}\label{primasec}\\
Let $B_{r}=B_{r}(x_{0})$ be any ball such that $B_{2r}\Subset \Omega$ and $r \leq 1/2$; more in general, every ball \ccc{$B$ considered in the rest of the proof will have} radius $r(B)\leq 1/2$. We assume that the smallness condition
\begin{flalign}\label{smallfull}
\mint_{B_{2r}(x_0)}H(x, Du) \, dx< H^{-}_{B_{2r}(x_0)}\left(\frac{\varepsilon}{2r}\right)\;,
\end{flalign}
holds for some $\eps \in (0,1)$ which is going to be chosen in due course of the proof.
Let $v \in W^{1,H}_{u}(\brx, \SN)$ be a solution to the frozen Dirichlet problem
$$
v \mapsto  \min_{w \in W^{1,H}_{u}(B_{r},\SN)}\int_{B_{r}}F\left(x,(u)_{B_{r}},Dw\right) \, dx\;.
$$
This functional satisfies the same growth assumptions (in particular \rif{assF}$_1$) of the original one minimized by $u$ and therefore Lemma \ref{C1} applies, giving 
\eqn{maggiore-original-v}
$$
\left(\mint_{B_{r/2}}[H(x,Dv)]^{1+\delta_{g}} \, dx\right)^{1/(1+\delta_{g})}\le \tilde c_1\mint_{B_{r}}H(x, Dv) \, dx\;,
$$
where the exponent $\delta_g\equiv \delta_g(\data)>0$ is the same one appearing in \rif{maggiore-original} and $\tilde c_1\equiv \tilde c_1(\data)$. Taking into account the content of Section \ref{eldp}, and in particular \rif{EL1}, $v$ solves the Euler-Lagrange equation
\begin{flalign}\label{ELfz}
\mint_{\brx}\partial F\left(x,(u)_{\brx},Dv\right) \cdot D\varphi \, dx=\mint_{\brx}\tilde{ F}'\left(x,(u)_{\brx},Dv\right)\snr{Dv}(v\cdot \varphi) \, dx\;,
\end{flalign}
which is valid for any $\varphi \in (W^{1,H}_{0}\cap L^{\infty})(\brx,\RN)$. Moreover, \cc{\rif{pre} becomes}
\begin{flalign}\label{042}
&\snr{V_{p}(z_{2})-V_{p}(z_{1})}^{2}+a(x)\snr{V_{q}(z_{2})-V_{q}(z_{1})}^{2}
+\partial F\left(x,(u)_{\brx},z_1\right)\cdot (z_{2}-z_{1})\nonumber \\
&\quad \quad \quad \le c\left[F\left(x,(u)_{\brx}, z_{2}\right)-F\left(x,(u)_{\brx},z_{1}\right)\right]
\end{flalign}
\cc{which holds} for any choice of $z_{1},z_{2}\in \mathbb{R}^{N\times n}$ and $x \in \Omega$, for a constant \cc{$c\equiv c(n,N,\nu,p,q)$}, 
see for instance \cite[(90)]{BCM3}. The map $w=u-v\in (W^{1,H}_{0}\cap L^{\infty})(\brx,\RN)$ is an admissible test function in \rif{ELfz}, therefore we have
\begin{flalign}\label{043}
&\mint_{\brx}\left(\snr{V_{p}(Du)-V_{p}(Dv)}^{2}+a(x)\snr{V_{q}(Du)-V_{q}(Dv)}^{2}\right) \, dx\nonumber \\
&\quad \quad \quad \stackleq{042} c\mint_{\brx}\left[F\left(x,(u)_{\brx},Du\right)-F\left(x,(u)_{\brx},Dv\right)\right] \, dx\nonumber-\mint_{\brx} \partial F\left(x,(u)_{\brx},Dv\right) \cdot\left(Du-Dv\right) \, dx\nonumber\\
&\quad \quad \quad \stackrel{\rif{ELfz}}{=} c\mint_{\brx}\left[F\left(x,(u)_{\brx},Du\right)-F\left(x,(u)_{\brx},Dv\right)\right] \, dx\nonumber -\mint_{\brx}\tilde{ F}'\left(x,(u)_{\brx},Dv\right)\snr{Dv}v\cdot(u-v) \, dx\nonumber \\
&\quad \quad \quad \ =c\mint_{\brx}\left[F\left(x,(u)_{\brx},Du\right)-F(x,u,Du) \right]\, dx+c\mint_{\brx}\left[ F(x,u,Du)-F(x,v,Dv)\right] \, dx\nonumber \\
&\qquad \quad \qquad+c \mint_{\brx}\left[F(x,v,Dv)-F\left(x,(v)_{\brx},Dv\right)\right] \, dx+c \mint_{\brx}\left[F\left(x,(v)_{\brx},Dv\right)-F\left(x,(u)_{\brx},Dv\right)\right] \, dx\nonumber \\
&\qquad \quad \qquad -\mint_{\brx}\tilde{ F}'\left(x,(u)_{\brx},Dv\right)\snr{Dv}v\cdot(u-v) \, dx\notag \\ &\quad \quad \quad \ =: \mathrm{(I)}+\mathrm{(II)}+\mathrm{(III)}+\mathrm{(IV)}+\mathrm{(V)} \;,
\end{flalign}
where \cc{$c\equiv c(n,N,\nu,p,q)$}. 
%Before estimating the five terms resulting from \rif{043}, we shall make a few remarks to stress that all the manipulations we are going to make are perfectly legal.
Before starting working on terms $\mathrm{(I)}$-$\mathrm{(V)}$ in \rif{043}, let us estimate \cd{some quantities} which will be recurrent in the forthcoming computations. 
First, notice that the minimality of $v$ and $\rif{assF}_{1}$ yield
\begin{flalign}\label{046}
\nu \mint_{\brx}H(x,Dv)\, dx\le \mint_{\brx}F\left(x,(u)_{\brx},Dv\right) \, dx \le \mint_{\brx}F\left(x,(u)_{\brx},Du\right) \, dx\le L\mint_{\brx}H(x, Du) \, dx\;.
\end{flalign}
Lemma \ref{C2} gives
\eqn{har6fre}
$$
\left(\mint_{B_{r}}[H(x,Dv)]^{1+\sigma_{g}} \, dx\right)^{1/(1+\sigma_{g})}\le c\left(\mint_{B_{r}}[H(x,Du)]^{1+\sigma_{g}} \, dx\right)^{1/(1+\sigma_{g})}\stackleq{maggiorefa} c\mint_{B_{2r}}H(x,Du) \, dx\;, 
$$
where $0<\sigma_g \equiv \sigma_g(\data)< \delta_g$. 
We are next going to use the function $H^{-}_{B_{2r}}(t):= t^p + \ai(B_{2r})t^q$. By Jensen's inequality (recall that $\omega(\cdot)$ in \rif{omegabeta} is a concave function, while $H^{-}_{B_{2r}}(\cdot)$ is convex), Remark \ref{h-} and the smallness condition \rif{smallfull}, we get
\begin{eqnarray}\label{omegau}
\mint_{B_{r}}\omega\left(\snr{u-(u)_{B_{r}}}\right) \, dx &\le &\omega\left(r\mint_{B_{r}}\left |\frac{u-(u)_{B_{r}}}{r} \right | \, dx\right)\nonumber \\
&=&\omega \left[r\left(H^{-}_{B_{2r}}\right)^{-1}\circ \left(H^{-}_{B_{2r}}\right)\left(\mint_{B_{r}}\left | \frac{u-(u)_{B_{r}}}{r}\right | \, dx\right)\right]\nonumber\\
&\le &\omega \left[r\left(H^{-}_{B_{2r}}\right)^{-1}\left(\mint_{B_{r}}H^{-}_{B_{2r}}\left(\frac{u-(u)_{B_{r}}}{r}\right) \, dx\right)\right]\nonumber \\
&\le &\omega \left[r\left(H^{-}_{B_{2r}}\right)^{-1}\left(\mint_{B_{r}}H\left(x, \frac{u-(u)_{B_{r}}}{r}\right) \, dx\right)\right]\nonumber \\
&\stackleq{poinc}& c\omega \left[r\left(H^{-}_{B_{2r}}\right)^{-1}\left(c\mint_{B_{2r}}H(x,Du) \, dx\right)\right]\notag  \\
&\stackrel{\rif{homh},\rif{smallfull}}{\leq} & c\omega\left[r\left(H^{-}_{B_{2r}}\right)^{-1}\circ H^{-}_{B_{2r}}\left(\frac{\varepsilon}{2r}\right)\right]\le c\varepsilon^{\beta}\;,
\end{eqnarray}with \ccc{$c\equiv c(\data,\beta)$}. Similarly, we have 
\begin{eqnarray}\label{omegav}
\mint_{B_{r}}\omega\left(\snr{v-(v)_{B_{r}}}\right) \, dx
&\leq &c\omega \left[r\left(H^{-}_{B_{2r}}\right)^{-1}\left(\mint_{B_{r}}H(x,Dv) \, dx\right)\right]\notag \\ &\stackleq{046}  & c\omega \left[r\left(H^{-}_{B_{2r}}\right)^{-1}\left(\mint_{B_{2r}}H(x,Du) \, dx\right)\right]\stackleq{omegau} c\varepsilon^{\beta}\;,
\end{eqnarray}with \ccc{$c\equiv c(\data,\beta)$}. In a totally similar way, in particular again using Lemma \ref{poinclemma} and repeatedly the content of Remark \ref{h-}, we get 
\begin{eqnarray}\label{omegauv}
\notag \omega\left(\snr{(u)_{B_{r}}-(v)_{B_{r}}}\right)& \le &\omega\left(\mint_{B_{r}}\left |u-v\right | \, dx\right)\\
&\leq &\omega\left[r\left(H^{-}_{B_{2r}}\right)^{-1}\circ \left(H^{-}_{B_{2r}}\right)\left(\mint_{B_{r}}\left | \frac{u-v}{r}\right | \, dx\right)\right]\nonumber \\
&\le &c\omega\left[r\left(H^{-}_{B_{2r}}\right)^{-1}\left(\mint_{B_{r}}H(x,Du-Dv) \, dx\right)\right]\nonumber \\
&\le &c\omega\left[r\left(H^{-}_{B_{2r}}\right)^{-1}\left(\mint_{B_{r}}\left[H(x,Du)+H(x,Dv)\right] \, dx\right)\right]\nonumber \\
&\stackleq{046}  &c\omega\left[r\left(H^{-}_{B_{2r}}\right)^{-1}\left(\mint_{B_{2r}}H(x,Du) \, dx\right)\right]\stackleq{omegau} c\varepsilon^{\beta}\;,
\end{eqnarray}again with \ccc{$c\equiv c(\data,\beta)$}. We can now start estimating the terms $\mathrm{(I)}$-$\mathrm{(V)}$ in \rif{043}; we have
\begin{eqnarray}\label{049}
\snr{\mathrm{(I)}}&\le &c\mint_{\brx}\omega\left(\snr{u-(u)_{\brx}}\right)H(x, Du) \, dx\nonumber \\
&\le &c\left(\mint_{\brx}\omega\left(\snr{u-(u)_{\brx}}\right)^{\frac{1+\delta_{g}}{\delta_{g}}} \, dx\right)^{\frac{\delta_{g}}{1+\delta_{g}}}\left(\mint_{\brx}[H(x,Du)]^{1+\delta_{g}} \, dx\right)^{\frac{1}{1+\delta_{g}}}\nonumber\\
&\stackleq{maggiore-original} &c\left(\mint_{\brx}\omega\left(\snr{u-(u)_{\brx}}\right)\, dx\right)^{\frac{\delta_{g}}{1+\delta_{g}}}
\left(\mint_{B_{2r}}H(x, Du) \, dx\right)\notag \\ &\stackleq{omegau} & \ccc{c(\data,\beta)}\varepsilon^{\frac{\beta\delta_{g}}{1+\delta_{g}}}\mint_{B_{2r}}H(x,Du) \, dx\;.
\end{eqnarray}
By minimality we see that
\begin{flalign}\label{050}
\mint_{\brx}F(x,u,Du) \, dx\le \mint_{\brx}F(x,v,Dv) \, dx\Longrightarrow \mathrm{(II)}\le 0\;. 
\end{flalign}
As for $(\mathrm{III})$, we have 
\begin{eqnarray}\label{051}
\snr{(\mathrm{III})}&\le &c\mint_{\brx}\omega\left(\snr{v-(v)_{\brx}}\right)H(x,Dv) \, dx\nonumber \\
&\le &c\left(\mint_{\brx}\omega\left(\snr{v-(v)_{\brx}}\right)^{\frac{1+\sigma_{g}}{\sigma_{g}}} \, dx\right)^{\frac{\sigma_{g}}{1+\sigma_{g}}}\left(\mint_{\brx}[H(x,Dv)]^{1+\sigma_{g}} \, dx\right)^{\frac{1}{1+\sigma_{g}}}\nonumber\\
&\stackleq{har6fre}&c\left(\mint_{B_{r}}\omega\left(\snr{v-(v)_{\brx}}\right) \, dx\right)^{\frac{\sigma_{g}}{1+\sigma_{g}}}\mint_{B_{2r}}H(x,Du) \, dx\notag \\
&\stackleq{omegav} &\ccc{c(\data,\beta)}\varepsilon^{\frac{\beta\sigma_{g}}{1+\sigma_{g}}}\mint_{B_{2r}}H(x,Du) \, dx
\;.
\end{eqnarray}
The estimation of $\mathrm{(IV)}$ is analogous to that of $\mathrm{(III)}$, the only difference being that in this case we must use \eqref{omegauv}; we end up with
\begin{flalign}\label{052}
\snr{\mathrm{(IV)}}\le \ccc{c(\data,\beta)}\varepsilon^{\frac{\beta\sigma_{g}}{1+\sigma_{g}}}\mint_{B_{2r}}H(x, Du) \, dx\;.
\end{flalign}
Finally we look at term $\mathrm{(V)}$. Proceeding as for the previous terms, and in particular using the smallness condition \rif{smallfull} as done in the last line of display \rif{omegau}, we have
\begin{eqnarray}\label{053}
\snr{\mathrm{(V)}}&\le &c\mint_{\brx}H(x,Dv)\snr{u-v} \, dx\nonumber \\&
\le & c\left(\mint_{B_{r}}\snr{u-v}^{\frac{1+\sigma_{g}}{\sigma_{g}}}\, dx\right)^{\frac{\sigma_{g}}{1+\sigma_{g}}}\left(\mint_{B_{r}}[H(x,Dv)]^{1+\sigma_{g}} \, dx\right)^{\frac{1}{1+\sigma_{g}}}\nonumber\\&
\le & c\left(\mint_{B_{r}}\snr{u-v} \, dx\right)^{\frac{\sigma_{g}}{1+\sigma_{g}}}\left(\mint_{B_{r}}[H(x,Dv)]^{1+\sigma_{g}} \, dx\right)^{\frac{1}{1+\sigma_{g}}}\nonumber\\
&\le &c\left[r\left(H^{-}_{B_{2r}}\right)^{-1}\left(\mint_{B_{r}}H^{-}_{B_{2r}}\left(\frac{u-v}{r}\right) \, dx\right)\right]^{\frac{\sigma_{g}}{1+\sigma_{g}}}\mint_{B_{2r}}H(x,Du) \, dx\nonumber\\
&\le &c\left[r\left(H^{-}_{B_{2r}}\right)^{-1}\left(\mint_{B_{r}}H\left(x,\frac{u-v}{r}\right)\, dx\right) \right]^{\frac{\sigma_{g}}{1+\sigma_{g}}}\mint_{B_{2r}}H(x,Du) \, dx\nonumber\\
&\le &c\left[r\left(H^{-}_{B_{2r}}\right)^{-1}\left(\mint_{B_{r}}H(x,Du-Dv) \, dx\right)\right]^{\frac{\sigma_{g}}{1+\sigma_{g}}}\mint_{B_{2r}}H(x,Du) \, dx\nonumber\\
&\le &c\left[r\left(H^{-}_{B_{2r}}\right)^{-1}\left(\mint_{B_{2r}}H(x,Du) \, dx\right)\right]^{\frac{\sigma_{g}}{1+\sigma_{g}}}\mint_{B_{2r}}H(x,Du) \, dx
\notag \\ &\le& \ccc{c(\data,\beta)}\varepsilon^{\frac{\sigma_{g}}{1+\sigma_{g}}}\mint_{B_{2r}}H(x,Du) \, dx\;.
\end{eqnarray}
Connecting estimates \rif{049}-\rif{053} to \rif{043}, and recalling that \cd{$\varepsilon < 1$}, $\beta \le 1$ and $\sigma_{g}<\delta_{g}$, we conclude with
\begin{flalign}\label{uv}
\mint_{\brx}\left(\snr{V_{p}(Du)-V_{p}(Dv)}^{2}+a(x)\snr{V_{q}(Du)-V_{q}(Dv)}^{2} \right)\, dx\le c\varepsilon^{\frac{\beta\sigma_{g}}{1+\sigma_{g}}}\mint_{B_{2r}}H(x,Du)\, dx\;,
\end{flalign}
holds for \ccc{$c\equiv c(\data,\beta)$}.\\
\emph{Step 2: $\partial F(\cdot,(u)_{B_{r}},\cdot)$-harmonic approximation.}\\
We aim to show that $v$ matches the assumptions of Lemma \ref{int}, with the choice $g(x,z)\equiv \partial F(x,(u)_{B_{r}},z) $; obviously, \cd{$\trif{assF}_{1,2,3,4}$} are satisfied, as well as \cd{$\rif{assF2}_{1}$}, by the very definition of $g(\cdot)$. As for \rif{har1}, we have 
\begin{eqnarray}\label{ckhar1}
\notag E(v;B_{r}):=\mint_{B_{r}}H(x,Dv) \, dx & \stackleq{046} & \frac{L}{\nu}\mint_{B_{r}}H(x,Du) \, dx\\ &=:&\frac{L}{\nu} E(u;B_{r}) \stackrel{\rif{smallfull}}{<}   2^{n}\frac{L}{\nu}H^{-}_{B_{2r}}\left(\frac{\varepsilon}{2r}\right)< c_1H^{-}_{B_{r}}\left(\frac{\varepsilon}{r}\right) \;,
\end{eqnarray}
which is in fact \eqref{har1} with $c_1:= 2^nL/\nu$. On the other hand, the validity of \rif{maggiorefa} is stated in \rif{maggiore-original-v}. To verify \eqref{har2}, we look at the Euler-Lagrange equation \eqref{EL1} solved by $v$ on $B_{r/2}$. By $\eqref{assF}_{1}$, for any $\varphi \in W^{1,\infty}_{0}(B_{r/2},\RN)$ we have
\begin{eqnarray}\label{ckhar2}
\left |\ \mint_{B_{r/2}}\partial F\left(x,(u)_{B_{r}},Dv\right)\cdot D\varphi \, dx\  \right |&=&\left |\ \mint_{B_{r/2}}\tilde{ F}'\left(x,(u)_{B_{r}},Dv\right)\snr{Dv}(v\cdot \varphi) \, dx\  \right |\nonumber\\
&\le &c\mint_{B_{r/2}}H(x,Dv)\snr{\varphi} \, dx\nonumber \\
&\le& cr\nr{D\varphi}_{L^{\infty}(B_{r/2})}\mint_{B_{r/2}}H(x,Dv) \, dx\notag \\ &\le & cr\nr{D\varphi}_{L^{\infty}(B_{r/2})}E(v;B_{r})
\end{eqnarray}
with $c\equiv c(n,L,p,q)$. The last term in display \eqref{ckhar2} can be estimated via \rif{Youngh} 
\cd{\begin{flalign}\label{you11}
cr\nr{D\varphi}_{L^{\infty}(B_{r/2})}E(v;B_{r})\le \delta_{1}H^{-}_{B_{r}}\left(\nr{D\varphi}_{L^{\infty}(B_{r/2})}\right)+\frac{c}{\delta_{1}^{1/(p-1)}}\left(H^{-}_{B_{r}}\right)^{*}\left(rE(v;B_{r})\right)\;,
\end{flalign}}
with $c\equiv c(n,L,p,q)$ and $\delta_1 \in (0,1)$, where \cd{$\left(H^{-}_{B_{r}}\right)^{*}$} denotes the convex conjugate of \cd{$H^{-}_{B_{r}}$}. Since 
\eqn{equistar}
$$\cd{\left(\left(H^{-}_{B_{r}}\right)^{*}\right)'=\left(\left(H^{-}_{B_{r}}\right)'\right)^{-1}\;,}$$ then, for 
\eqn{sceltaep1}
$$\cd{\varepsilon<\frac{p}{c_1}=\frac{p\nu}{2^nL}\;,}$$
(recall that $\eps$ in \rif{smallfull} is chosen in due course of the proof via various size restrictions), we find
\cd{\begin{eqnarray*}
\left(\left(H^{-}_{B_{r}}\right)^{*}\right)'\left(rE(v;B_{r})\right)&\stackleq{ckhar1} &\left(\left(H^{-}_{B_{r}}\right)^{*}\right)'\left(c_1rH^{-}_{B_{r}}\left(\frac{\varepsilon}{r}\right)\right)\le \left(\left(H^{-}_{B_{r}}\right)^{*}\right)'\left(\frac{c_1}{p}\varepsilon \left(H^{-}_{B_{r}}\right)'\left(\frac{\varepsilon}{r}\right)\right)\\
&\stackleq{sceltaep1} & \left(\left(H^{-}_{B_{r}}\right)^{*}\right)'\left( \left(H^{-}_{B_{r}}\right)'\left(\frac{\varepsilon}{r}\right)\right)
\stackrel{\rif{equistar}}{=}\left(\left(H^{-}_{B_{r}}\right)'\right)^{-1}\left( \left(H^{-}_{B_{r}}\right)'\left(\frac{\varepsilon}{r}\right)\right)
= \frac{\varepsilon}{r}\;,
\end{eqnarray*}}
so we can conclude with
$$
\cd{\left(H^{-}_{B_{r}}\right)^{*}\left(rE(v;B_{r})\right)\le rE(v;B_{r})\left(\left(H^{-}_{B_{r}}\right)^{*}\right)'(rE(v;B_{r}))\le \varepsilon E(v;B_{r})\;.}
$$
In this way \rif{you11} becomes 
$$
\cd{cr\nr{D\varphi}_{L^{\infty}(B_{r/2})}E(v;B_{r})\le \delta_{1}H^{-}_{B_{r}}\left(\nr{D\varphi}_{L^{\infty}(B_{r/2})}\right)+\frac{c\eps}{\delta_{1}^{1/(p-1)}}\mint_{B_{r}}H(x,Dv)\, dx\;.}
$$
Now select $\delta_{1}=\varepsilon^{(p-1)/2}$, and define $2t:=\min \left\{p-1,1\right\}$. The last inequality used in \eqref{ckhar2} gives
\begin{flalign*}
\left |\ \mint_{B_{r/2}}\partial F\left(x,(u)_{B_{r}},Dv\right)\cdot D\varphi \, dx \  \right |\le c_2\varepsilon^{t}\mint_{B_{r}}\left[H(x,Dv)+H\left(x,\nr{D\varphi}_{L^{\infty}(B_{r/2})}\right)\right] \, dx\;,
\end{flalign*}
for some $c_{2}=c(n,L,p,q)$ which is in fact \eqref{har2}. So Lemma \ref{int} applies and yields a $\partial F(\cdot,(u)_{B_{r}},\cdot)$-harmonic map $h\in W^{1,H}_{v}(B_{r/2},\mathbb{R}^{N})$, specifically, a solution to
\eqn{byvirtue}
$$
h \mapsto \min_{w\in W^{1,H}_{v}(B_{r/2},\RN)} \int_{B_{r/2}}F\left(x,(u)_{B_{r}},Dw\right) \, dx\;,
$$
such that \rif{hhh} holds; this, together with \rif{046}, allows to get
\begin{flalign}\label{hhh22}
\mint_{B_{r/2}}\left(\snr{V_{p}(Dv)-V_{p}(Dh)}^{2}+a(x)\snr{V_{q}(Dv)-V_{q}(Dh)}^{2} \right)\, dx \le c\varepsilon^{m}\mint_{B_{r}}H(x,Du) \, dx\;,
\end{flalign}
where \ccc{$c\equiv c(\data,\beta)$} and $m=m(\data)$. Moreover, \cd{there} holds that 
\eqn{massimo22}
$$
\|h\|_{L^{\infty}(B_{r/2})} \leq \sqrt{N}\;.
$$
By virtue of \rif{byvirtue} and of the previous inequality, we are then able to apply Theorem \ref{morreydecayth}. For every $\sigma \in (0,n]$, estimate \rif{morreydecayH} reads as
\eqn{morreydecayfinal}
$$\mint_{B_{t}}H(x,Dh) \, dx \le c\left(\frac ts\right)^{-\sigma}\mint_{B_{s}}H(x,Dh) \, dx\;,$$ that holds whenever 
\cd{$B_t\subset B_s \subset B_{r/2}$} are concentric balls, and where \ccc{$c\equiv c (\data, \beta, \sigma)$}, again by virtue of \rif{massimo22}; in the following we take $\sigma <1/4$. 
With \cd{$\tau \in (0,1/2)$}, recalling \rif{elemH} and using \rif{uv} and \rif{hhh22}, we can then estimate
\begin{eqnarray*}
\mint_{B_{2\tau r}}H(x,Du) \, dx&\le &c\mint_{B_{2\tau r}}\left(\snr{V_{p}(Du)-V_{p}(Dv)}^{2}+a(x)\snr{V_{q}(Du)-V_{q}(Dv)}^{2}\right) \, dx\nonumber\\
&&\quad +c\mint_{B_{2\tau r}}\left(\snr{V_{p}(Dv)-V_{p}(Dh)}^{2}+a(x)\snr{V_{q}(Dv)-V_{q}(Dh)}^{2} \right)\, dx\\ && \quad +c\mint_{B_{2\tau r}}H(x,Dh) \, dx \nonumber \\
&\le &c\tau^{-n}\mint_{B_{r}}\left(\snr{V_{p}(Du)-V_{p}(Dv)}^{2}+a(x)\snr{V_{q}(Du)-V_{q}(Dv)}^{2}\right) \, dx \nonumber\\
&&\quad +c\tau^{-n}\mint_{B_{r/2}}\left(\snr{V_{p}(Dv)-V_{p}(Dh)}^{2}+a(x)\snr{V_{q}(Dv)-V_{q}(Dh)}^{2}\right) \, dx\notag \\ && \quad +c\mint_{B_{2\tau r}}H(x,Dh) \, dx\nonumber \\
&\leq &c\left(\tau^{-n}\varepsilon^{\frac{\beta\sigma_{g}}{1+\sigma_{g}}}+\tau^{-n}\varepsilon^{m}+\tau^{-\sigma}\right)\mint_{B_{2r}}H(x,Du) \, dx\;,
\end{eqnarray*}
where \ccc{$c\equiv c (\data, \beta, \sigma)$}. Recalling the notation adopted in \eqref{ckhar1}, the conclusion of the last display reads as
\begin{eqnarray}\label{est1}
\notag E\left(u;B_{2\tau r}\right)&\le & \ccc{c(\data, \beta, \sigma)}\left(\tau^{-n}\varepsilon^{\frac{\beta\sigma_{g}}{1+\sigma_{g}}}+\tau^{-n}\varepsilon^{m}+\tau^{-\sigma}\right)E(u;B_{2r})\\
&= & \ccc{c(\data,\beta, \sigma)}\tau^{\sigma}\left(\tau^{-n-\sigma}\varepsilon^{\frac{\beta\sigma_{g}}{1+\sigma_{g}}}+\tau^{-n-\sigma}\varepsilon^{m}+\tau^{-2\sigma}\right)E(u;B_{2r})
\;. 
\end{eqnarray} 
We can now determine \ccc{$\tau \equiv \tau (\data, \beta,\sigma)\in (0,1)$} such that 
\eqn{sceltacostanti}
$$
c(\data, \sigma)\tau^{\sigma}\leq \frac 16 
$$
It is now time to choose the number $\eps$ coming from \rif{smallfull}. 
Recalling \rif{sceltaep1}, we now further reduce $\eps$ to have 
\eqn{sceltaeppi}
$$\cd{\varepsilon< \min\left\{\frac{p\nu}{2^{n}L},\tau^{\frac{(n-\sigma)(1+\sigma_{g})}{\beta\sigma_{g}}},\tau^{\frac{n-\sigma}{m}}\right\}}$$ and notice that this fixes the dependence \ccc{$\varepsilon\equiv \varepsilon(\data,\beta, \sigma)$}. By using \rif{sceltacostanti} and \rif{sceltaeppi} in  \eqref{est1}, this last inequality reads as
\eqn{est5pre}
$$
\cc{E\left(u;B_{2\tau r}\right)\le \tau^{-2\sigma}E(u;B_{2r})}
$$
that is, recalling the definition in \rif{ckhar1}
\begin{flalign}\label{est5}
\cc{\int_{B_{2\tau r}}H(x, Du) \, dx\le  \tau^{n-2\sigma}\int_{B_{2r}}H(x, Du) \, dx}\;. 
\end{flalign}
Next, we observe that
$$
\cc{E\left(u;B_{2\tau r}\right)\stackleq{est5pre} \tau^{-2\sigma}E(u;B_{2r})\stackrel{\eqref{smallfull}}{<} \tau^{1-2\sigma}\tau^{-1}H^{-}_{B_{2r}}\left(\frac{\varepsilon}{2r}\right)\le\tau^{-1} H^{-}_{B_{2\tau r}}\left(\frac{\varepsilon}{2 r}\right)\le H^{-}_{B_{2\tau r}}\left(\frac{\varepsilon}{2\tau r}\right)}\;,
$$
and we conclude with 
$$
E\left(u;B_{2\tau r}\right)<H^{-}_{B_{2\tau r}}\left(\frac{\varepsilon}{2\tau r}\right)\;.
$$
We have therefore proved that, for the choice of \ccc{$\tau\equiv \tau(\data,\beta, \sigma)$} and \ccc{$\eps\equiv\eps(\data,\beta, \sigma)$} made in \rif{sceltacostanti} and \rif{sceltaeppi}, respectively, if the smallness condition \rif{smallfull} is satisfied on the ball $B_{2r}$ it is also satisfied on the ball $B_{2\tau r}$. We can therefore repeat the whole argument developed after \rif{smallfull} starting from the ball $B_{2\tau r}$ instead of $B_{2r}$, thereby arriving at the analog of \rif{est5pre}, that is $E(u;B_{2\tau^2 r})\le \tau^{-2\sigma}E\left(u;B_{2\tau r}\right)$. This argument can obviously be iterated on the family of shrinking balls $\{B_{\tau^j r}\}$, thereby concluding that, for every $j \in \en$, it holds that 
$$
E\left(u;B_{2\tau^j r}\right)<H^{-}_{B_{2\tau^j r}}\left(\frac{\varepsilon}{2\tau^j r}\right)
$$
and
$$
\int_{B_{2\tau^j r}}H(x, Du) \, dx< \tau^{(n-2\sigma)j}\int_{B_{2r}}H(x, Du) \, dx\;.$$
In turn, a standard interpolation argument leads to conclude that
\eqn{decayfinale} 
$$\int_{B_{t}(x_0)}H(x,Du) \, dx \le c\left(\frac t{r}\right)^{n-2\sigma}\int_{B_{2r}(x_0)}H(x,Du) \, dx\;,\qquad \forall \ t  \leq 2r\;,$$
where \ccc{$c\equiv c(\data, \beta,\sigma)$}. Notice that the above inequality has been derived for $4\sigma <1$ but it is then easily seen to hold whenever $\sigma \in (0,1)$. Going back to \rif{smallfull}, we observe that the two functions 
$$
x_0 \mapsto \mint_{B_{2r}(x_0)}H(x, Du) \, dx\qquad \mbox{and}\qquad x_0\mapsto H^{-}_{B_{2r}(x_0)}\left(\frac{\varepsilon}{2r}\right)
$$
are continuous. This is a consequence of the absolute continuity of the integral for the former, and of Remark \ref{h-} for the latter. We conclude that, with $\sigma$ being fixed, if \rif{smallfull} is satisfied at a point $x_0\in \Omega$, then there exists ball $B_{4r_{x_0}}(x_0)$ such that 
\eqn{piccolo-raggio}
$$
y \in B_{4r_{x_0}}(x_0) \Longrightarrow \mint_{B_{2r}(y)}H(x, Du) \, dx< H^{-}_{B_{2r}(y)}\left(\frac{\varepsilon}{2r}\right)\;.
$$
We then conclude that \rif{decayfinale} holds (with $y$ replacing $x_0$), and with the same constant \ccc{$c\equiv c (\data,\beta, \sigma)$}, whenever 
$y \in B_{4r_{x_0}}(x_0)$. By a standard characterization of H\"older continuity it then follows that $u \in C^{0, \gamma}\left(B_{r_{x_0}}(x_0)\right)$ with $\gamma=1-2\sigma/p$ (see Remark \ref{camplocal} below). As we can choose $\sigma\in (0,1/4)$ arbitrarily, we have finally proved the following (we can switch from $2r$ to $r$ now):
\begin{proposition}\label{main1p} Let $u \in W^{1,1}_{\loc}(\Omega,\SN)$ be a \cd{constrained} local minimizer of the functional $\mathcal F$ in \trif{genF},  under the assumptions \trif{assF}-\trif{mainbound}.  Assume that $p(1+\delta_g)\leq n$. Then, for every positive exponent $\gamma<1$, there exists another positive number  \ccc{$\eps_{\gamma}\equiv \eps_{\gamma}(\data,\beta, \gamma)$} such that if $B_{r}(x_0)\Subset \Omega$, $r \leq 1$, and 
\eqn{carepsdopo}
$$
\left[H^{-}_{\br}\left(\frac{\eps_{\gamma}}{r}\right)\right]^{-1}\mint_{\br}H(x, Du) \, dx< 1\;,
$$
then $u$ is of class $C^{0,\gamma}$ in a neighbourhood of $x_0$.
\end{proposition}
\begin{remark} \label{camplocal}
\emph{Let us make the last argument somehow more quantitative. With $\gamma \in (0,1)$ being fixed, let $\tilde r(x_0)$ be the largest radius, such that the smallness condition \rif{carepsdopo} is satisfied with $r \equiv \tilde r(x_0)$ (together with $B_{2\tilde r(x_0)}\Subset \Omega$). Then we have that 
\begin{eqnarray*}
\int_{B_{t}(x_0)}H(x, Du) \, dx &\le & c\left(\frac t{\tilde r(x_0)}\right)^{n-p+p\gamma}\int_{B_{2\tilde r(x_0)}(x_0)}H(x,Du) \, dx\\ &\stackleq{caccine2}& ct^{n-p+p\gamma}\left[[\tilde r(x_0)]^{p-p\gamma}\mint_{B_{4\tilde r(x_0)}(x_0)}\left( \frac{|u|^p}{[\tilde r(x_0)]^p} + \|a\|_{L^{\infty}}\frac{|u|^q}{[\tilde r(x_0)]^q} \right)dx\right]\;. 
\end{eqnarray*}
By using Poincarè's inequality, we get 
$$
\cd{\mint_{B_{t}(x_0)}|u-(u)_{B_t}|^p \, dx  \leq \frac{\ccc{c(\data,\nr{a}_{L^{\infty}},\beta, \gamma)}}{[\tilde r(x_0)]^{q-p+p\gamma}} t^{p\gamma}\;}.
$$
This estimate remains stable whenever $x_0$ is replaced by $y \in B_{4r_{x_0}}(x_0)$ as in \trif{piccolo-raggio} and therefore, by a standard integral characherization of H\"older continuity, for all $\gamma\in (0,1)$ it follows that
\eqn{ricopri}
$$
[u]_{0,\gamma;B_{r_{x_0}}(x_0)} \leq \frac{\ccc{c(\data,\nr{a}_{L^{\infty}},\beta, \gamma)}}{[\tilde r(x_0)]^{q/p-1+\gamma}}\;.
$$
}
\end{remark}

\noindent \emph{Step 3: Dimension of the singular set; the first estimate and proof of \trif{careps2}.} \\
Following a standard terminology, we denote by 
$$
\Omega_u := \left\{x_0 \in \Omega \, \colon \, \mbox{there exists a ball $B_{r_{x_0}}(x_0)\subset \Omega$ such that $u \in C^{0, \gamma}\left(B_{r_{x_0}}(x_0)\right)$ for some $\gamma \in (0,1)$} \right\}\;.
$$
This set is open by definition and we denote $\Sigma_{u}:=\Omega\setminus \Omega_{u}$, the so-called singular set. Indeed, we shall later on prove that $Du$ is locally H\"older continuous on $\Omega_u$ and this justifies the terminology used here with respect with the one in the statement of Theorem \ref{main1}. Let us now prove \rif{careps2}; call $\Sigma$ the set in the right-hand side of \rif{careps2}. The inclusion $\Sigma_u\subset \Sigma$ is obvious in view of \cd{Proposition \ref{main1p}}. On the other hand, take, by contradiction $x_0\in \Sigma\setminus \Sigma_u$. Then there exists $\gamma \in (0,1)$ such that $u$ is of class $C^{0, \gamma}\left(B_{r_{x_0}}(x_0)\right)$ for some ball $B_{r_{x_0}}(x_0)$. We then look at the Caccioppoli type inequality in \rif{caccine2dopo}; this implies that, for $\varrho \leq r_{x_0}/2$
$$
\cd{\mint_{B_{\rr}(x_0)}H(x, Du) \, dx\le c\mint_{B_{2\rr}(x_0)}H^{-}_{B_{2\rr}(x_0)}\left(\frac{u-(u)_{B_{2\rr}(x_0)}}{2\rr} \right) \, dx \leq cH^{-}_{B_{\rr}(x_0)}\left(\rr^{\gamma-1}\right)}  
$$
and
\eqn{takesmall}
$$
\left[H^{-}_{B_{\varrho}(x_0)}\left(\frac{1}{\rr}\right)\right]^{-1}\mint_{B_{\varrho}(x_0)}H(x, Du) \, dx \leq
c\varrho^{\gamma p} \;, 
$$
for \ccc{$c\equiv c (\data,\beta, [u]_{0,\gamma})$} (here we have used that $H^{-}_{B_{2\rr}}\leq H^{-}_{B_{\rr}}$). Letting $\rr \to 0$ in the above display implies $x_0 \not\in \Sigma$, a contradiction. This proves that $\Sigma \subset \Sigma_u$ and therefore \rif{careps2}. Next, observe that by definition we have $H^{-}_{B_{\varrho}(x_{0})}(1/\varrho)\geq \varrho^{-p}$. Combining this with Lemma \ref{C1} (which is used to assert the integrability of $[H(x,Du)]^{1+\delta_g} $) we conclude with
\begin{flalign}\label{ss}
\Sigma_{u}\subset \left\{x_{0}\in \Omega\, \colon \, \limsup_{\varrho\to 0}\, \varrho^{p(1+\delta_{g})-n}\int_{B_{\varrho}(x_{0})}[H(x,Du)]^{1+\delta_g} \, dx>0\right\}\;.
\end{flalign}
Giusti's Lemma (\cite[Proposition 2.7]{G}) then implies 
\eqn{dimensione-rid1}
$$\dim_{\mathcal{H}}(\Sigma_{u})\le n-p-p\delta_{g}\Longrightarrow \mathcal{H}^{n-p}(\Sigma_{u})=0\;.$$ Recall we are treating the case when $p(1+\delta_{g})\leq n$. In particular, we have $|\Sigma_u|=0$. Notice also that, once proved that $Du$ is locally H\"older continuous in $\Omega_u$, we shall have proved the validity of \rif{zerop}. In a totally similar way, assume also that $q(1+\delta_g)\leq n$; we observe that 
if $x_{0}\in \Sigma_u$ is such that $a(x_0)>0$, then $H^{-}_{B_{\varrho}(x_{0})}(1/\varrho)\ge \ai(B_{\varrho}(x_{0}))\varrho^{-q}>0$ for $\varrho$ sufficiently small, and we have
$$
\Sigma_{u}\cap\{a(x)>0\}\subset \left \{x_{0}\in \Omega\, \colon \,  \limsup_{\varrho\to 0}\, \varrho^{q(1+\delta_{g})-n}\int_{B_{\varrho}(x_{0})}[H(x,Du)]^{1+\delta_{g}} \, dx >0\right\}\;.
$$
Therefore, again by Giusti's Lemma, we also have 
\eqn{dimensione-rid2}
$$\dim_{\mathcal{H}}(\Sigma_{u})\le n-q-q\delta_{g}\Longrightarrow \mathcal{H}^{n-q}(\Sigma_{u}\cap\{a(x)>0\})=0\;.$$
Finally, we observe that in fact we have
\eqn{nuovaide}
$$
\Omega_u = \left\{x_0 \in \Omega \, \colon \, \mbox{there exists a ball $B_{r_{x_0}}(x_0)\subset \Omega$ such that $u \in C^{0, \gamma}\left(B_{r_{x_0}}(x_0)\right)$ for every $\gamma \in (0,1)$} \right\}\;.
$$
Indeed, call $\tilde \Omega_u$ the set in the right-hand side of the previous display; \cd{$\tilde \Omega_u \subset \Omega_u$, again by Proposition \ref{main1p}}. On the other hand, the us take $x_0 \in \Omega_u$; it follows that there exists $B_{r_{x_0}}(x_0)\subset \Omega$ with $u \in C^{0, \tilde \gamma}\left(B_{r_{x_0}}(x_0)\right)$ for some $\tilde \gamma <1$; then, fix $\gamma<1$ and determine the corresponding \ccc{$\eps_{\gamma}\equiv \eps_{\gamma} (\data,\beta, \gamma)$} according to Proposition \ref{main1p}. We can take $\varrho$ small enough in \rif{takesmall} (this time with $\gamma$ replaced by $\tilde \gamma$) in such a way that the smallness condition in \rif{carepsdopo} is satisfied. This implies that $u$ is $\gamma$-H\"older continuous in a neighbourhood of $x_0$. As $\gamma \in (0,1)$ has been chosen arbitrarily, we deduce that $x_0 \in \tilde \Omega_u$ and therefore $\Omega_u=\tilde \Omega_u$, that is, \rif{nuovaide} is completely proved.  

By \rif{nuovaide} and matching the content of Remark \ref{camplocal} (in particular, see \rif{ricopri}) with a standard covering argument, we conclude with 
\begin{proposition}\label{main1pp} Let $u \in W^{1,1}_{\loc}(\Omega,\SN)$ be a constrained local minimizer of the functional $\mathcal F$ in \trif{genF}, under the assumptions \trif{assF}-\trif{mainbound}.  Assume that $p(1+\delta_g)\leq n$. For every open subset $\Omega_0 \Subset \Omega_u$ and every $\gamma \in (0,1)$, there exist constants \ccc{$c, \tilde c\equiv c, \tilde c  (\data,\nr{a}_{L^{\infty}},\beta, \gamma, \Omega_0)$} such that 
\eqn{parthol}
$$
\mint_{B_{\varrho}}H(x,Du) \, dx \le c\varrho^{(\gamma-1)q}\qquad \mbox{and} \qquad [u]_{0, \gamma;\Omega_0}\leq \tilde c\;,
$$
with the first one that holds whenever $B_{2 \varrho} \subset \Omega_0$ with $2\varrho \leq 1$.
\end{proposition}
In \rif{parthol} we notice that the second inequality actually implies the first one via \rif{caccine2dopo}. 
As for the rest of the proof, as mentioned above, we only need to show that $Du$ is locally H\"older continuous in $\Omega_u$.

\noindent \emph{Step 4: Passage to coordinates.}\\
After the proof of the local partial H\"older continuity of $u$, we can now pass to coordinates using stereographic projections. The procedure is standard in the case of functionals with $p$-growth but, since we are dealing with non-standard growth conditions, we need to check extra regularity conditions and therefore we shall repeat it in some detail. Having \rif{parthol} in mind, we fix a certain initial $\gamma$, say $ \gamma = 1/2$. We then consider  bounded open subsets $\tilde{\Omega}\Subset \Omega_0\Subset \Omega_{u}$ and cover $\tilde \Omega$ with finitely many balls $B \subset \Omega_0$ with sufficiently small radius (size and number here only depend on $n,N$, $[u]_{0,1/2;\Omega_0}$ and diam$(\Omega_0)$) such that $u(B)$ lies in single  coordinate neighbourhood of $\SN$. More precisely, \cd{if $B\equiv B_{r(B)}(x_0)$ is one of such balls}, up to rotations we can assume that $u(x_{0})=(-1,0,\cdots,0)$ and that $u^1(x)\leq -1/2$ for every $x\in B$. Given this, with no loss of generality we can reduce to the case \ccc{in which} we are working on an open subset $\tilde \Omega \Subset \Omega$ such that $u^1(x)\leq -1/2$ for every $x \in \tilde \Omega$. This is the setting we shall use in the rest of the proof and our next goal is now to prove that $Du$ is \cc{locally $\beta_{0}$-H\"older continuous} in $\tilde \Omega$, \cc{with $\beta_{0}$ depending only on \ccc{$(\data,\beta,\beta_{1})$}, where $\beta_{1}$ is the exponent appearing in $\eqref{assF2}_{2}$}. The full statement of Theorem \ref{main1} then follows again via a standard covering argument. To proceed, denoting by $P(\cdot)$ the usual stereographic projection $P\colon \SN \setminus \{(1, \ldots, 0)\}\to \er^{N-1}$ and by $S:= P^{-1}\colon \er^{N-1}\to \SN \setminus \{(1, \ldots, 0)\}$ its inverse, i.e.,
\eqn{stereorecall}
$$S(y)=\left(\frac{\snr{y}^{2}-1}{\snr{y}^{2}+1},\frac{2y}{\snr{y}^{2}+1}\right)\;,\qquad S^{-1}(v)=\left(\frac{v^{i}}{1-v^{1}}\right)_{i=2}^{N}\;,$$
we then define $\tilde u := S^{-1}(u)$. We note that
\begin{flalign}\label{S}
\nr{\nabla S}_{L^{\infty}}\le c(N),\quad \nr{\nabla^{2} S}_{L^{\infty}}\le c(N) \quad \mathrm{and} \quad \snr{\nabla (S^{-1})(u(x))}\le c(N)\;,
\end{flalign}
the last inequality being valid for all $x \in \tilde{\Omega}$ (as $u^{1}(x)\leq -1/2$ whenever $x\in \tilde \Omega$). Recalling \rif{stereorecall}, again that $u^{1}(x)\leq -1/2$ whenever $x\in \tilde \Omega$, and that \cd{$\tilde{u}=S^{-1}(u)$}, we get
\eqn{stereo}
$$
\cd{\|\tilde u\|_{L^{\infty}(\tilde \Omega)}\leq \frac{2}{3}\leq 1\;}.
$$
Again \rif{S} implies that if $\tilde w \in W^{1,H}(\tilde{\Omega},\er^{N-1})$, then $S(\tilde w) \in W^{1,H}(\tilde{\Omega},\SN)$. Therefore, by the minimality of $u$ it follows that the map $\tu\in W^{1,H}(\tilde{\Omega},\mathbb{R}^{N-1})$ is a local minimizer of the functional 
\begin{flalign}\label{fung}
W^{1,H}(\tilde{\Omega},\mathbb{R}^{N-1})\ni w \mapsto \int_{\tilde{\Omega}}G(x,w, Dw) \, dx\;,
\end{flalign}
where the integrand $G(\cdot)$ is defined by $$G(x,y,z):=F\left(x, S(y), \nabla S(y) z\right)\qquad \mbox{for}\ x\in \tilde \Omega, \ y\in \er^{N-1}, \ z\in \er^{(N-1)\times n}\;.$$
As it is 
\eqn{ilnabla}
$$|\nabla S(y) z|=\frac{2}{(1+\snr{y}^{2})}\snr{z}$$ 
(as it follows from an elementary but lengthy computation), recalling \cd{$\rif{assF2}_{1}$} we conclude with
\eqn{strutturaher}
$$G(x,y,z)= \tilde{G}(x,y,\snr{z})\equiv \tilde{G}_{x, y}(\snr{z}):=\tilde{F}\left(x,S(y),2(1+\snr{y}^{2})^{-1}\snr{z}\right)\;.$$
By using the starting assumptions \rif{assF}, it is now not difficult to show that for every \cc{$M\geq 3N$} there exist new constants $0 < \tilde \nu\equiv \tilde \nu(\data, M)\leq 1 \leq \tilde L\equiv \tilde L(\data)$, such that
\begin{flalign}\label{assG}
\begin{cases}
\ \tilde \nu H(x,z) \leq G(x,y,z)\le \tilde LH(x,z)\\
\ \snr{\partial G(x,y,z)}\snr{z}+\snr{\partial^{2}G(x,y,z)}\snr{z}^{2}\le \tilde LH(x,z)\\
\ \tilde \nu(\snr{z}^{p-2}+a(x)\snr{z}^{q-2})\snr{\xi}^{2}\le \langle \partial^{2}G(x,y,z)\xi,\xi\rangle\\
\ \snr{\partial G(x_{1},y,z)-\partial G(x_{2},y,z)}\snr{z}\le \tilde L\omega(\snr{x_{1}-x_{2}})[H(x_{1},z)+H(x_{2},z)]+\tilde L\snr{a(x_{1})-a(x_{2})}|z|^{q}\\
\ \snr{G(x,y_{1},z)-G(x,y_{2},z)}\le \tilde L\omega(\snr{y_{1}-y_{2}})H(x,z)\;,
\end{cases}
\end{flalign}
hold whenever $x, x_{1}, x_{2}\in \Omega$, $z \in \er^{(N-1)\times n}\setminus \{0\}$, $\xi \in \er^{(N-1)\times n}$ and $y, y_1, y_2\in \er^{N-1}$ are such that $|y|+|y_1|+ |y_2|\leq M$. In the lines above, $\tilde \nu $ is a non-increasing function of $M$. All the inequalities in \rif{assG} are consequences of the definition in \rif{strutturaher} and of \rif{ilnabla} and we leave the details of the verification to the reader. We just spend a few words on the verification of \rif{assG}$_3$. By using \rif{leduezero}, from the explicit representation in \rif{strutturaher} we get
\eqn{idueterminidddd}
$$
\cc{ \partial^2 G(x,y,z) 
  = 4\frac{\tilde{F}''(x,S(y),\snr{\tilde z})}{(1+\snr{y}^{2})^2} \frac{\tilde z \otimes \tilde z}{|\tilde z|^2} + 4\frac{\tilde{F}'(x,S(y),\snr{\tilde z})}{ (1+\snr{y}^{2})^2}\left[ \frac{\mathbb{I}_{N\times n}}{|\tilde z|}- \frac{\tilde z \otimes \tilde z}{|\tilde z|^3} \right] \;, }
$$ 
where we have denoted $\tilde z := 2 (1+\snr{y}^{2})^{-1}z$. Taking $\xi \in \er^{(N-1)\times n}$ and adding one more null component to both $\tilde z$ and $\xi$ (thereby making then $\er^{N\times n}$ matrices), and using \rif{iduetermini} and \rif{assF}$_3$, yields
$$
 \langle \partial^{2}G(x,y,z)\xi,\xi\rangle \geq \frac{\nu}{2^q} \left(\frac{|z|^{p-2}}{(1+|y|^2)^p}+a(x)\frac{|z|^{q-2}}{(1+|y|^2)^q}\right)|\xi|^2\geq  \frac{\nu(\snr{z}^{p-2}+a(x)\snr{z}^{q-2})}{(2+4M^2)^q} |\xi|^2\;,$$
that is, \rif{assG}$_3$. 
Finally, fix $x \in \tilde \Omega$ and $y \in \er^{N-1}$. By \eqref{assF2}$_1$ and \rif{strutturaher} it follows that
\begin{flalign}\label{tck7}
t\mapsto \tilde{G}_{x, y}(\cdot,t) \ \mbox{is \ non-decreasing}
\end{flalign}
and this means that the structure assumption \eqref{assF2}$_1$ is verified also by $G(\cdot)$. As for the analog of \eqref{assF2}$_2$, observe that \rif{strutturaher} implies that 
$
\tilde{G}_{x,y}''(t)=[2(1+\snr{y}^{2})^{-1}]^{2}\tilde{F}_{x,S(y)}''(2(1+\snr{y}^{2})^{-1}t). 
$
Then, with $\snr{s}<t/2$ and $t>0$, we define $\tilde{s}:=2(1+\snr{y}^{2})^{-1}s$ and $\tilde{t}:=2(1+\snr{y}^{2})^{-1}t$; taking into account \rif{newform} we find
\begin{eqnarray}
\notag \snr{\tilde{G}_{x,y}''(t+s)-\tilde{G}_{x,y}''(t)}&=&\frac{4}{(1+\snr{y}^{2})^{2}}\snr{\tilde{F}_{x,S(y)}''(\tilde{t}+\tilde{s})-\tilde{F}_{x,S(y)}''(\tilde{t})}\\
&\leq & \cc{\frac{c}{(1+\snr{y}^{2})^{2}}}\tilde{F}_{x,S(y)}''(\tilde{t})\left(\frac{\snr{\tilde{s}}}{\tilde{t}}\right)^{\beta_1}= c\tilde{G}_{x,y}''(t)\left(\frac{\snr{s}}{t}\right)^{\beta_1}\label{anche G}\;,
\end{eqnarray}
where it is again $c\equiv c(n,N,\nu,L,p,q)$. We moreover remark that, by the growth conditions in \rif{assG}, exactly as done for \rif{equivalenti} in Remark \ref{discuss-re} we infer that 
\eqn{equivalentiG}
$$
\tilde{G}_{x,y}''(t) t \approx\tilde{G}_{x,y}'(t) \qquad \mbox{holds for every}\ t>0$$
and here the implied constants depend on $n,N,\nu,L,p,q,M$ and \ccc{they} are independent of $(x,y)$.  

\noindent \emph{Step 5: Partial H\"older continuity of the gradient.}\\
First of all, let us observe \cc{that} for reasons that will be clear in a few lines, and in view of \rif{stereo}, in the following, when considering \rif{assG}, we shall permanently use the choice \cd{$M=10N$}. Recall from \emph{Step 2} that $u\in C^{0,\gamma}_{\mathrm{\rm{loc}}}(\Omega_{u},\SN)$ for every $\gamma \in (0,1)$, and so, by $\eqref{S}$ and \rif{parthol}, we have that $\tu \in C^{0,\gamma}(\tilde{\Omega},\mathbb{R}^{N-1})$ for every $\gamma \in (0,1)$, with 
\eqn{komzet0}
$$[\tu]_{0, \gamma;\tilde{\Omega}}\le c(N)[u]_{0, \gamma;\tilde{\Omega}}\leq \ccc{c(\data,\nr{a}_{L^{\infty}},\beta, \gamma, \tilde \Omega)}\;.$$ For any ball $B_{4r}\Subset \tilde{\Omega}$ with $r\leq 1/8$, as $\tu$ minimizes the functional in \eqref{fung} and \rif{stereo} holds, also taking Remark \ref{reun} into account, Lemma \ref{L5} provides
\eqn{holcac}
$$
\mint_{B_{2r}}H(x,D\tu) \ dx \le c\mint_{B_{4r}}H^{-}_{B_{4r}}\left(\frac{\tu-(\tu)_{B_{4r}}}{4r} \right) \ dx\stackleq{komzet0} cH^{-}_{B_{r}}\left(r^{\gamma-1}\right)\;,
$$
holds with \ccc{$c\equiv c(\data, \beta,\gamma, \tilde \Omega)$}, for every $\gamma \in (0,1)$. In particular, it follows that
\begin{flalign}\label{komzet}
\mint_{B_{2r}}H(x,D\tu) \, dx \le cr^{(\gamma-1)q}\;,
\end{flalign}
for all $\gamma \in (0,1)$, where \ccc{$c\equiv c(\data,\nr{a}_{L^{\infty}},\beta, \gamma, \tilde \Omega)$}.  We start fixing $\gamma\ge 1/2$; we shall further increase the value of $\gamma$ in due course of the proof (and the constants involved will increase accordingly). Fix \cd{$B_{r}\equiv B_{r}(x_{0})$} such that $B_{4r}\Subset \tilde{\Omega}$, $r \leq 1/64$, and let $\tv\in W^{1,H}(B_{r},\mathbb{R}^{N-1})$ be a solution to the frozen Dirichlet problem
\begin{flalign}\label{frzet}
\tilde v \mapsto \min_{w \in W^{1,H}_{\tu}(B_{r},\mathbb{R}^{N-1})} \int_{B_{r}}G\left(x,(\tu)_{B_{r}},Dw\right) \, dx\;.
\end{flalign}
By \rif{stereo} the integrand $G(\cdot, (\tu)_{B_{r}},\cdot)$ satisfies assumptions \rif{assG} with $\tilde \nu$ and $\tilde L$ only depending on $\data$ as we have fixed \cd{$M=10N$}. In particular, there exist positive numbers $\tilde \nu, \tilde L$ such that
\eqn{bilanciocrescite}
$$
\tilde \nu(\data) H(x,z) \leq G\left(x,(\tu)_{B_{r}},z\right) \leq \tilde L(\data) H(x,z)
$$
holds whenever $x \in \Omega$ and $z \in \er^{(N-1)\times n}$. By \rif{tck7} and the maximum principle \cite[Theorem 2.3]{leosie} we then have 
\eqn{limitatissimo}
$$\|\tv\|_{L^{\infty}(B_{r})}\leq \sqrt{N}\|\tu\|_{L^{\infty}(B_{r})}\stackleq{stereo} \cd{\frac{2\sqrt{N}}{3}}\;.$$ The validity of the Euler-Lagrange equation for \eqref{frzet} can be checked as done in \cite{CM2} (see also Section \ref{eldp}
and apply the same arguments exposed there without using projections, that is, when no constraints are involved). Specifically, $\tv$ solves
\begin{flalign}\label{elv}
\int_{B_{r}}\partial G\left(x,(\tu)_{B_{r}},D\tv\right)\cdot D\varphi \, dx=0\;, 
\end{flalign}
for all $\varphi \in W^{1,H}_{0}(B_{r},\mathbb{R}^{N-1})$.

\begin{remark}
\emph{From now on, we adopt the following convention. Also taking the content of Remarks \ref{reun}-\ref{reun2} and \rif{bilanciocrescite}-\rif{limitatissimo} into account, the results of Lemma \ref{C1} and \ref{C2} apply to $\tilde u, \tilde v$ and lead to new higher integrability exponents $\delta_g$ and $\sigma_g$. The values of $\delta_g$ and $\sigma_g$ are different from those used in the previous steps for $u$, but essentially equivalent to them. Indeed, they still depend on the same set of parameters, that is $\data$. Therefore, with some abuse of notation, we shall keep on denoting by \cd{$\sigma_g < \delta_g$} the higher integrability exponents provided by the application of Lemmas \ref{C1}-\ref{C2} in the present setting (notice that what is denotes by $N$ in Lemmas \ref{C1}-\ref{C2} is actually $N-1$ here; this \cc{can} be mde rigorous eventually taking the smallest amongst all the exponents considered when the \cc{values} of $\nu$ and $L$ attain \cc{their} minimum and maximum, respectively). All in all, the following inequalities hold as in \rif{046} and \rif{har6fre}:
\eqn{046final}
$$
\mint_{\brx}H(x,D\tv)\, dx \le c\mint_{\brx}H(x, D\tu) \, dx\,,\quad \left(\mint_{B_{r/2}}[H(x,D\tv )]^{1+\delta_{g}} \, dx\right)^{1/(1+\delta_{g})}  \leq \tilde c_1 \mint_{B_{r}}H(x,D\tv )\, dx
$$
and 
\eqn{046final2}
$$
\left(\mint_{B_{r}}[H(x,D\tv )]^{1+\sigma_{g}} \, dx\right)^{1/(1+\sigma_{g})}\le c\left(\mint_{B_{r}}[H(x,D\tu )]^{1+\sigma_{g}} \, dx\right)^{1/(1+\sigma_{g})}\leq  c\mint_{B_{2r}}H(x,D\tu ) \, dx\;, 
$$
for $c, \tilde c_1\equiv c, \tilde c_1 (\data)$. This last dependence on the constants is a consequence of \rif{bilanciocrescite}-\rif{limitatissimo}. 
}
\end{remark}
As $\varphi \equiv \tu-\tv$ is a legal choice in \rif{elv}, using \rif{042} as in \rif{043}, we end up with
\begin{eqnarray}
&& c\mint_{B_{r}}\left(\snr{V_{p}(D\tu)-V_{p}(D\tv)}^{2}+a(x)\snr{V_{q}(D\tu)-V_{q}(D\tv)}^{2}\right) \ dx \notag \\
&&\qquad \le \mint_{B_{r}}\left[G\left(x,(\tu)_{B_{r}},D\tu\right)-G\left(x,(\tu)_{B_{r}},D\tv\right)\right] \ dx \notag \\
&&\qquad =\mint_{B_{r}}\left[G\left(x,(\tu)_{B_{r}},D\tu\right)-G(x,\tu,D\tu) \right]\ dx +\mint_{B_{r}}\left[G(x,\tu,D\tu)-G(x,\tv,D\tv) \right]\ dx \notag \\
&&\qquad \quad +\mint_{B_{r}}\left[G\left(x,\tv,D\tv\right)-G\left(x,(\tv)_{B_{r}},D\tv\right) \right]\ dx \notag \\ && \qquad \quad  +\mint_{B_{r}}\left[G\left(x,(\tv)_{B_{r}},D\tv\right)-G\left(x,(\tu)_{B_{r}},D\tv\right) \right]\ dx =:\sum_{j=1}^{4}(\mathrm{I})_{j}\;,\label{buy}
\end{eqnarray}with $\cd{c\equiv c(\texttt{data})}$. Before taking care of terms $(\mathrm{I})_{1}$-$(\mathrm{I})_{4}$,  we derive a few preliminary inequalities. The first one is an obvious  consequence of \rif{komzet0} and is
\eqn{omegazt} 
$$
\sup_{x\in B_{r}}\omega\left(\snr{\tu(x)-(\tu)_{B_{r}}}\right)\le cr^{\beta\gamma}\leq  cr^{\beta/2}\;,
$$
for \ccc{$c\equiv c(\data, \nr{a}_{L^{\infty}},\beta, \gamma, \tilde \Omega)$} (here recall that $\gamma \geq  1/2$ and $r \leq 1$). Proceeding as for \rif{omegav}, and recalling \rif{046final}, we instead have 
\begin{eqnarray}\label{omegazt2}
\mint_{B_{r}}\omega\left(\snr{\tv-(\tv)_{B_{r}}}\right) \, dx
&\leq &c\omega \left[r\left(H^{-}_{B_{r}}\right)^{-1}\left(\mint_{B_{r}}H(x,D\tv) \, dx\right)\right]\notag \\ &\leq & c\omega \left[r\left(H^{-}_{B_{r}}\right)^{-1}\left(\mint_{B_{r}}H(x,D\tu) \, dx\right)\right]\notag \\ &\stackleq{holcac} &  c\omega \left[r\left(H^{-}_{B_{r}}\right)^{-1}\left(cH^{-}_{B_{r}}\left(r^{\gamma-1}\right)\right)\right] \leq cr^{\beta\gamma}\leq cr^{\beta/2}\;,\label{oppi}
\end{eqnarray}
with \ccc{$c\equiv c(\data, \nr{a}_{L^{\infty}},\beta, \gamma, \tilde \Omega)$}. 
In a totally similar fashion, \cc{as done in} \rif{omegauv}, we have 
\eqn{omegazt3}
$$
\omega\left(\snr{(\tu)_{B_{r}}-(\tv)_{B_{r}}}\right) \le cr^{\beta\gamma}\leq cr^{\beta/2}\;,
$$
where \ccc{$c\equiv c(\texttt{data},\nr{a}_{L^{\infty}}, \beta,\gamma,\tilde \Omega)$}. We are now ready to estimate terms $(\mathrm{I})_{1}$-$(\mathrm{I})_{4}$. We have
$$
(\mathrm{I})_{1}\leq \mint_{B_{r}}\omega\left(\snr{\tu-(\tu)_{B_{r}}}\right)H(x,D\tu) \ dx \stackleq{omegazt} cr^{\beta/2}\mint_{B_{2r}}H(x,D\tu) \ dx\;,
$$
with \ccc{$c\equiv c(\texttt{data},\nr{a}_{L^{\infty}},\beta, \gamma,\tilde \Omega)$. The m}inimality of $\tu$ gives
$
(\mathrm{I})_{2}\le 0.
$
\cc{As for term} $(\mathrm{I})_{3}$, using \cd{\rif{046final2} and \rif{oppi}}, we get
\begin{eqnarray*}
\notag (\mathrm{I})_{3}&\leq &c\left(\mint_{B_{r}}\omega\left(\snr{\tv-(\tv)_{B_{r}}}\right) \ dx\right)^{\frac{\sigma_{g}}{1+\sigma_{g}}}\left(\mint_{B_{r}}[H(x,D\tv)]^{1+\sigma_{g}} \ dx\right)^{\frac{1}{1+\sigma_{g}}}\\ 
&\stackleq{046final2} &c\left(\mint_{B_{r}}\omega\left(\snr{\tv-(\tv)_{B_{r}}}\right) \ dx\right)^{\frac{\sigma_{g}}{1+\sigma_{g}}}\mint_{B_{2r}}H(x,D\tu) \ dx\notag \\
&\stackleq{oppi}& cr^{\frac{\beta\sigma_{g}}{2(1+\sigma_{g})}}\mint_{B_{2r}}H(x,D\tu) \ dx\;,
\end{eqnarray*}
for \ccc{$c\equiv c(\texttt{data},\nr{a}_{L^{\infty}},\beta, \gamma,\tilde \Omega)$} and, in a totally similar fashion, arguing as  for \rif{053}, but using this time \rif{omegazt3}, we find 
$$
(\mathrm{I})_{4}\stackleq{omegazt}cr^{\frac{\beta\sigma_{g}}{2(1+\sigma_{g})}}\mint_{B_{2r}}H(x,D\tu) \ dx\;,
$$
where, again it is \ccc{$c\equiv c(\texttt{data},\nr{a}_{L^{\infty}},\beta, \gamma,\tilde \Omega)$}. Collecting the estimates found above for the terms $(\mathrm{I})_{1}$-$(\mathrm{I})_{4}$ to \rif{buy} yields
$$
\mint_{B_{r}}\left(\snr{V_{p}(D\tu)-V_{p}(D\tv)}^{2}+a(x)\snr{V_{q}(D\tu)-V_{q}(D\tv)}^{2} \right)\ dx\le cr^{\frac{\beta\sigma_{g}}{2(1+\sigma_{g})}}\mint_{B_{2r}}H(x,D\tu) \ dx\;,
$$
where \ccc{$c\equiv c(\texttt{data},\nr{a}_{L^{\infty}},\beta,\gamma,\tilde \Omega)$}. In particular, there holds
\eqn{gr4}
$$
\cd{\mint_{B_{r}}\left(\snr{V_{p}(D\tu)-V_{p}(D\tv)}^{2}+\ai(B_{r})\snr{V_{q}(D\tu)-V_{q}(D\tv)}^{2} \right)\ dx\le cr^{\frac{\beta\sigma_{g}}{2(1+\sigma_{g})}}\mint_{B_{2r}}H(x,D\tu) \ dx\;}.
$$
We now select $x_m \in \cd{\bar{B}_{r}}$ such that 
\eqn{il punto}
$$a(x_m)=\ai(B_{r})$$ and define (keep in mind the notation in \rif{strutturaher})
\eqn{integrandoG}
$$
G_{m}(z):= G\left(x_m,(\tu)_{B_{r}},z\right)\equiv \tilde G_{x_m,(\tu)_{B_{r}}}(|z|) \qquad \mbox{for every}\,  z \in \er^{(N-1)\times n}\;.
$$
The newly defined integrand $G_m(\cdot)$ is of the type $g_0(\cdot)$ considered in Lemma \ref{harg} and satisfies assumptions \rif{assg0} with the choice \cd{$H_0(\cdot)\equiv H_{B_{r}}^{-}(\cdot)$} and for suitable constants $\nu, L$ depending on $\data$ (see the discussion in Step 4 and, in particular, \rif{assG} and \rif{bilanciocrescite}). 
We now proceed applying Lemma \ref{harg} to $\tv$; notice that \rif{maggiorefadopo} is automatically satisfied by the second inequality in \rif{046final} \cc{and therefore verify \rif{harg2}}. According to the terminology adopted in \cite{BCM3,CM1,CM2} the $p$-phase occurs if
\eqn{pph}
$$
\cd{\ai(B_{r})\le 4[a]_{0,\alpha}r^{\alpha-s},\quad s:=\alpha+(\gamma-1)(q-p)\stackrel{\rif{mainbound}_{1}}{>}0\;},
$$
while the $(p,q)$-phase is defined by the complementary condition
\begin{flalign}\label{pqph}
\cd{\ai(B_{r})>4[a]_{0,\alpha}r^{\alpha-s}}.
\end{flalign}
It is then easy to see that 
\begin{flalign}\label{a}
\cd{
\begin{cases}
\ \as(B_{r})\le 6[a]_{0,\alpha}r^{\alpha-s}\quad &\mbox{holds in the }p\mbox{-phase}\\
\ \as(B_{r})\le \frac{3}{2}\ai(B_{r})\quad &\mbox{holds in the }(p,q)\mbox{-phase}\;.
\end{cases}}
\end{flalign}
Notice that the two phases described above depend on the number $\gamma \geq 1/2$, which is going to be chosen later as an absolute function of $\data$. We start estimating, for every $\varphi \in C^{\infty}_{c}(B_{r/2},\RN)$ 
\begin{eqnarray}
\notag \left | \ \mint_{B_{r/2}}\partial G_{m}(D\tv)\cdot D\varphi \ dx \ \right | &\stackrel{\eqref{elv}}{=}&\left | \ \mint_{B_{r/2}}\left[\partial G_{m}(D\tv)-\partial G\left(x,(\tu)_{B_{r}},D\tv\right)\right]\cdot D\varphi \ dx \ \right |\\ 
\notag
&\stackrel{\eqref{assG}_4}{\le}&c\nr{D\varphi}_{L^{\infty}(B_{r/2})}\mint_{B_{r/2}}\omega(|x_{m}-x|)\left[\frac{H(x_{m},D\tv)}{\snr{D\tv}}+\frac{H(x,D\tv)}{\snr{D\tv}}\right] \ dx\\
\notag
&&+c\nr{D\varphi}_{L^{\infty}(B_{r/2})}\mint_{B_{r/2}}\left[a(x)-a(x_{m})\right]\snr{D\tv}^{q-1} \ dx\\
\notag
&\stackrel{\eqref{omegabeta}}{\le} & cr^{\beta}\nr{D\varphi}_{L^{\infty}(B_{r/2})}\mint_{B_{r/2}}\snr{D\tv}^{p-1} \ dx\\
\notag
&&+cr^{\beta}\nr{D\varphi}_{L^{\infty}(B_{r/2})}\mint_{B_{r/2}}\cd{\left(a(x)+a_{i}(B_{r})\right)}\snr{D\tv}^{q-1} \ dx\\
\notag
&&+c[a]_{0, \alpha}r^{\alpha}\nr{D\varphi}_{L^{\infty}(B_{r/2})}\mint_{B_{r/2}}\snr{D\tv}^{q-1} \ dx\\ &=:&(\mathrm{II})_{1}+(\mathrm{II})_{2}+(\mathrm{II})_{3}\label{found}\;,
\end{eqnarray}
and proceed estimating the terms appearing in the last three lines. In any case we have, 
Young's inequality gives
\eqn{gr6}
$$
(\mathrm{II})_{1}\le  c(\texttt{data})r^{\beta}\mint_{B_{r}}\left( \snr{D\tv}^{p}+\nr{D\varphi}^{p}_{L^{\infty}(B_{r/2})} \right)\ dx\;.
$$
In order to bound the remaining two terms we distinguish between the $p$-phase \eqref{pph} and the $(p,q)$-phase \eqref{pqph}. We start noticing that in the $p$-phase we have
\eqn{gr5}
$$
\mint_{B_{r/2}}H(x,D\tv) \ dx \stackleq{046final} c\mint_{B_{r}}H(x,D\tu) \ dx\stackleq{holcac}cH^{-}_{B_{r}}(r^{\gamma-1})\stackleq{pph}cr^{(\gamma-1)p}\;,
$$
for \ccc{$c\equiv c(\texttt{data},\nr{a}_{L^{\infty}},\beta, \gamma,\tilde \Omega)$}. As $q-1<p$, using H\"older's inequality we have
\begin{eqnarray}
(\mathrm{II})_{2}&\stackrel{\eqref{a}_{1}}{\le} &cr^{\beta+\alpha-s}\nr{D\varphi}_{L^{\infty}(B_{r/2})}\left(\mint_{B_{r/2}}\snr{D\tv}^{p} \ dx\right)^{\frac{q-p}{p}}\left(\mint_{B_{r/2}}\snr{D\tv}^{p} \ dx\right)^{\frac{p-1}{p}}\nonumber \\
&\leq  &cr^{\beta+\alpha-s}\nr{D\varphi}_{L^{\infty}(B_{r/2})}\left(\mint_{B_{r/2}}H(x, D\tv) \ dx\right)^{\frac{q-p}{p}}\left(\mint_{B_{r/2}}\snr{D\tv}^{p} \ dx\right)^{\frac{p-1}{p}}\nonumber \\&
\stackleq{gr5}&cr^{\beta+\alpha-s+(\gamma-1)(q-p)}\nr{D\varphi}_{L^{\infty}(B_{r/2})}
\left(\mint_{B_{r/2}}\snr{D\tv}^{p} \ dx\right)^{\frac{p-1}{p}}\notag \\ &\stackleq{pph} & cr^{\beta}\mint_{B_{r}} \left(\snr{D\tv}^{p}+\nr{D\varphi}^{p}_{L^{\infty}(B_{r/2})} \right)\ dx\;,\notag
\end{eqnarray}
for \ccc{$c\equiv c(\texttt{data},\nr{a}_{L^{\infty}},\beta, \gamma,\tilde \Omega)$}. Finally, we similarly have
\begin{eqnarray*}%\label{gr8}
(\mathrm{II})_{3}&\le &cr^{s}r^{\alpha-s}\nr{D\varphi}_{L^{\infty}(B_{r/2})}\left(\mint_{B_{r/2}}\snr{D\tv}^{p} \ dx\right)^{\frac{q-p}{p}}\left(\mint_{B_{r/2}}\snr{D\tv}^{p} \ dx\right)^{\frac{p-1}{p}}\notag \\ &\stackrel{\rif{pph}, \rif{gr5}}{\leq}&cr^{s}\mint_{B_{r}}\left( \snr{D\tv}^{p}+\nr{D\varphi}^{p}_{L^{\infty}(B_{r/2})} \right)\ dx\;,
\end{eqnarray*}
with \ccc{$c\equiv c(\texttt{data},\nr{a}_{L^{\infty}},\beta, \gamma,\tilde \Omega)$}. We now consider the occurrence of the \cd{$(p,q)$-phase \rif{pqph}}. We have, again by H\"older's inequality
\begin{eqnarray*}
(\mathrm{II})_{2}&\stackrel{\eqref{a}_{2}}{\le}&cr^{\beta}\nr{D\varphi}_{L^{\infty}(B_{r/2})}\left[\ai(B_{r})\right]^{\frac{1}{q}}\mint_{B_{r}}\left[\ai(B_{r})\right]^{\frac{q-1}{q}}\snr{D\tv}^{q-1} \ dx\nonumber\\
&\le &cr^{\beta}\nr{D\varphi}_{L^{\infty}(B_{r/2})}\left[\ai(B_{r})\right]^{\frac{1}{q}}\left(\mint_{B_{r}}\ai(B_{r})\snr{D\tv}^{q} \ dx\right)^{\frac{q-1}{q}}\nonumber\\
&\le &cr^{\beta}\mint_{B_{r}}\left( \ai(B_{r})\snr{D\tv}^{q}+\ai(B_{r})\nr{D\varphi}_{L^{\infty}(B_{r/2})}^{q} \right) \ dx
\end{eqnarray*}
and
\begin{eqnarray*}
\notag (\mathrm{II})_{3}&=&c[a]_{0, \alpha}r^{\alpha-s}r^{s}\nr{D\varphi}_{L^{\infty}(B_{r/2})}\mint_{B_{r/2}}\snr{D\tv}^{q-1} \ dx\\&\stackleq{pqph} &cr^{s}\nr{D\varphi}_{L^{\infty}(B_{r/2})}\left[\ai(B_{r})\right]^{\frac{1}{q}}\mint_{B_{r}}\left[\ai(B_{r})\right]^{\frac{q-1}{q}}\snr{D\tv}^{q-1}\ dx\notag \\ &\le& cr^{s}\mint_{B_{r}}\left( \ai(B_{r})\snr{D\tv}^{q}+\ai(B_{r})\nr{D\varphi}_{L^{\infty}(B_{r/2})}^{q} \right) \ dx\;,
\end{eqnarray*}where \ccc{$c\equiv c(\texttt{data},\beta)$}. Collecting the estimates founds for the terms $(\mathrm{II})_{1}, (\mathrm{II})_{2}, (\mathrm{II})_{3}$ to \rif{found} and recalling \rif{pph}, in any case we conclude with
\begin{flalign*}
\left | \ \mint_{B_{r/2}}\partial G_{m}(D\tv)\cdot D\varphi \ dx \ \right |\le cr^{\min\left\{\beta,\alpha-(q-p)/2\right\}}\mint_{B_{r}}\left[H^{-}_{B_{r}}(D\tv)+H^{-}_{B_{r}}\left(\nr{D\varphi}_{L^{\infty}(B_{r/2})}\right) \right]\ dx\;,
\end{flalign*}
with \ccc{$c\equiv c(\texttt{data},\nr{a}_{L^{\infty}},\beta, \gamma,\tilde \Omega)$}. Here we have used $\gamma\geq 1/2$, so that it is $s\ge \alpha-(q-p)/2>0$. Lemma \ref{harg} (notice that the number $N$ used there is actually $N-1$ in this context; recall again that here it is $N>1$) yields the existence of $\tilde{h}\in W^{1,H^{-}_{B_{r}}}_{\tilde v}(B_{r/2},\mathbb{R}^{N-1})$ satisfying 
$$
\mint_{B_{r/2}}\partial G_{m}(D\tilde h)\cdot D\varphi \ dx=0 \quad \mbox{for every $\varphi\in W^{1,H^{-}_{B_{r}}}_0(B_{r/2},\mathbb{R}^{N-1})$}\;,
$$
and, such that
\eqn{gr1300}
$$
\cd{\mint_{B_{r/2}}\left(\snr{V_{p}(D\tv)-V_{p}(D\tilde{h})}^{2}+\ai(B_{r})\snr{V_{q}(D\tv)-V_{q}(D\tilde{h})}^{2}\right) \ dx\le cr^{m}\mint_{B_{r}}H_{B_r}^{-}(D\tilde v)\ dx\;,}
$$
where \ccc{$c\equiv c(\texttt{data},\nr{a}_{L^{\infty}},\beta,\gamma,\tilde \Omega)$} and \cc{$m=m(n,N,\nu,L,p,q,\alpha)$}. Needless to say, $\tilde h$ soves 
\eqn{ilminimo}
$$ \tilde h  \mapsto  \min_{w\in W^{1,H_{B_{r}}^{-}}_{\tilde v}(B_{r/2},\er^{N-1})}\mint_{B_{r/2}}G_{m}(D w) \ dx\;$$
so that, recalling again \rif{il punto}, we find
\eqn{sososo}
$$
\mint_{B_{r/2}}H_{B_r}^{-}(D\tilde h) \ dx \leq c \mint_{B_{r}}H_{B_r}^{-}(D\tilde v) \ dx \leq 
\mint_{B_{r/2}}H(x,D\tilde v) \ dx \stackleq{046final} c \mint_{B_{r}}H(x,D\tu) \ dx\;.
$$
Hence, from \eqref{gr4}, \rif{gr1300}, and the last inequality in the above display, we obtain
\eqn{gr11}
$$
\cd{\mint_{B_{r/2}}\left(\snr{V_{p}(D\tu)-V_{p}(D\tilde{h})}^{2}+\ai(B_{r})\snr{V_{q}(D\tu)-V_{q}(D\tilde{h})}^{2} \right)\ dx \le cr^{\kappa}\mint_{B_{2r}}H(x,D\tu) \ dx\;},
$$
for $$\kappa:=\min\left\{m,\frac{\beta\sigma_{g}}{2(1+\sigma_{g})}\right\}<1$$ and \ccc{$c \equiv c(\texttt{data},\nr{a}_{L^{\infty}},\beta, \gamma,\tilde \Omega)$}. After some standard manipulations on \rif{gr11}, see e.g. \cite[Section 10]{BCM3} and \cite[pp. 483]{CM1}, we have 
\begin{flalign}\label{gr12}
\mint_{B_{r/2}}\snr{D\tu-D\tilde{h}}^{p} \ dx \le \mint_{B_{r/2}}H^{-}_{B_{r}}\left(D\tu-D\tilde{h}\right) \ dx \le cr^{\kappa/2}\mint_{B_{2r}}H(x,D\tu) \ dx\;,
\end{flalign}
for \ccc{$c\equiv c(\texttt{data},\nr{a}_{L^{\infty}},\beta, \gamma,\tilde \Omega)$}. Now we make a further restriction on the size of $\gamma$ imposing that 
\eqn{previous}
$$\gamma\geq 1-\frac\kappa{4q}$$ 
(which is still larger than $1/2$), and apply \eqref{holcac}; we use the resulting inequality in  \eqref{gr12} to obtain
\begin{flalign}\label{gr13}
\mint_{B_{r/2}}\snr{D\tu-D\tilde{h}}^{p} \ dx \le cr^{\kappa/4}\;,
\end{flalign}
where \ccc{$c\equiv c(\texttt{data}, \nr{a}_{L^{\infty}},\beta,\gamma,\tilde \Omega)$}. By using the content of Remark \ref{spiega} below we have that
\begin{flalign}\label{dec}
\mint_{B_{\rr}}H^{-}_{B_{r}}\left(D\tilde{h}-(D\tilde{h})_{B_{\rr}}\right) \ dx\le c\left(\frac{\rr}{r}\right)^{\mu}\mint_{B_{r}}H(x,D\tu) \ dx\;,
\end{flalign}
holds for concentric balls $B_{\rr}\subset B_{r/2}$; here we take $\rr \leq r/8$. Here it is  $c\equiv c(n,N,\nu,L,p,q)$ and \cd{$\mu\equiv \mu(n,N,\nu,L,p,q,\beta_{1})$}. We estimate
\begin{eqnarray}\label{gr14}
\notag \mint_{B_{\rr}}\snr{D\tu-(D\tu)_{B_{\rr}}}^{p} \ dx &\leq & c\mint_{B_{\rr}}\snr{D\tu-(D\tilde h)_{B_{\rr}}}^{p} \ dx\\ &\stackleq{gr13} &c\left\{\left(\frac{r}{\rr}\right)^{n}r^{\kappa/4}+\mint_{B_{\rr}}H^{-}_{B_{r}}\left(D\tilde{h}-(D\tilde{h})_{B_{\rr}}\right) \ dx\right\}\nonumber \\
&\stackleq{dec}&c\left\{\left(\frac{r}{\rr}\right)^{n}r^{\kappa/4}+\left(\frac{\rr}{r}\right)^{\mu}\mint_{B_{r}}H(x,D\tu) \ dx\right\}\notag \\ &\stackleq{holcac}&c\left\{\left(\frac{r}{\rr}\right)^{n}r^{\kappa/4}+\left(\frac{\rr}{r}\right)^{\mu}r^{(\gamma-1)q}\right\}\;,
\end{eqnarray}
where \ccc{$c\equiv c(\texttt{data}, \nr{a}_{L^{\infty}},\beta,\gamma,\tilde \Omega)$}. In \eqref{gr14} we pick
$$\rr=\frac{r^{1+a}}{8} \quad \mbox{with} \quad a:=\frac{(1-\gamma)q+\kappa/4}{\mu+n} \quad \mbox{and}\quad \gamma = 1-\frac{ \kappa\mu}{8nq}\;,$$ ($\gamma$ also meets the condition in \rif{previous}) so that \rif{gr14} yields
\eqn{ilbeta}
$$
\mint_{B_{\rr}}\snr{D\tu-(D\tu)_{B_{\rr}}}^{p} \ dx\le c \rr^{\frac{\kappa/4-na}{1+a}}\leq c\rr^{\beta_{0}p}\;, \ \ \qquad \beta_{0}:=\frac{\kappa\mu}{16p(n+\mu)}\;.
$$
The classical H\"older continuity characterization of Campanato and Meyers, a standard covering argument, and the fact that $\tilde \Omega\Subset \Omega_u$ is arbitrary, allow to conclude that $D\tu\in C^{0,\beta_{0}}_{\mathrm{loc}}\left(\Omega_{u},\mathbb{R}^{(N-1)\times n}\right)$, with $\beta_{0}$ as in \rif{ilbeta} and therefore depends on \ccc{$(\texttt{data},\beta,\beta_{1})$}. Finally, using $\eqref{S}_{1,2}$ we get $Du= D(S(\tilde u))\in C^{0,\beta_{0}}_{\mathrm{loc}}\left(\Omega_{u},\mathbb{R}^{N\times n}\right)$ and the proof of the partial local H\"older continuity of the gradient as stated in \rif{mainassertion} is complete in the case it is $p(1+\delta_{g})\leq n$.\\\\
\emph{Step 6: The case $p(1+\delta_{g})>n$.}

In this case the singular set is empty $\Sigma_u=\Omega_u=\Omega$ as the right-hand side in \rif{ss} is empty. Therefore we see from Step 3 that $u \in C^{0, \gamma}_{\loc}(\Omega, \SN)$ for every $\gamma <1$ and the rest of the proof, i.e., the local H\"older continuity of $Du$ follows as for the case when $p(1+\delta_{g})\leq n$. 
\begin{remark}\label{spiega}
\emph{We briefly explain how to get estimate \rif{dec} from the results of \cite{DSV2}. Recalling the notation in \rif{integrandoG}, by \rif{equivalentiG} we can argue exactly as in Remark \ref{discuss-re} to get that 
\eqn{decayvv0}
$$
\sup_{B_{r/4}}\, H^{-}_{B_{r}}\left(D\tilde{h}\right)  \leq c\mint_{B_{r/2}}H^{-}_{B_{r}}\left(D\tilde{h}\right) \,dx
$$
holds for a constant \cd{$c\equiv c(n,N,\nu, L,p,q)$}. Moreover, by \rif{anche G} we are able to satisfy \cite[Assumption 2.2]{DSV}, where we can take $\varphi(\cdot) \equiv \tilde G_{x_m,(\tu)_{B_{r}}}(\cdot)$
and therefore (also taking into account the definitions in \rif{vpvq} and in \cite[(1.3)]{DSV}) we can apply \cite[Theorem 6.4]{DSV} that in the present setting gives
\begin{eqnarray}
&&\notag \mint_{B_{\rr}}\left(\snr{V_{p}(D\tilde h)-(V_{p}(D\tilde h))_{B_{\rr}}}^{2}+\ai(B_{r})\snr{V_{q}(D\tilde h)-(V_{q}(D\tilde h))_{B_{\rr}}}^{2} \right)\ dx\\ &&
\qquad  \leq c\left(\frac{\rr}{r}\right)^{2\mu}
\mint_{B_{r/2}}\left(\snr{V_{p}(D\tilde h)-(V_{p}(D\tilde h))_{B_{r/2}}}^{2}+\ai(B_{r})\snr{V_{q}(D\tilde h)-(V_{q}(D\tilde h))_{B_{r/2}}}^{2}\right) \ dx\;,
\label{decayvv01}
\end{eqnarray}
whenever $B_{\varrho} \subset  B_{r/2}$ is concentric to $B_{r/2}$, where $c\geq 1$ and $\mu \in (0, 1/2)$ are both depending on $n,N,p,q,\ratio $ and $\beta_{1}$. Following the method explained in \cite[Theorem 3.1]{BCM2}, \ccc{see also \cite[Proposition 3.3]{meom}}, and combining \rif{decayvv0}-\rif{decayvv01} with \rif{V}, finally yields
$$
\mint_{B_{\rr}}H^{-}_{B_{r}}\left(D\tilde{h}-(D\tilde{h})_{B_{\rr}}\right) \ dx\le c\left(\frac{\rr}{r}\right)^{\mu}\mint_{B_{r/2}}H^{-}_{B_{r}}(D\tilde h) \ dx\;,
$$
from which \rif{dec} obviously follows. 
}
\end{remark}

\section{Weighted Hausdorff measures and singular sets}\label{cap-proofs}
\subsection{Proof of Proposition \ref{triviaros1}}
Observe that, for a ball $B$ such that $\beta_4^{-1}B\subset \Omega$ and $r(B)\leq 1$ we then have
\begin{eqnarray*}
\snr{\beta_4^{-1}B}\supess_{x\in \beta_4^{-1}B}\Phi(x,1/r(\beta_4^{-1}B))&=&\frac{\snr{B}}{\beta_4^{n}}\supess_{x\in \beta_4^{-1}B}\Phi(x,\beta_4/r(B))\\
&\stackleq{controllo} & \frac{c_{d}}{\beta_4^{n}}\snr{B}\infess_{x\in \beta_4^{-1}B}\Phi(x,1/r(B))\le \frac{c_{d}}{\beta_4^{n}}\snr{B}\infess_{x\in B}\Phi(x,1/r(B))\;.
\end{eqnarray*}
By taking balls $B$ such that \cc{$r(B)\leq \kappa$ and $\kappa \leq 1$, we find
$
\mathcal{H}_{\Phi, \kappa/\beta_4}^{+} \leq (c_d/\beta_4^{n}) \mathcal{H}_{\Phi, \kappa}^{-}, 
$}
from which \rif{othercontrol} follows by letting \cc{$\kappa \to 0$}. As for the proof of \rif{othercontrol2}, we observe that in this case it is $\Phi (x, t)= [t^p +a(x)t^q]^{1+\sigma}\approx t^{p(1+\sigma)} +[a(x)]^{1+\sigma}t^{q(1+\sigma)}$, \cc{with constants implicit in "$\approx$" depending on $\sigma$}. Then, for $B \subset \Omega$ and $1 \leq t \leq 1/r(B)$, we have
\begin{eqnarray}
\notag \supess_{x \in B}\, \Phi\left(x,  t\right)&\leq & \infess_{x \in B}\, \Phi\left(x, t \right) + \left\{[\as(B)]^{1+\sigma}-[\ai(B)]^{1+\sigma}\right\}t^{q(1+\sigma)}\\
\notag &\leq & \infess_{x \in B}\, \Phi\left(x, t \right) +c [r(B)]^{\alpha(1+\sigma)}t^{q(1+\sigma)} \\ &\leq & 
 \infess_{x \in B}\, \Phi\left(x, t \right) +c \ccc{[r(B)]^{(\alpha-q+p)(1+\sigma)}t^{p(1+\sigma)}} \leq  c\infess_{x \in B}\, \Phi\left(x, t \right) \label{condsi}
\end{eqnarray}
where $c \equiv c ( [a]_{0, \alpha},\sigma)$. Therefore \rif{controllo} is satisfied for $\beta_4=1$ and \cd{assertion \rif{othercontrol2} follows by \rif{othercontrol}}. 
\subsection{Proof of Theorem \ref{L11}}
The proof is a suitable modification of the one which is valid for the standard $W^{1,p}$-capacity.  Thanks to the Choquet property \rif{cho}, we can reduce to the case when $E$ is a compact subset. Therefore, recalling that ${Cap}_{\Phi}^*(K)={Cap}_{\Phi}(K)$ whenever  $K \subset \Omega$ is a compact subset, we can then compute \cd{${Cap}_{\Phi}(E)$} via \rif{cappi0}. We now claim that there exists a positive constant $c$, essentially depending on $E$, such that if $V$ is a bounded open set such that $E\Subset V \subset \Omega$, then there exists an open set $W$ and a function $f\in \mathcal{R}(E)$ with the following features:
\begin{flalign}\label{indu}
\begin{cases}
 \ E\subset W\subset \left\{ x \in \Omega\colon f(x)=1 \right\}\,,\ \ 
 \ \supp\, f\subset V\;, f \in C_0(\Omega);\\
 \ \displaystyle \int_{\Omega} \f(x,\snr{Df}) \, dx <c\;.
\end{cases}
\end{flalign}
Let $V\subset \Omega$ be an open set as above and fix \cd{$\kappa=\frac{1}{4}\min\left\{\dist(E,\mathbb{R}^{n}\setminus V),1\right\}$}. Since $\mathcal{H}_{\f}(E)<\infty$ and $E$ is compact, there exists a positive integer $m=m(E)$ and a finite collection of open balls $\{B_{r_{j}}(x_{j})\}_{j\leq m}\in \cd{\mathcal{C}^{\kappa/2}_{E}}$ such that $\{x_{j}\}_{j\leq m}\subset E$, $B_{2r_{j}}(x_{j})\Subset \Omega$ for all $j \in \{1,\cdots,m\}$ and
\begin{flalign}\label{086}
B_{r_{j}}(x_{j})\cap E \not = \emptyset \ \mathrm{for \ all \ }j \in \{1,\cdots, m\} \quad \mathrm{and} \quad \sum_{j=1}^{m}\int_{B_{r_{j}}(x_{j})}\f\left(x,1/r_{j}\right) \, dx\le 2\left[\mathcal{H}_{\f}(E)+1\right]\;. 
\end{flalign}
We introduce $W:=\bigcup_{j=1}^{m}B_{r_{j}}(x_{j})$ and the maps $f_j$, defined on the whole $\mathbb{R}^{n}$, such that $f_{j}(x):=1$ if $\snr{x-x_{j}}\le r_{j}$, $f_{j}(x):=2- \snr{x-x_{j}}/r_j$ if
$r_{j}<\snr{x-x_{j}}\le 2r_{j}$ and $f_{j}(x):=0$ if $2r_{j}<\snr{x-x_{j}}$. With $\beta_4 \in (0,1)$ being the constant appearing in \rif{controllo},  we have 
\begin{eqnarray*}
\int_{\Omega}\f(x,\snr{Df_{j}}) \, dx& \le & \int_{B_{2r_{j}}(x_{j})}\f\left(x,1/r_{j}\right) \, dx\notag \\ 
&\stackleq{0777-2}  & \frac{c_g}{\beta_4^q}\int_{B_{2r_{j}}(x_{j})}\f\left(x,\beta_4/r_{j}\right) \, dx\nonumber\\ &\leq & \cd{\frac{2^{n}c_g}{\beta_4^q}}\snr{B_{r_{j}}(x_{j})}\cc{\supess_{x \in B_{2r_{j}}(x_{j})}\, \Phi\left(x, \beta_4/r_{j}\right)} \nonumber\\ &\stackleq{controllo} & \cd{\frac{2^{n}c_gc_d}{\beta_4^q}}\snr{B_{r_{j}}(x_{j})}\cc{\infess_{x \in B_{r_{j}}(x_{j})}\, \Phi\left(x, 1/r_{j}\right)}
\leq  \cd{\frac{2^{n}c_gc_d}{\beta_4^q}}\int_{B_{r_{j}}(x_{j})}\f\left(x,1/r_{j}\right) \, dx\;.
\end{eqnarray*}
In particular, it follows that $f_{j}\in W^{1,\Phi}(\Omega)$. We then set $$f:=\max_{j \in \{1,\cdots, m\}}f_{j}\;,$$ which is continuous, as every $f_{j}$ is. Moreover, if $x\in W$, then $x \in B_{r_{j}}(x_{j})$ for some $j \in \{1,\cdots,m\}$ and, as a consequence, $f_{j}(x)=1$;  thus $f(x)=1$, so $f(W)\equiv \{1\}$ and using the content of the last display, the lattice property of Sobolev functions, see \cite[Theorem 1.20]{HKM} and \rif{086}, we obtain
\begin{eqnarray}
\notag \int_{\mathbb{R}^{n}}\f(x,\snr{Df}) \, dx & \le & \sum_{j=1}^{m}\int_{\Omega}\f(x,\snr{Df_{j}}) \, dx \\ &\le &  c\sum_{j=1}^{m}\int_{B_{r_{j}}(x_{j})}\f\left(x,1/r_{j}\right) \, dx\stackleq{086} c \left[\mathcal{H}_{\f}(E)+1\right]\;,\notag
\end{eqnarray}
with \cd{$c= c (n,\tilde c_g,c_{d},\beta_4, q)$}, hence $f \in W^{1,\f}(\Omega)\cap C_{0}(\Omega)$. Finally, observe that  if \cc{$x\in \mathbb{R}^{n}\setminus V$}, then $\cc{\snr{x-x_{j}}\ge \dist(E,\mathbb{R}^{n}\setminus V)\ge 4\kappa}$, so $f(x)=0$ and $\supp f\subset V$. This completes the proof of \rif{indu}. Using the above construction inductively, for any $k \in \N$ we find a collection of open sets $\{V_{k}\}_{k}$, with $V_{0}=\emptyset$ and a sequence of functions $\{\tilde f_{k}\}_{k}$ such that
\begin{flalign}\label{089}
\begin{cases}
 \ E\subset V_{k+1}\subset V_{k}\;, \ \ \overline{V}_{k+1}\subset \{x \in \Omega\colon \ \tilde f_{k}(x)=1\}\;, \ \ 
  \supp\, \tilde f_{k}\subset V_{k}\\
\displaystyle  \ \int_{\Omega} \f(x,\snr{D\tilde f_{k}})  \, dx \le c_*\;,
\end{cases}
\end{flalign}
with $c_*$ being independent of $k \in \en$. For $j \in \N$, define $$S_j=\sum_{k=1}^{j}\frac 1k \qquad \mbox{and}\qquad g_{j}=\frac{1}{S_{j}}\sum_{k=1}^{j}\frac{\tilde f_{k}}{k}\;.$$ From the above discussion, $\tilde f_{j}$ belongs to $\mathcal{R}(E)$ for every $j$, so, by construction, $g_{j}\in \mathcal{R}(E)$, and, given that $\supp\,  \snr{D\tilde f_{k}}\subset V_{k}\setminus \overline{V}_{k+1}$, we find
\begin{flalign*}
{Cap}_{\f}(E)\le& \int_{\mathbb{R}^{n}}\f(x,\snr{Dg_{j}}) \, dx=\sum_{k=1}^{j}\int_{V_{k}\setminus V_{k+1}}\f(x,S_j^{-1}k^{-1}\snr{D\tilde f_{k}}) \, dx\stackrel{ \eqref{0777-2},\eqref{089}}{\leq} \frac {c_*\tilde{c}_g}{S_{j}^{p}}\sum_{k=1}^{j}\frac 1{k^{p}}\to0 
\end{flalign*}
as $j\to \infty$, because $p>1$. The proof is complete. 
\subsection{Proof of Theorem \ref{T3}}
The assertions \rif{zerop}-\rif{zeroq} have already been proved in Step 3 from the proof of Theorem \ref{main1}; see \rif{dimensione-rid1}-\rif{dimensione-rid2}. We therefore proceed with the proof of \rif{zerom}. We abbreviate 
$\Sigma_{u}=\Sigma_{u}^p\cup \Sigma_{u}^q$ where $\Sigma_{u}^p:=\Sigma_{u}\cap \{x_{0}\in \Omega \colon a(x_{0})=0\}$ and $\Sigma_{u}^q:=\Sigma_{u}\cap \{x_{0}\in \Omega \colon a(x_{0})>0\}$. 
It is therefore sufficient to \cc{show}
\eqn{ledue}
$$
\mathcal{H}_{H^{1+\delta_g}}(\Sigma_{u}^{p})=0\quad \mathrm{and}\quad \mathcal{H}_{H^{1+\delta_g}}(\Sigma_{u}^{q})=0\;.
$$
The implication concerning the capacity in \rif{zerom} will then be a consequence of Theorem \ref{L11}. 
To prove \rif{ledue}, we use \rif{othercontrol2} from Proposition \ref{triviaros1}, that gives, in particular, that $\mathcal{H}_{H^{1+\delta_g}}  \lesssim \mathcal{H}_{H^{1+\delta_g}}^{-}$. On the other hand, by the very definition of $\Sigma^p_u$, we have that 
\eqn{unah}
$$
\mathcal{H}_{H^{1+\delta_g}}^{-}(\Sigma^p_u)\lesssim \mathcal{H}^{n-p-p\delta_g}(\Sigma^p_u)\;.$$
Indeed, taking a covering from \cd{$\mathcal{C}_{\Sigma_u^p}^{\kappa}$} for any \cd{$\kappa\in (0,1)$} as in \rif{077}, we see that every ball $B$ of the covering (that for obvious reasons can be assumed to touch $\Sigma_u^p$) is such that $\ai(B)=0$. Therefore $\mathcal{H}_{H^{1+\delta_g}}^{-}(\Sigma^p_u)$ is equivalent to the $(n-p-p\delta_g)$-dimensional  spherical Hausdorff measure of $\Sigma_u^p$ and \rif{unah} is proved. We conclude that the first equality in \rif{ledue} follows from the already proved fact that $\mathcal{H}^{n-p-p\delta_g}(\Sigma^p_u)=0$. We now prove the second equality in \rif{ledue}. For this we \cc{show} that $
\mathcal{H}_{H^{1+\delta_g}}(\Sigma_{u,m}^{q})=0$ for every integer $m$, where $\Sigma_{u,m}^{q}:= \Sigma_{u}^{q}\cap \{a(x)>1/m\}\ $. The key observation is that there exists a positive number \cd{$\tilde \kappa\equiv \tilde \kappa (\alpha, [a]_{0, \alpha}, m)\in (0,1)$} such that, if \cd{$r(B)<\tilde \kappa $} and $B$ touches $\Sigma_{u,m}^q$, then $2m\ai(B)>1$. Therefore, we take a covering \cd{$\{B_j\}\in \mathcal{C}_{\Sigma_{u,m}^q}^{\kappa}$} (again we assume that each of the balls from the covering is touching $\Sigma_{u,m}^q$) with \cd{$\kappa < \tilde \kappa$}, and estimate as follows
\begin{eqnarray*}
&&\sum_{j \in \N}[r(B_j)]^{n}\left\{[r(B_j)]^{-p}+\ai(B_j)[r(B_j)]^{-q}\right\}^{1+\delta_g}\\ &&\qquad \leq  cm^{1+\delta_g}\sum_{j \in \N}[r(B_j)]^{n}\left\{\ai(B_j)[r(B_j)]^{-p}+\ai(B_j)[r(B_j)]^{-q}\right\}^{1+\delta_g}\\
&&\qquad \leq cm^{1+\delta_g}\|a\|_{L^\infty}^{1+\delta_g}\sum_{j \in \N}[r(B_j)]^{n-q(1+\delta_g)}\;.
\end{eqnarray*}
The above relation implies then that $$ \mathcal{H}_{H^{1+\delta_g}}^{-}(\Sigma_{u,m}^{q})\lesssim  \mathcal{H}^{n-q-q\delta_g}(\Sigma_{u,m}^{q})\leq \mathcal{H}^{n-q-q\delta_g}(\Sigma_{u}^{q})\Longrightarrow \mathcal{H}_{H^{1+\delta_g}}^{-}(\Sigma_{u,m}^{q})=0$$ by \rif{zeroq} for every positive integer $m$. By again appealing to \rif{othercontrol2} we deduce that $\mathcal{H}_{H^{1+\delta_g}}(\Sigma_{u,m}^{q})=0$ for every positive integer $m$ and the proof is complete. 

\vspace{5mm}

{\bf Acknowledgements.} This work is supported by the Engineering and Physical Sciences Research Council [EP/L015811/1]. The authors thank I. Chlebicka and J. Ok for comments on a preliminary version of the manuscript.


\begin{thebibliography}{}


\bibitem {AM2}\name[Acerbi, E.] \& \name[Mingione, G.]: Regularity results for electrorheological fluids: the stationary case. {\em C. R. Math. Acad. Sci. Paris} 334 (2002), 817--822. 


 \bibitem {BCM1}
\name[Baroni, P.] \& \name[Colombo, M.] \& \name[Mingione, G.]: Harnack inequalities for double phase functionals. {\em Nonlinear Anal.} 121 (2015), 206--222. 


 \bibitem {BCM2}
\name[Baroni, P.] \& \name[Colombo, M.] \& \name[Mingione, G.]: Non-autonomous functionals, borderline cases and related function classes.  {\em St. Petersburg Math. J.} 27 (2016), 347--379. 

 \bibitem {BCM3}
\name[Baroni, P.] \& \name[Colombo, M.] \& \name[Mingione, G.]: Regularity for general functionals with double phase. {\em Calc.~Var.~\& PDE} 57 (2018), Paper No. 62, 48 pages. 


\bibitem {BHH}\name[Baruah, D.] \& \name[Harjulehto, P.] \& \name[H\"asto, P.]: Capacities in generalized Orlicz spaces. {\em J. Funct. Spaces} 2018, Article ID 8459874.



\bibitem {BO1}\name[Byun, S.S.] \& \name[Oh, J.]: Global gradient estimates for non-uniformly elliptic equations. {\em Calc.~Var.~\& PDE} 56 (2017), no. 2, Paper No. 46, 36 pp.


\bibitem{Radu1} \name[Cencelj, M.] \& \name[Radulescu, V.] \& \name[Repovš, D.D.]: Double phase problems with variable growth. {\em Nonlinear Anal.} 177 (2018), 270--287. 



\bibitem {CM1}
\name[Colombo, M.] \& \name[Mingione, G.]: Regularity for double phase variational problems. {\em Arch.~Rat.~Mech.~Anal.}  215 (2015), 443--496. 

\bibitem {CM2}
\name[Colombo, M.] \& \name[Mingione, G.]: Bounded minimisers of double phase variational integrals. {\em Arch.~Rat.~Mech.~Anal.} 218 (2015), 219--273. 

\bibitem {CM3}
\name[Colombo, M.] \& \name[Mingione, G.]: Calder\'on-Zygmund estimates and non-uniformly elliptic operators. {\em J. Funct. Anal.} 270 (2016), 1416--1478. 

\bibitem {choquet}
\name[Choquet, G.]: Theory of capacities. {\em Ann. Inst. Fourier, Grenoble} 5 (1955), 131--295. 


\bibitem {cristiana} \name[De Filippis, C.]: Higher integrability for constrained minimizers of integral functionals with $(p,q)$-growth in low dimension. {\em Nonlinear Analysis} 170 (2018), 1--20. 

\bibitem{meom} \name[De Filippis, C.]: On the regularity of the $\omega$-minima of $\varphi$-functionals. \emph{Nonlinear Anal.}, to appear.

\bibitem{me} \name[De Filippis, C.]: Partial regularity for manifold constrained $p(x)$-harmonic maps. \emph{Calc. Var. \& PDE}, 58:47, (2019). \texttt{https://doi.org/10.1007/s00526-019-1483-6}.

\bibitem {cristiana3} \name[De Filippis, C.] \& \name[Mingione, G.]: A borderline case of Calder\'on-Zygmund  estimates for non-uniformly elliptic problems. {\em St. Petersburg Math. J.}, to appear


\bibitem{DOh} \name[De Filippis, C.] \& \name[Oh, J.]: Regularity for multi-phase variational problems. {\em J.~Differential Equations} https://doi.org/10.1016/j.jde.2019.02.015
 

\bibitem{DP} \name[De Filippis, C.] \& \name[Palatucci, G.]: H\"older regularity for nonlocal double phase equations. \emph{J.~Differential Equations}. \texttt{ https://doi.org/10.1016/j.jde.2019.01.017}

 \bibitem {the4} \name[Diening, L.] \& \name[Harjulehto, P.] \& \name[H\"asto, P.] \& \name[\Ruz{}, M.]: {\em Lebesgue and Sobolev spaces with a variable growth exponent}. Springer Lecture notes Math. 2017 (2011). 
 
 \bibitem{DSV2} \name[Diening, L.] \& \name[Stroffolini, B.] \& \name[Verde, A.]: Everywhere regularity of functionals with $\varphi$-growth. {\em Manuscripta Math.} 129 (2012), 449--481. 
 
\bibitem {DSV}
\name[Diening, L.] \& \name[Stroffolini, B.] \& \name[Verde, A.]: The $\varphi$-harmonic approximation and the regularity of $\varphi$-harmonic maps. {\em J. Diff. Equ.} 253 (2012), 1943--1958. 

\bibitem {DGloria} \name[Duerinckx, M.] \& \name[Gloria, A.]: 
Stochastic homogenization of nonconvex unbounded integral functionals with convex growth. {\em 
Arch. Ration. Mech. Anal.} 221 (2016), 1511--1584. 

\bibitem {DM1} \name[Duzaar, F.] \& \name[Mingione, G.]: The $p$-harmonic approximation and the regularity of $p$-harmonic
maps. {\em Calc.~Var.~\& PDE} 20 (2004),
 235--256.




\bibitem {ES} \name[Eells, J.] \& \name[Sampson, J. H.]: 
Harmonic mappings of Riemannian manifolds. {\em  
Amer. J. Math.} 86 (1964), 109--160. 

\bibitem {ELM}
\name[Esposito, L.] \& \name[Leonetti, F.] \& \name[Mingione, G.]:
Sharp regularity for functionals with $(p,q)$ growth. {\em
J.~Diff. Equ.} 204 (2004), 5--55.

\bibitem{EG} \name[Evans, L. C.] \& \name[Gariepy, R. F.]:
{\em Measure theory and fine properties of functions}. CRC Press, (1991).

\bibitem{F} \name[Federer, H.]:
{\em Geometric Measure Theory}. Springer-Verlag, (1969).
%
\bibitem {FMM} \name[Fonseca, I.] \& \name[Mal\'y, J.] \& \name[Mingione, G.]: Scalar minimizers
with fractal singular sets. {\em Arch.~Rat.~Mech.~Anal.} 172
(2004), 295--307.

\bibitem {frehse} \name[Frehse, J.]: Capacity methods in the theory of partial differential equations. {\em Jahresber. Deutsch. Math.-Verein.} 84 (1982), 1--44.



\bibitem{fuchs1} \name[Fuchs, M.]: $p$-harmonic obstacle problems. I: Partial regularity theory.
{\em Ann. Mat. Pura Appl. (IV)} 156 (1990), 127--158.



\bibitem{fuchs3} \name[Fuchs, M.]: $p$-harmonic obstacle problems. III. Boundary regularity. {\em Ann. Mat. Pura Appl. (IV)} 156 (1990), 159--180. 




\bibitem {G} \name[Giusti, E.]: {\em Direct methods in the calculus of
variations.} World Scientific Publishing Co., Inc., River Edge, NJ,
2003.

\bibitem {HL} \name[Hardt, R.M.] \& \name[Lin, F. G.]: Mappings minimizing the $L^p$ norm of the gradient. {\em Comm. Pure Appl. Math.} 40 (1987), no. 5, 555--588. 


\bibitem {HKL} \name[Hardt, R.M.] \&  \name[Kinderlehrer, D.] \& \name[Lin, F. G.]: Stable defects of minimizers of constrained variational principles. {\em Ann.~Inst.~H.~Poincar\'e Anal.~Non Lin\`eaire} 5 (1988), 297--322.

\bibitem {HH} \name[Harjulehto, P.] \& \name[H\"ast\"o, P.]: {\em Orlicz Spaces and Generalized Orlicz Spaces}. Book, 2018. 

\bibitem {HHT} \name[Harjulehto, P.] \& \name[H\"ast\"o, P.] \& \name[Toivanen, O.]: H\"older regularity of quasiminimizers under generalized growth conditions. {\em Calc.~Var.~\& PDE} 56 (2017), Paper No. 22, 26 pp.

\bibitem {Hasto} \name[H\"ast\"o, P.]: The maximal operator on generalized Orlicz spaces. {\em J. Funct. Anal.}
269 (2015), 4038--4048. 

\bibitem{HKM} \name[Heinonen, J.] \& \name[Kilpel\"ainen, T.] \& \name[Martio, O.]: {\em Nonlinear potential theory of degenerate elliptic equations}. Dover Publications Inc., Mineola, NY, 2006. Unabridged
republication of the 1993 original.

\bibitem{hopper} \name[Hopper, C.]: Partial regularity for holonomic minimizers of quasiconvex functionals. 
{\em Arch. Rational Mech. Anal.} 222 (2016), 91--141. 

\bibitem {KM1} \name[Kristensen, J.] \& \name[Mingione, G.]:
The singular set of minima of integral functionals. {\em
Arch.~Rat.~Mech.~Anal.} 180 (2006), 331--398.

\bibitem {KM2} \name[Kristensen, J.] \& \name[Mingione, G.]:
The singular set of Lipschitzian minima of multiple integrals. {\em
Arch.~Rat.~Mech.~Anal.} 184 (2007), 341-369.

\bibitem {KMbull} \name[Kuusi, T.] \& \name[Mingione, G.]: Guide to nonlinear potential estimates. {\em Bull. Math. Sci.} 4 (2014), 1-82. 

\bibitem {KMjems} \name[Kuusi, T.] \& \name[Mingione, G.]: Vectorial nonlinear potential theory. {\em J.~Europ.~Math.~Soc. (JEMS)} 20 (2018), 929--1004. 

\bibitem{leosie}\name[Leonetti, F.] \& \name[Siepe, F.]: Maximum Principle for Vector Valued Minimizers. {\em J. Convex Anal.} 12 (2005), 267--278.



\bibitem{luck} \name[Luckhaus, S.]: Partial H\"{o}lder continuity for minima of certain energies
among maps into Riemannian manifold. {\em Indiana Univ. Math. J.}
37 (1988), 349--367.


%
\bibitem {Manth} \name[Manfredi, J.J.]: Regularity of the gradient for a class of nonlinear possibly
degenerate elliptic equations. {\em Ph.D. Thesis}. University of
Washington, St. Louis, 1986.
%
\bibitem{M1} \name[Marcellini, P.]: Regularity of
minimizers of integrals of the calculus of variations with non
standard growth conditions. {\em Arch.~Rat.~Mech.~Anal.}  105
(1989), 267--284.
%
\bibitem{M2} \name[Marcellini, P.]: Regularity and
existence of solutions of elliptic equations with $p,q$-growth
conditions. {\em J.~Differential Equations}
  90  (1991), 1--30.
  
  

\bibitem {MMS2} \name[Mazowiecka, M.] \& \name[Mi\'skiewicz, M.] \& \name[Schikorra, A.]: On the size of the singular set of minimizing harmonic maps into the 2-sphere in dimension four and higher. {\em Arxiv preprint, 2019}



\bibitem{mazya} \name[Mazya, V.]: {\em Sobolev spaces with applications to elliptic partial differential equations. Second, revised and augmented edition.}  Grundlehren der Mathematischen Wissenschaften, 342. Springer, Heidelberg, 2011. xxviii+866 pp


\bibitem{Dark} \name[Mingione, G.]: Regularity of minima: an invitation to the Dark Side of the Calculus of Variations. {\em Appl. Math.} 51 (2006), 355-425.

\bibitem{niemi} \name[Nieminen, E.]: Hausdorff measures, capacities, and Sobolev spaces with weights. {\em Ann. Acad. Sci. Fenn. Ser. A I Math. Dissertationes} No. 81 (1991), 39 pages.

    \bibitem{OK1} \name[Ok, J.]: Gradient estimates for elliptic equations with $L^{p(\cdot)}\log L$ growth. {\em Calc.~Var.~\& PDE} 55 (2016), no. 2, Paper No. 26, 30 pages. 


    \bibitem{OK2} \name[Ok, J.]: Regularity of $\omega$-minimizers for a class of functionals with non-standard growth. {\em Calc.~Var.~\& PDE} 56 (2017) Paper No. 48, 31 pages. 
    
    
     \bibitem{PRR} \name[Papageorgiou, N.] \& \name[Radulescu, V.] \& \name[Repovš, D.]: Double-phase problems with reaction of arbitrary growth. {\em Z. Angew. Math. Phys.} 69 (2018),  Art. 108, 21 pp.

\bibitem{Radu2}
Radulescu V. and Repovs D.: {\em Partial differential equations with variable exponents: variational methods and qualitative nalysis}. Chapman \& Hall/CRC Monographs and Research Notes in Mathematics. 2015

    
\bibitem{RT1} \name[Ragusa, M.A.] \& \name[Tachikawa, A.]: Boundary regularity of minimizers of $p(x)$-energy functionals. {\em Ann.~Inst.~H.~Poincar\'e Anal.~Non Lin\`eaire} 33  (2017), 451--476.  
    
\bibitem{RT2} \name[Ragusa, M.A.] \& \name[Tachikawa, A.]: Partial regularity of $p(x)$-harmonic maps. {\em Trans. Amer. Math. Soc.} 365 (2013), 3329--3353. 

\bibitem{SU} \name[Schoen, R.] \& \name[Uhlenbeck, K.]:
A regularity theory for harmonic maps. {\em J.~Diff. Geom.}
17 (1983), 307-336.

\bibitem{SU2} \name[Schoen, R.] \& \name[Uhlenbeck, K.]: Boundary regularity and the Dirichlet problem for
harmonic maps. {\em J. Diff. Geom.}  18 (1983) , 253-268.

\bibitem{Simon} \name[Simon, L.]: {\em
Lectures on regularity and singularities of harmonic maps}.
Birkhäuser-Verlag, Basel-Boston-Berlin 1996.
    
    
 \bibitem{turesson} \name[Turesson, B.O.]: {\em Nonlinear potential theory and weighted Sobolev spaces}. Lecture Notes in Math., 1736. Springer-Verlag, Berlin, 2000. xiv+173. 
  


\bibitem{Uh} \name[Uhlenbeck, K.]: Regularity for a class of non-linear
elliptic systems. {\em Acta Math.} 138 (1977), 219--240.
%

\bibitem{Ur} \name[Ural'tseva, N.N.]: Degenerate quasilinear elliptic systems. {\em Zap.~Na.
Sem.~Leningrad.~Otdel.~Mat.~Inst.~Steklov.~(LOMI)} 7 (1968),
184--222.
%
\bibitem{UU} \name[Ural'tseva, N.N.] \& \name[Urdaletova, A.B.]: The boundedness of
the gradients of generalized solutions of degenerate quasilinear
non-uniformly elliptic equations. {\em Vestnik Leningrad Univ. Math.
19} (1983) (russian) english. tran.: 16 (1984), 263-270.
%
%

\bibitem{Z1} \name[Zhikov, V.V.]: Averaging of functionals of the calculus of variations and
elasticity theory. {\em Izv. Akad. Nauk SSSR Ser. Mat.} 50 (1986),
 675--710.
% 
%
\bibitem{Z2} \name[Zhikov, V.V.]: On Lavrentiev's Phenomenon. {\em Russian J. Math. Phys.} 3 (1995), 249--269.
%
\bibitem{Z3} \name[Zhikov, V.V.]: On some variational problems. {\em Russian J.
Math. Phys.} 5 (1997),  105--116.
%
\end{thebibliography}
\end{document}